\documentclass[10pt]{article}

\usepackage[colorlinks = true, pdfstartview = FitV, linkcolor = blue, citecolor = blue, urlcolor = blue, hypertexnames=false]{hyperref}
\usepackage{booktabs}
\usepackage{graphicx}
\usepackage{amsmath}
\usepackage{amssymb}
\usepackage{mathtools}
\usepackage{latexsym}
\usepackage{subfigure}
\usepackage{crop}
\usepackage{algorithmic,algorithm}
\usepackage{caption}
\usepackage{multirow}
\usepackage{tabularx}
\usepackage{bm}
\usepackage{bbm}
\usepackage{enumerate}
\usepackage{csquotes}
\usepackage{url}
\usepackage{array}
\usepackage{paralist}
\usepackage{diagbox}
\usepackage{booktabs}
\usepackage[dvipsnames]{xcolor}

\usepackage{rotating}
\usepackage{tikz}
\usetikzlibrary{shapes, arrows, positioning}
\usepackage[capitalise]{cleveref}
\crefname{equation}{}{}
\crefname{figure}{Figure}{Figures}
\creflabelformat{equation}{\textup{(#2#1#3)}}
\crefname{assumption}{Assumption}{Assumptions}
\crefname{condition}{Condition}{Conditions}
\crefname{fact}{Fact}{Facts}
\usepackage{xspace}

\renewcommand\th{\textsuperscript{th}\xspace}
\usepackage{fullpage}
\usepackage{multirow}

\usepackage[sort,nocompress]{cite}
\usepackage{arydshln}
\usepackage{enumitem}
\usepackage{pifont}% http://ctan.org/pkg/pifont
\usepackage{accents}

\usepackage{stackengine}
\stackMath
\newcommand\tsup[2][2]{%
	\def\useanchorwidth{T}%
	\ifnum#1>1%
	\stackon[-.5pt]{\tsup[\numexpr#1-1\relax]{#2}}{\scriptscriptstyle\sim}%
	\else%
	\stackon[.5pt]{#2}{\scriptscriptstyle\sim}%
	\fi%
}
\definecolor{forestgreen}{rgb}{0.13, 0.55, 0.13}

\definecolor{amber}{rgb}{1.0, 0.75, 0.0}

\definecolor{bananayellow}{rgb}{.8, 0.6, 0}

\definecolor{uqp}{RGB}{152,24,147} 
\definecolor{oxb}{RGB}{0,33,71}

\newcounter{comment}

\usepackage{amsthm}
\usepackage[framemethod=TikZ]{mdframed}

\newmdtheoremenv[%
linewidth = 1pt,%
roundcorner = 10pt,%
leftmargin = 0,%
rightmargin = 0,%
backgroundcolor = green!3,%
outerlinecolor = blue!70!black,%
splittopskip = \topskip,%
ntheorem = true,%
]{theorem}{Theorem}

\newmdtheoremenv[%
linewidth = 1pt,%
roundcorner = 10pt,%
leftmargin = 0,%
rightmargin = 0,%
backgroundcolor = green!3,%
outerlinecolor = blue!70!black,%
splittopskip = \topskip,%
ntheorem = true,%
]{corollary}{Corollary}

\newmdtheoremenv[%
linewidth = 1pt,%
roundcorner = 10pt,%
leftmargin = 0,%
rightmargin = 0,%
backgroundcolor = green!3,%
outerlinecolor = blue!70!black,%
splittopskip = \topskip,%
ntheorem = true,%
]{lemma}{Lemma}

\newmdtheoremenv[%
linewidth = 1pt,%
roundcorner = 10pt,%
leftmargin = 0,%
rightmargin = 0,%
backgroundcolor = green!3,%
outerlinecolor = blue!70!black,%
splittopskip = \topskip,%
ntheorem = true,%
]{fact}{Fact}

\newmdtheoremenv[%
linewidth = 1pt,%
roundcorner = 10pt,%
leftmargin = 0,%
rightmargin = 0,%
backgroundcolor = green!3,%
outerlinecolor = blue!70!black,%
splittopskip = \topskip,%
ntheorem = true,%
]{proposition}{Proposition}

\newmdtheoremenv[%
linewidth = 1pt,%
roundcorner = 10pt,%
leftmargin = 0,%
rightmargin = 0,%
backgroundcolor = blue!3,%
outerlinecolor = blue!70!black,%
splittopskip = \topskip,%
ntheorem = true,%
]{definition}{Definition}

\newmdtheoremenv[%
linewidth = 1pt,%
roundcorner = 10pt,%
leftmargin = 0,%
rightmargin = 0,%
backgroundcolor = yellow!3,%
outerlinecolor = blue!70!black,%
splittopskip = \topskip,%
ntheorem = true,%
]{condition}{Condition}

\newmdtheoremenv[%
linewidth = 1pt,%
roundcorner = 10pt,%
leftmargin = 0,%
rightmargin = 0,%
backgroundcolor = yellow!3,%
outerlinecolor = blue!70!black,%
splittopskip = \topskip,%
ntheorem = true,%
]{assumption}{Assumption}

\theoremstyle{definition}
\newmdtheoremenv[%
linewidth = 1pt,%
roundcorner = 10pt,%
leftmargin = 0,%
rightmargin = 0,%
backgroundcolor = cyan!3,%
outerlinecolor = blue!70!black,%
splittopskip = \topskip,%
ntheorem = true,%
]{example}{Example}

\theoremstyle{definition}
\newmdtheoremenv[%
linewidth = 1pt,%
roundcorner = 10pt,%
leftmargin = 0,%
rightmargin = 0,%
backgroundcolor = red!3,%
outerlinecolor = blue!70!black,%
splittopskip = \topskip,%
ntheorem = true,%
]{remark}{Remark}

\mdtheorem[%
linewidth = 1pt,%
roundcorner = 10pt,%
leftmargin = 0,%
rightmargin = 0,%
backgroundcolor = green!3,%
outerlinecolor = blue!70!black,%
splittopskip = \topskip,%
ntheorem = true,%
]{informal}{Contributions and Main Results (Informal)}

\mdtheorem[%
linewidth = 1pt,%
roundcorner = 10pt,%
leftmargin = 0,%
rightmargin = 0,%
backgroundcolor = yellow!3,%
outerlinecolor = blue!70!black,%
splittopskip = \topskip,%
ntheorem = true,%
]{NPC}{NPC Condition}

\mdtheorem[%
linewidth = 1pt,%
roundcorner = 10pt,%
leftmargin = 0,%
rightmargin = 0,%
backgroundcolor = yellow!3,%
outerlinecolor = blue!70!black,%
splittopskip = \topskip,%
ntheorem = true,%
]{SOL}{Inexactness Condition}

\usepackage{tikz}
\usepackage{xparse}% So that we can have two optional parameters

\NewDocumentCommand\DownArrow{O{2.0ex} O{black}}{%
	\mathrel{\tikz[baseline] \draw [<-, line width=0.5pt, #2] (0,0) -- ++(0,#1);}
}

\usepackage{listings} % to inser code

\definecolor{mygreen}{rgb}{0,0.6,0}
\definecolor{mygray}{rgb}{0.5,0.5,0.5}
\definecolor{mymauve}{rgb}{0.58,0,0.82}

\usepackage{dsfont}

\newcommand{\defeq}{\triangleq}

\definecolor{purplep}{RGB}{152,24,147}
\newcommand{\cpp}[1]{{\color{purplep}{#1}}}

\newcommand{\real}{\mathbb{R}}
\newcommand{\comp}{\mathbb{C}}
\newcommand{\herm}{\mathbb{H}}
\newcommand{\reals}{\mathbb{R}}
\newcommand{\symm}{\mathbb{S}}

\newcommand{\T}{\intercal}
\renewcommand{\H}{*}
\newcommand{\N}{\natural}
\DeclareMathOperator*{\argmin}{arg\,min}

\newcommand {\dotprod}[1]{\langle #1 \rangle}

\newcommand {\vnorm}[1]{\| #1 \|}

\newcommand {\abs}[1]{| #1 |}

\newcommand {\bigO}[1]{\mathcal O(#1)}

\newcommand {\range}  {\textnormal{Range}}
\newcommand {\Null}  {\textnormal{Null}}

\newcommand {\Span}  {\textnormal{Span}}

\newcommand {\vect}  {\textnormal{vec}}

\newcommand {\zero}   {\mathbf{0}}

\newcommand{\sA}{\mathcal{A}}

\renewcommand {\AA}  {\mathbf{A}}
\newcommand {\AAA}  {\mathbf{A}}
\newcommand {\AAT}  {\AA^{\T}}
\newcommand {\bAA}  {\bar{\AAA}}

\newcommand {\tAA}  {\tilde{\AAA}}
\newcommand {\tAAd}  {{\tilde{\AAA}}^{\dagger}}

\newcommand {\AAd}  {\AAA^{\dagger}}
\newcommand {\AAH}  {\AAA^{\H}}
\newcommand {\BB}  {\mathbf{B}}

\newcommand {\bb}   {\mathbf{b}}
\newcommand {\bbb}  {\bar{\bb}}

\newcommand {\tbb}  {\tilde{\bb}}

\newcommand {\CC}   {\mathbf{C}}

\newcommand {\hCC}  {\hat{\CC}}

\newcommand {\bc}  {\bar{c}}

\newcommand {\DD}  {\mathbf{D}}

\newcommand {\dd}   {\mathbf{d}}

\newcommand {\ddt}  {{\dd}_{t}}

\newcommand {\tdd}  {\tilde{\dd}}
\newcommand {\tddt}  {\tilde{\dd}_{t}}

\newcommand {\cdd}  {\check{\dd}}
\newcommand {\cddt}  {\check{\dd}_{t}}

\newcommand {\dn}   {d}

\newcommand {\EE}  {\mathbf{E}}

\newcommand {\ee}   {\mathbf{e}}

\newcommand {\tg}  {\tilde{g}}

\newcommand {\eye}  {\mathbf{I}}

\newcommand {\KK}  {\mathbf{K}}

\newcommand {\KKH}  {\KK^{\H}}

\newcommand {\Kt}[1]{\mathcal{K}_{t}(#1)}
\newcommand {\Ktn}[1]{\mathcal{K}_{t+1}(#1)}
\newcommand {\Kg}[1]{\mathcal{K}_{g}(#1)}

\newcommand {\MM}  {\mathbf{M}}

\newcommand {\MMs}  {\MM^{1/2}}

\newcommand {\MMsin}  {\MM^{-1/2}}
\newcommand {\MMd}  {\MM^{\dagger}}

\newcommand {\bMM}  {\bar{\MM}}

\newcommand {\PP}  {\mathbf{P}}
\newcommand {\PPT}  {\PP^{\T}}
\newcommand {\PPH}  {\PP^{\H}}

\newcommand {\bPP}  {\bar{\PP}}

\newcommand {\pp}   {\mathbf{p}}

\newcommand {\QQ}  {\mathbf{Q}}
\newcommand {\QQt}  {\QQ_{t}}
\newcommand {\QQtT}  {\QQ_{t}^{\T}}

\newcommand {\qq}   {{\mathbf{q}}}

\newcommand {\qqt}  {\qq_{t}}

\newcommand {\RR}  {\mathbf{R}}

\newcommand {\RRt}  {\RR_{t}}
\newcommand {\hRR}  {\hat{\RR}}
\newcommand {\hRRt}  {\hat{\RR}_{t}}

\newcommand {\rr}   {\mathbf{r}}

\newcommand {\rrt}  {{\rr}_{t}}
\newcommand {\rrtH}  {{\rr}_{t}^{\H}}

\newcommand {\rrtp}  {{\rr}_{t-1}}

\newcommand {\trr}  {\tilde{\rr}}
\newcommand {\trrt}  {\tilde{\rr}_{t}}

\newcommand {\rrg}  {\rr_{g}}
\newcommand {\rrd}  {\rr^{+}}
\newcommand {\rrtg}  {\rr_{\tg}}

\newcommand {\trrtg}  {\tilde{\rr}_{\tg}}

\newcommand {\brrt}  {\bar{\rr}_{t}}

\newcommand {\brrg}  {\bar{\rr}_{g}}

\newcommand {\hrr}  {\hat{\rr}}

\newcommand {\hrrt}  {\hrr_{t}}
\newcommand {\hrrtp}  {\hrr_{t-1}}

\newcommand {\hrrg}  {\hrr_{g}}
\newcommand {\hrrtg}  {\hrr_{\tg}}

\newcommand {\crr}  {\check{\rr}}

\newcommand {\crrt}  {\crr_{t}}
\newcommand {\crrtp}  {\crr_{t-1}}

\newcommand {\crrtg}  {\crr_{\tg}}

\newcommand {\cjl}[1]{\textnormal{Conj}\left(#1\right)}
\newcommand {\cj}[1]{\textnormal{Conj}(#1)}

\newcommand {\rg}[1]{\range(#1)}

\renewcommand {\SS}  {\mathbf{S}}
\newcommand {\SSd}  {\SS^{\dagger}}

\newcommand {\bSS}  {\bar{\SS}}

\newcommand {\SST}  {\SS^{\T}}

\newcommand {\SSdH}  {[\SSd]^{\H}}

\newcommand {\SSH}  {\SS^{\H}}

\newcommand {\St}[1]{\mathcal{S}_{t}(#1)}
\newcommand {\TT}  {\mathbf{T}}
\newcommand {\TTt}  {\mathbf{T}_{t}}

\newcommand {\tTTt}  {\tilde{\mathbf{T}}_{t}}

\newcommand {\hTT}  {\hat{\TT}}
\newcommand {\hTTt}  {\hat{\TT}_{t}}

\newcommand {\ttt}  {\mathbf{t}}
\newcommand {\UU}  {\mathbf{U}}

\newcommand {\UUT}  {\UU^{\intercal}}

\newcommand {\UUp}  {\UU_{\perp}}

\newcommand {\bUU}  {\bar{\UU}}

\newcommand {\bUUp}  {\bar{\UU}_{\perp}}

\newcommand {\uu}   {\mathbf{u}}

\newcommand {\VV}  {\mathbf{V}}

\newcommand {\VVt}  {\VV_{t}}
\newcommand {\VVtH}  {\VV_{t}^{\H}}

\newcommand {\VVtn}  {\VV_{t+1}}

\newcommand {\bVV}  {\bar{\VV}}
\newcommand {\bVVt}  {\bar{\VV}_{t}}

\newcommand {\vv}   {\mathbf{v}}

\newcommand {\vvt}  {{\vv}_{t}}
\newcommand {\vvtn}  {{\vv}_{t+1}}

\newcommand {\bvv}  {\bar{\vv}}

\newcommand {\tvv}  {\tilde{\vv}}
\newcommand {\tvvt}  {\tilde{\vv}_{t}}
\newcommand {\tvvtn}  {\tilde{\vv}_{t+1}}

\newcommand {\ww}   {\mathbf{w}}

\newcommand {\wwt}  {{\ww}_{t}}
\newcommand {\wwtn}  {\ww_{t+1}}

\newcommand {\bww}  {\bar{\ww}}
\newcommand {\bwwt}  {\bar{\ww}_{t}}
\newcommand {\bwwtn}  {\bar{\ww}_{t+1}}

\newcommand {\XX}   {\mathbf{X}}
\newcommand {\tXX}  {\tilde{\XX}}

\newcommand {\xx}   {\mathbf{x}}
\newcommand {\xxt}  {{\xx}_{t}}

\newcommand {\xxtn}  {\xx_{t+1}}
\newcommand {\txx}  {\tilde{\xx}}
\newcommand {\txxt}  {\tilde{\xx}_{t}}

\newcommand {\txxtn}  {\tilde{\xx}_{t+1}}

\newcommand {\xxo}  {{\xx}_{0}}

\newcommand {\xxd}  {\xx^{+}}

\newcommand {\hxx}  {\hat{\xx}}

\newcommand {\hxxd}  {\hxx^{+}}

\newcommand {\cxx}  {\check{\xx}}
\newcommand {\cxxt}  {\check{\xx}_{t}}
\newcommand {\bxx}  {\bar{\xx}}

\newcommand {\hxxtN}  {\hat{\xx}_{t}^{\N}}

\newcommand {\xxg}  {\xx_{g}}

\newcommand {\xxtg}  {\xx_{\tg}}

\newcommand {\txxtg}  {\tilde{\xx}_{\tg}}

\newcommand {\txxd}  {\tilde{\xx}^{+}}

\newcommand {\yy}   {\mathbf{y}}

\newcommand {\ZZ}  {\mathbf{Z}}

\newcommand {\zz}  {\mathbf{z}}
\newcommand {\zzt}  {\zz_{t}}

\newcommand {\bzz}  {\bar{\zz}}
\newcommand {\bzzt}  {\bar{\zz}_{t}}
\newcommand {\bzztn}  {\bar{\zz}_{t+1}}
\newcommand {\zztn}  {\zz_{t+1}}

\newcommand {\alphat}  {\alpha_{t}}

\newcommand {\betat}  {\beta_{t}}
\newcommand {\betatn}  {\beta_{t+1}}

\newcommand {\SIGMA}  {\mathbf{\Sigma}}

\newcommand {\taut}  {{\tau}_{t}}

\newcommand {\sds}[2]{{\mathcal{S}_{#1}} (#2)}
\newcommand {\kryl}[2]{{\mathcal{K}_{#1}} (#2)}
\newcommand {\krylovt}[1]{ {\mathcal{K}_{t}} (\ensuremath{#1}) }

\begin{document}
	
	\title{Obtaining Pseudo-inverse Solutions With MINRES}
	
	\author{Yang Liu\thanks{Mathematical Institute, University of Oxford. Email: \tt{yang.liu@maths.ox.ac.uk}}
		\and
		Andre Milzarek\thanks{
			School of Data Science, The Chinese University of Hong Kong, Shenzhen, (CUHK-Shenzhen), China. Email: \tt{andremilzarek@cuhk.edu.cn}}
		\and
		Fred Roosta\thanks{School of Mathematics and Physics, University of Queensland, Australia, and International Computer Science Institute, Berkeley, USA. Email: \tt{fred.roosta@uq.edu.au}}
	}
	
	\date{\today}
	\maketitle
	
	\abstract{The celebrated minimum residual method (MINRES), proposed in the seminal paper of Paige and Saunders \cite{paige1975solution}, has seen great success and widespread use in solving Hermitian (and complex-symmetric) systems $ \AA \xx = \bb $. Unless the system is consistent, MINRES is not guaranteed to obtain the pseudo-inverse solution. We propose a novel and remarkably simple minimum-norm refinement (MN refinement) that seamlessly integrates with the final MINRES iteration, enabling us to obtain the minimum-norm solution with negligible additional computational cost. We extend our MN refinement to complex-symmetric systems, building on S.-C. Choi's extension of MINRES for solving these systems. Given the flexibility of MINRES to accommodate singular preconditioners, we further investigate the MN refinement in preconditioned settings that involve singular preconditioners.
		We also provide numerical experiments to support our analysis and showcase the effects of our MN refinement.
	}

	\section{Introduction}
\label{sec:intro}

We consider the linear least-squares problem,
\begin{align}
	\label{eq:least_squares}
	\min_{\xx \in \comp^{\dn}} \vnorm{\bb - \AA \xx}^2,
\end{align}
involving a complex-valued matrix $ \AA \in \comp^{\dn \times \dn} $ that is either Hermitian or complex-symmetric, and a right-hand side complex vector $ \bb \in \comp^{\dn} $. Our primary focus is on cases where the system is inconsistent, i.e., $ \bb \notin \rg{\AA} $, which arise in many situations such as optimization of non-convex problems \cite{roosta2018newton,liu2022newtonmr}, numerical solution of partial differential equations \cite{kaasschieter1988preconditioned}, and low-rank matrix computations \cite{gallivan1996high}. Even if the system is theoretically consistent, due to errors from measurements, discretizations, truncations, or round-off, $\AA$ can become numerically singular, resulting in a numerically incompatible system \cite{choi2013minimal}.
In these situations, problem \cref{eq:least_squares} admits infinitely many solutions. Among all of these solutions, the one with smallest Euclidean norm is commonly referred to as the pseudo-inverse solution, also known as the minimum-norm solution, and is given by  $ \AAd \bb $, where $ \AAd $ is the Moore-Penrose pseudo-inverse of $ \AA $. This particular solution, among all other alternatives, holds a special place and offers numerous practical and theoretical advantages in various applications including, e.g., linear matrix equations \cite{engl1981new}, optimization \cite{roosta2018newton,liu2021convergence}, and statistical linear regression \cite{derezinski2020exact,dar2022double}.

When $ d $ is small, one can obtain $ \AAd \bb $ by explicitly computing $ \AAd $ via direct methods, e.g., \cite{courrieu2005fast,katsikis2011improved,bjorck2015numerical,bjorck1996numerical,klema1980singular}. Among them, the best known direct method to obtain $ \AAd $ is via the singular value decomposition \cite{klema1980singular}. However, rather than explicitly computing $ \AAd $ for large-scale problems, say when $d \gg 1000$, it is more appropriate, if not indispensable, to obtain the pseudo-inverse solution $ \AAd \bb $ by an iterative scheme.
% , e.g., \cite{lee2016regularized,stewart1977perturbation,peters1970least,bjorck2015numerical,bjorck1996numerical} \fred{Are these references for obtaining $ \AAd \bb $ directly? Shouldn't we just talk about LSQR, LSMR, or MINRES-QLP?}. \yang{I didn't mention they are iterative methods, in fact, all of them are direct methods. [However, a more common criteria is to obtaining pseudo-inverse solution $ \AAd \bb $ of a least-squared problem as in \cref{eq:least_squares}, e.g., \cite{lee2016regularized,stewart1977perturbation,peters1970least,bjorck2015numerical,bjorck1996numerical}.]}
%
LSQR \cite{paige1982lsqr} and LSMR \cite{fong2011lsmr} are iterative Krylov subspace methods that can recover the pseudo-inverse solution $ \AAd \bb $ for any rectangular $ \AA \in \comp^{m \times \dn} $ (in contrast to MINRES). However, they are mathematically equivalent to the conjugate gradient method (CG) \cite{hestenes1952methods} and the minimum residual method (MINRES) \cite{paige1975solution}, respectively, applied to the normal equation of \cref{eq:least_squares}. Consequently, each iteration of LSQR and LSMR necessitates two matrix-vector product evaluations, which can be computationally challenging when $d \gg 1$.\footnote{A new solver, M\textsc{in}A\textsc{res}, which minimizes $\| \AA \rr \|$ in each Krylov subspace and requires only one matrix-vector product per iteration, has recently been proposed~\cite{montoison2025minares}.} While CG iterations are not well-defined unless $ \AA \succeq \zero $ and $ \bb \in \rg{\AA} $, MINRES can effectively solve \cref{eq:least_squares} even for systems that are indefinite or singular. Nonetheless, MINRES is not guaranteed to obtain $ \AAd \bb $ unless $ \bb \in \rg{\AA} $ \cite[Theorem 3.1]{choi2011minres}. To address this limitation, a variant of MINRES, called MINRES-QLP was introduced in \cite{choi2011minres,choi2014algorithm}. MINRES-QLP ensures convergence to the pseudo-inverse solution in all cases. Compared with MINRES, the rank-revealing QLP decomposition of the tridiagonal matrix from the Lanczos process in MINRES-QLP necessitates additional computations per iteration \cite{choi2011minres}. In particular, the number of floating-point operations is up to $14d$ per iteration for MINRES-QLP compared to $9d$ for MINRES. Moreover, owing to its increased complexity, the implementation of MINRES-QLP can pose greater challenges than that of MINRES.

\textbf{Contributions.} Our contributions focus on introducing a novel and remarkably simple strategy, named \emph{MN refinement}, which modifies the final MINRES iteration to obtain the minimum-norm solution with negligible additional computational overhead and minimal changes to the original MINRES algorithm.
\begin{itemize}
    \item After addressing the common case of Hermitian settings (\cref{thm:MINRES_dagger}), we build on the MINRES solver's extension to complex-symmetric systems described in \cite{choi2013minimal} and propose a corresponding MN refinement for complex-symmetric systems (\cref{thm:CSMINRES_dagger}).
    \item %Given MINRES's ability to accommodate semi-definite preconditioners, 
    We also investigate generalizations of the MN refinement to preconditioned settings with singular preconditioners (\cref{thm:MN-refinement,thm:MN-refinement_CS}).
	\item Numerical experiments show that the developed MN refinements effectively recover pseudo-inverse solutions for Hermitian and complex-symmetric systems. We further demonstrate applications to solving PDEs, like the Maxwell problem, and to large-scale, ill-conditioned image deblurring problems using semi-definite preconditioners.
\end{itemize}

The rest of the paper is organized as follows. We end this section by introducing our notation and definitions. 
The theoretical analyses underpinning our MN refinement in a variety of settings are given in \cref{sec:pinverse,sec:pMINRES}. In particular, \cref{sec:hermitian,sec:CSMINRES} consider Hermitian and complex-symmetric systems, respectively, while \cref{sec:pMINRES_H,sec:pMINRES_CS} study our MN refinement in the presence of positive semi-definite preconditioners in each matrix class. Numerical results are given in \cref{sec:exp}. Conclusions and further thoughts are gathered in \cref{sec:conclusion}.

\subsection*{Notation and definitions}
%\label{sec:notation}
Throughout this paper, vectors and matrices are denoted by bold lower- and upper-case letters, respectively. Regular lower-case letters are reserved for scalars. 
The space of complex-symmetric and Hermitian $d \times d$ matrices are denoted by $ \symm^{d \times d} $  and $ \herm^{d \times d} $, respectively. 
The inner-product of complex vectors $ \xx$ and $\yy $ is defined as $ \dotprod{\xx, \yy} = \xx^{\H} \yy $, where $ \xx^\H $ represents the conjugate transpose of $ \xx $. The Euclidean norm of a vector $ \xx \in \comp^{\dn} $ is given by $ \|\xx\| $. The conjugate of a vector or a matrix is denoted by $ \bbb = \cj{\bb} $ or $ \bAA = \cj{\AA} $.
The zero vector and the zero matrix are denoted by $ \zero $ while the identity matrix of dimension $ t \times t $ is given by $ \eye_{t} $.
We use $ \ee_{j} $ to denote the $ j\th $ column of the identity matrix. We use $ \MM \succeq\,(\succ)\,\zero$, to indicate that $ \MM $ is positive semi-definite (PSD) (positive definite (PD)). For a PSD matrix $ \MM$, we write $ \vnorm{\xx}_{\MM} = \sqrt{\dotprod{\xx,\MM \xx}} $ (this is an abuse of notation since unless $ \MM \succ \zero $, this does not imply a norm).
%and its transpose is defined as $ \AA^{\ddagger} \defeq [\AAd]^{\T}$.
%The element of a matrix $ \AA $ located at the $ i\th $ row and the $ j\th $ column is denoted by $ [\AA]_{ij} $. 
For any $ t\geq 1 $, the set $\mathcal{K}_{t}(\AA, \bb) = \Span\left\{\bb,\AA\bb,\ldots,\AA^{t-1}\bb\right\}$ denotes the Krylov subspace of degree $ t $ generated using $ \AA $ and $ \bb $.
The residual vector at $ t\th $ iteration $ \xxt $ is denoted by $ \rr_{t} = \bb - \AA \xx_{t} $.

%The Moore-Penrose generalized inverse of matrix $ \AA \in \comp^{\dn \times m} $ is denoted by $ \AA^{\dagger} $. 
Although we focus on square matrices $\AA$ in this paper, let us recall the standard definition of the Moore-Penrose pseudo-inverse for general $ \AA $.
\begin{definition}
	\label{def:pinverse}
	For given $ \AA \in \comp^{\dn \times m} $, the (Moore-Penrose) pseudo-inverse is the matrix $ \BB \in \comp^{m \times \dn} $ satisfying the Moore-Penrose conditions, namely  
	\begin{subequations}     
		\label{eq:pi}   
		\begin{align}
			\AA \BB \AA &= \AA, \label{eq:pi_1} \\
			\BB \AA \BB &= \BB, \label{eq:pi_2} \\
			(\AA \BB)^{\H} &= \AA \BB, \label{eq:pi_3} \\
			(\BB \AA)^{\H} &= \BB \AA. \label{eq:pi_4}
		\end{align}
	\end{subequations}
	The matrix $\BB$ is unique and is denoted by $ \AAd $. The point $ \xxd \defeq \AAd \bb $ is said to be the (Moore-Penrose) pseudo-inverse solution of \cref{eq:least_squares}.
\end{definition}

The interplay between $\AA$ and $\bb$ is a critical factor in the convergence of MINRES and many other iterative solvers. This interplay is encapsulated in the notion of the grade of $ \bb $ with respect to $ \AA $. 

\begin{definition}
	\label{def:grade_b}
	The grade of $ \bb \in \comp^{\dn} $ with respect to $ \AA \in \comp^{\dn \times \dn} $ is the positive integer $ g(\AA, \bb) $ such that $\text{dim}(\mathcal{K}_{t}(\AA, \bb)) = \min\{t,g(\AA,\bb)\} $.
	In the special case where $\AA$ is complex-symmetric, i.e., $ \AAT = \AA $, we define the grade $g(\AA,\bb)$ with respect to a modified Krylov subspace, the so-called Saunders subspace, cf. \cite{choi2013minimal,saunders1988two} and \eqref{eq:saunders}, i.e., $\text{dim}(\mathcal{S}_{t}(\AA, \bb)) = \min\{t,g(\AA,\bb)\} $. 
\end{definition}

Recall that $ g(\AA, \bb) $ determines the iteration at which a Krylov subspace algorithm terminates with a solution (in exact arithmetic). 
For simplicity, for given pairs $(\AA, \bb)$ and $(\tAA, \tbb)$, we set $ g = g(\AA, \bb) $ and $ \tg = g(\tAA, \tbb) $. 
In \cref{def:grade_b}, the Saunders subspace is introduced to distinguish between Hermitian and complex-symmetric systems, where the standard Lanczos process \cite{paige1975solution} is replaced by the Saunders process, which operates in the Saunders subspace \cite{choi2013minimal,saunders1988two}; see also \cref{sec:CSMINRES,app:CSMINRES}.

\section{Pseudo-inverse solutions in MINRES}
\label{sec:pinverse}
The most common method to compute the pseudo-inverse of a matrix is via singular value decomposition (SVD). The computational cost of the SVD is $ \bigO{\dn^3} $, which essentially makes such an approach intractable in high-dimensions. In our context, we do not seek to directly compute $\AAd$ as our interest lies only in obtaining the pseudo-inverse solution, i.e., $ \xxd = \AAd \bb $. As alluded to earlier in \cref{sec:intro}, existing iterative methods that can ensure the pseudo-inverse solution of \cref{eq:least_squares} in the incompatible setting demand significantly higher computational effort per iteration compared to plain MINRES.

\begin{algorithm}[t!]
	\caption{The left column presents the MINRES algorithm for the Hermitian case, while the right column outlines the particular modifications needed for complex-symmetric settings. \\[0.5ex] \textbf{Hermitian} \hfill \textbf{\cpp{Complex-symmetric}}}
	\label{alg:MINRES}
	\begin{tikzpicture}[scale=1]
		\node at (0,0) {\begin{minipage}{.98\linewidth}
				\begin{algorithmic}[1]
					\STATE \textbf{Inputs:} $ \AA \in \herm^{\dn \times \dn} $, $ \bb \in \comp^{\dn} $, \hfill\COMMENT{$\cpp{\AAT = \AA} \in \symm^{\dn \times \dn}$}
					\vspace{1mm}
					\STATE $ \phi_0 = \beta_1 = \vnorm{\bb} $, $ \rr_{0} = \bb $, $ \vv_1 = \rr_{0} /\phi_0 $,
					\vspace{1mm}
					\STATE $ \vv_0 = \xx_0 = \dd_0 = \dd_{-1} = \zero $,
					\vspace{1mm}
					\STATE $ s_0 = 0 $, $ c_0 = -1 $, $ \delta_1 = \tau_0 = 0 $, $ t = 1 $, 
					\vspace{1mm}
					\WHILE { \text{True} }
					\vspace{1mm}
					\STATE $ \qq_{t} = \AA \vv_{t} $, $\alpha_{t} = \dotprod{\vvt, \qq_{t}} $, \hfill\COMMENT{$ \qq_{t} = \AA \cpp{\bvv_{t}}$}
					\vspace{1mm}
					\STATE $ \qq_{t} = \qq_{t} - \alpha_{t} \vv_{t} - \beta_{t} \vv_{t-1} $, $ \beta_{t+1} = \| \qq_{t} \| $, 
					\vspace{2mm}
					\STATE %\hspace*{-2mm} 
					$\displaystyle 
					\begin{bmatrix}
						\delta^{[2]}_{t} & \epsilon_{t+1} \\
						\gamma_{t} & \delta_{t+1}
					\end{bmatrix} = \begin{bmatrix}
						c_{t-1} & s_{t-1} \\
						s_{t-1} & - c_{t-1}
					\end{bmatrix} \begin{bmatrix}
						\delta_{t} & 0 \\
						\alpha_{t} & \beta_{t+1}
					\end{bmatrix}, $ \hfill\COMMENT{$\begin{bmatrix}
							\cpp{\bc_{t-1}} & s_{t-1} \\
							s_{t-1} & - c_{t-1}
						\end{bmatrix}$}
					\vspace{1mm}
					\STATE $ \gamma_{t}^{[2]} = \sqrt{\abs{\gamma_{t}}^2 + \beta_{t+1}^2} $,
					\vspace{1mm}
					\IF{ $ \gamma_{t}^{[2]} \neq 0 $ }
					\vspace{1mm}
					\STATE $ c_{t} = \gamma_{t}/\gamma_{t}^{[2]} $, $ s_{t} = \beta_{t+1} / \gamma_{t}^{[2]} $, 
					\vspace{1mm}
					\STATE $ \tau_{t} = c_{t} \phi_{t-1} $, $ \phi_{t} = s_{t} \phi_{t-1} $, \hfill\COMMENT{$ \tau_{t} = \cpp{\bc_{t}} \phi_{t-1} $}
					\vspace{1mm}
					\STATE $ \dd_{t} = (\vv_{t} - \delta^{[2]}_{t} \dd_{t-1} - \epsilon_{t} \dd_{t-2} ) / \gamma^{[2]}_{t} $, \hfill\COMMENT{$ \dd_{t} = (\cpp{\bvv_{t}} - \delta^{[2]}_{t} \dd_{t-1} - \epsilon_{t} \dd_{t-2} ) / \gamma^{[2]}_{t} $}
					\vspace{1mm}
					\STATE $ \xx_{t} = \xx_{t-1} + \tau_{t} \dd_{t} $, $ \| \rr_{t} \| = \phi_{t} $, 
					\vspace{1mm}
					\IF{$ \beta_{t+1} \neq 0 $}
					\STATE $ \vv_{t+1} = \qq_{t} / \beta_{t+1} $,
					\vspace{1mm}
					\STATE $ \rr_{t} = s_{t}^2 \rr_{t-1} - \phi_{t} c_{t} \vv_{t+1} $, \hfill\COMMENT{$ \rr_{t} = s_{t}^2 \rr_{t-1} - \phi_{t} \cpp{\bc_{t}} \vv_{t+1} $}
					\vspace{1mm}
					\ELSE
					\vspace{1mm}
					\STATE $ g = t $, \textbf{return}\;\;$ \hxxd = \xxg $.
					\vspace{1mm}
					\ENDIF
					\vspace{1mm}
					\ELSE
					\vspace{1mm}
					\STATE $ g = t $, $ c_{g} = \tau_{g} = 0 $, $ s_{g} = 1 $, $ \phi_{g} = \phi_{g-1} $, 
					\vspace{1mm}
					\STATE $ \rr_{g} = \rr_{g-1} $, $ \xx_{g} = \xx_{g-1} $, 
					\vspace{1mm}
					\STATE \textbf{return}\;\;$ \hxxd = \xxg - \dotprod{\rrg, \xxg} \rrg / \| \rrg \|^2 $. \hfill\COMMENT{$ \hxxd = \xxg - \dotprod{\cpp{\brrg}, \xxg} \cpp{\brrg} / \| \cpp{\brrg} \|^2 $}
					\vspace{1mm}
					\ENDIF
					\vspace{1mm}
					\STATE $ t \leftarrow t+1 $,
					\vspace{1mm}
					\ENDWHILE
					\vspace{1mm}
					\STATE \textbf{Outputs:} $ \hxxd = \AAd \bb $ w.r.t., \cref{eq:least_squares}.
				\end{algorithmic}
		\end{minipage}};
		\draw[densely dotted] (1,8) -- (1,-8);
	\end{tikzpicture}
\end{algorithm}

The MINRES method is shown in \cref{alg:MINRES}. Recall that for a Hermitian matrix $ \AA \in \herm^{\dn \times \dn} $ and a right-hand side vector $ \bb \in \comp^{\dn} $, the $ t\th $ iteration of MINRES for solving \cref{eq:least_squares} can be written as
\begin{align}
	\label{eq:MINRES}
	\min_{\xx \in \mathcal{K}_t(\AA, \bb)} \vnorm{\bb - \AA \xx}^2,
\end{align}
where $ t \leq g $ and the grade $ g $ is defined as in \cref{def:grade_b}; see \cite{paige1975solution,liu2022minres} for more details. When $ \bb \in \rg{\AA} $, it can be easily shown that MINRES returns the pseudo-inverse solution of \cref{eq:least_squares}, \cite{choi2011minres}. However, when $ \bb \notin \rg{\AA} $, the situation is more complicated. In such settings, as long as $ \rg{\AA} \subseteq \mathcal{K}_{g}(\AA, \bb) $, the final iterate of MINRES satisfies the properties \cref{eq:pi_1,eq:pi_2,eq:pi_3} in \cref{def:pinverse}, otherwise one can only guarantee \cref{eq:pi_2,eq:pi_3}; see, e.g., \cite[Theorem 2.27]{choi2006iterative}.
To obtain a solution that satisfies all four properties of the pseudo-inverse, Choi et al., \cite{choi2011minres}, provide an algorithm, called MINRES-QLP. However, not only is MINRES-QLP rather complicated to implement, but it also demands more computations per iteration compared to MINRES. 

Our goal in this section is to obtain the pseudo-inverse solution of \cref{eq:least_squares} with minimal additional costs and alterations to the MINRES algorithm.
Before doing so, we first restate a simple and well-known result for solutions of the least squares problem \cref{eq:least_squares}; the proof of Lemma \ref{lemma:MINRES_solve_exactly} can be found in, e.g., \cite[Page 7]{bjorck1996numerical}. %tirival by recalling that the general solution of \cref{eq:least_squares} is given by $\xx = \AAd \bb + \left(\eye_{d} - \AAd \AA \right) \yy$ for any $ \yy \in \comp^{d} $.

\begin{lemma}
	\label{lemma:MINRES_solve_exactly}
	The vector $ \xx $ is a solution of \cref{eq:least_squares} if and only if there exists a vector $\yy \in \comp^{d}$ such that $\xx = \AAd \bb + (\eye_{d} - \AAd \AA) \yy$. Furthermore, for any solution $ \xx $, we have $\AA \xx = \AA \AAd \bb$ and $\rr = \bb - \AA \xx = (\eye - \AA \AAd) \bb$.
\end{lemma}
By \cref{lemma:MINRES_solve_exactly}, the general solution of \cref{eq:least_squares} is given by 
\begin{align*}
	\xx = \xxd + \zz, \quad \text{where } \quad \zz \in \Null(\AA).
\end{align*}
In other words, at $ t = g $, the final iterate of MINRES is a vector that is implicitly given as a linear combination of two vectors, namely $ \AAd \bb $ and some $ \zz \in \Null(\AA) $. Our MN refinement technique amounts to eliminating the component $ \zz $ to obtain the pseudo-inverse solution $ \xxd $.

\subsection{Hermitian and skew-Hermitian systems}
\label{sec:hermitian}
We now detail our MN-re- finement to obtain the pseudo-inverse solution of \cref{eq:least_squares} from MINRES iterations. We show that not only does our strategy apply to the typical case where $ \AA $ is Hermitian, but it can also be readily adapted to cases where the underlying matrix is skew-Hermitian \cite{greif2009iterative}.

\begin{theorem}[MN refinement for Hermitian Systems]
	\label{thm:MINRES_dagger}
	Let $ \xxg $ be the final iterate of \cref{alg:MINRES} for solving \cref{eq:least_squares} with the Hermitian matrix $ \AA \in \herm^{\dn \times \dn} $. \begin{enumerate}
		\item If $ \rrg = \bm{0} $, then $ \xxg = \AAd \bb $.
		\item If $ \rrg \neq \bm{0} $, define the MN refinement vector 
		\begin{align}
			\label{eq:MINRES_MN-refinement}
			\hxxd = \xxg - \frac{\dotprod{\rrg, \xxg}}{\vnorm{\rrg}^2} \rrg = \left(\eye_{d} - \frac{\rrg \rrg^{\H}}{\vnorm{\rrg}^{2}}\right) \xxg,
		\end{align}
		where $ \rrg = \bb - \AA \xxg $. Then, we have $ \hxxd = \AAd \bb $.
	\end{enumerate}
\end{theorem}

\begin{proof}
	If $ \rrg = \bm{0} $, since $ \bb \in \rg{\AA} $, we naturally obtain $ \xxg = \xxd $ by \cite[Theorem 3.1]{choi2011minres}. Now, suppose $ \rrg \neq \zero $, which implies $ \bb \not\in \rg{\AA} $. Since $ \xxg \in \mathcal{K}_{g}(\AA, \bb) $, we can write $ \xxg = \sum_{i=1}^{g} a_{i} \AA^{i-1} \bb $ for some $a_i \in \comp$, $i=1,\dots,g$. It then follows
	\begin{align*}
		\xxg = \xxg - a_1(\eye-\AA\AAd)\bb + a_1(\eye-\AA\AAd)\bb = \pp + \qq,
	\end{align*}
	where
	\begin{align*}
		\pp &\defeq \xxg - a_1(\eye-\AA\AAd)\bb = a_1 \AA \AAd \bb + {\sum}_{i=2}^{g} a_{i} \AA^{i-1} \bb,\\
		\qq &\defeq a_1 (\eye - \AA \AAd) \bb = a_1 \rrg.
	\end{align*}
	Since $ \AA $ is Hermitian, we have $ \AA^{2} \AAd = \AA $ \cite[Fact 6.3.17]{bernstein2009matrix}. So $\AA \pp = \AA \xxg$, which implies that the normal equation is satisfied for $ \pp $. In other words, $ \pp $ is a solution to \cref{eq:least_squares}. We also recall from \cref{lemma:MINRES_solve_exactly} that $\xxg = \AAd \bb + \left(\eye_{d} - \AAd \AA \right) \yy $ for some $ \yy \in \comp^{d}$. Now, since $ \pp \in \rg{\AA} = \rg{\AAd} $ and $ \qq \in \Null(\AA) $, we must have that $ \pp = \AAd\bb $. Also, since $ \pp \perp \qq$, we can find $ a_1 $ as 
	\begin{align*}
		a_1  = \frac{\dotprod{\rrg, \xxg}}{\vnorm{\rrg}^{2}},
	\end{align*}
	which gives the desired result.
\end{proof}

\begin{remark}
From the decomposition $\xx = \xxd + \zz$, with $\zz \in \Null(\AA)$, we see that the foundation of applying MN refinement lies in obtaining a handle on the vectors $\zz$. For example, the proof of \cref{thm:MINRES_dagger} reveals that, for any Hermitian system, as long as $\xxg$ is a least-squares solution satisfying $\xxg \in \mathcal{K}_{g}(\AA, \bb)$, we have $\xxg = \xx^+ + \alpha \rrg$ for some scalar $\alpha$, allowing the construction of an MN-refinement. 
In this sense, it appears that MN-refinement could be extended to another least-squares solver whose solution involves a null-space component $\zz$ that can be expressed as a linear combination of \emph{known} vectors. However, focusing on the framework \cref{eq:MINRES} and to maintain clarity of scope, this paper considers only its most well-known implementation, namely the MINRES solver.
\end{remark}

In \cref{prop:MN-refinement_t}, we relate the MN refinement at iteration $ t $ of MINRES to the orthogonal projection of the iterate $ \xxt $ onto $ \AA \Kt{\AA, \bb} $. \cref{prop:MN-refinement_t} implies that, not only the MN refinement of $ \xxg $ recovers the pseudo-inverse solution, but also the MN refinement of $ \xxt $ is the minimum-norm component within the subspace $\AA \Kt{\AA, \bb}$ at every iteration $ t $. Consequently, performing MN refinement at the final iteration amounts to the orthogonal projection of $ \xxg $ onto $ \AA \Kg{\AA, \bb} $, effectively eliminating the contribution of the portion of $ \bb $ that lies in the null space of $ \AA $ from the final solution.
\begin{proposition}
	\label{prop:MN-refinement_t}
	$ \hxxtN $ is the orthogonal projection of $ \xxt $ onto $ \AA \Kt{\AA, \bb} $ in MINRES, where
	\begin{align}
		\label{eq:MN-refinement_t}
		\hxxtN \triangleq \xxt - \frac{\dotprod{\rrt, \xxt}}{\vnorm{\rrt}^2} \rrt = \left(\eye_{d} - \frac{\rrt \rrt^{\H}}{\vnorm{\rrt}^{2}}\right) \xxt.
	\end{align}
\end{proposition}
\begin{proof}
	Note that, for $ t \leq g-1 $, $ \AA \Kt{\AA, \bb} = \Span\{\AA \bb, \AA^2 \bb, \ldots, \AA^t \bb\} $ is a $ t $-dimensional subspace. The oblique projection framework gives $ \rrt \perp \AA \Kt{\AA, \bb} $. Therefore, $ \mathcal{K}_{t+1}(\AA, \bb) = \Span\{\rrt\} \oplus \AA \Kt{\AA, \bb} $ and $\mathcal{K}_{t+1}(\AA, \bb)$ has dimension $t+1$. From $ \xxt \in \Kt{\AA, \bb} \subset \Ktn{\AA, \bb} $, it follows that 
	\begin{align*}
		\xxt = \PP_{\rrt} \xxt + \PP_{\AA \mathcal{K}_{t}} \xxt,    
	\end{align*} 
	where $ \PP_{\rrt}$ and $\PP_{\AA \mathcal{K}_{t}}$ are orthogonal projections onto the span of $ \rrt $ and the subspace $ \AA \Kt{\AA,\bb} $, respectively. Noting that $ \PP_{\rrt} =  \rrt \rrtH/\vnorm{\rrt}^2$ gives the desired result.
\end{proof}

Recall that if $\AA$ is skew-Hermitian, i.e., $ \AAH = -\AA $, then we have $ i\AA \in \herm^{\dn \times \dn} $. Thus, by \cite{choi2013minimal}, we can apply MINRES with inputs $ i\AA $ and $ i\bb $ to solve \cref{eq:MINRES}. We immediately obtain the following result.
\begin{corollary}[MN refinement for Skew-Hermitian Systems]
	\label{coro:MINRES_SH_dagger}
	If $\AA \in \comp^{\dn \times \dn} $ is skew-Hermitian, the result of \cref{thm:MINRES_dagger} continues to hold as long as \cref{alg:MINRES} is applied with inputs $ i\AA $ and $ i\bb $. Note that, here, $ \rrg = i\bb - i\AA \xxg $. 
\end{corollary}
\begin{proof}
	We can easily verify that $ -i\AAd $ is the pseudo inverse of $ i\AA $ by \cref{def:pinverse}. Hence, we obtain the desired result as in the proof of \cref{thm:MINRES_dagger} by noticing $ [i \AA]^{\dagger} (i\bb) = - i\AAd i\bb = \AAd \bb $.
\end{proof}

\subsection{Complex-symmetric systems}
\label{sec:CSMINRES}

We now consider the case where $\AA \in \symm^{d \times d}$ is complex-symmetric, i.e., $ \AAT = \AA $ but $ \AAH \neq \AA $. Least-squares systems involving such matrices arise in many applications such as data fitting \cite{luk2002exponential,ammar1999computation}, viscoelasticity \cite{christensen2012theory,arts1992experimental}, quantum dynamics \cite{bar1997fast,bar1995new}, electromagnetics \cite{arbenz2004jacobi}, and power systems \cite{howle2005iterative}. Although MINRES was initially designed for handling Hermitian systems, it has since been extended to complex-symmetric settings \cite{choi2013minimal}. Here, we extend our MN refinement procedure for obtaining the pseudo-inverse solution in such settings. We also note that the QMR method of Freund \cite{freund1992conjugate,freund1995software} was designed for solving non-Hermitian systems, but unlike MINRES, it requires $\AA$ to be nonsingular.

In order to provide clear explanations for the MN refinement in this case, we first briefly review the construction of the extension of MINRES that is adapted by Choi \cite{choi2013minimal} for complex-symmetric systems. For this, we mostly follow the development in the original work of \cite{choi2013minimal} with a few modifications. The major difference for complex-symmetric systems is that the usual Lanczos process \cite{paige1975solution} is replaced with the Saunders process \cite{choi2013minimal,saunders1988two} with respect to the Saunders subspace \cite{choi2013minimal,saunders1988two}, which is defined as
\begin{align}
	\label{eq:saunders}
	\mathcal{S}_{t}(\AA, \bb) = \kryl{t_1}{\AA \bAA, \bb} \oplus \kryl{t_2}{\AA \bAA, \AA \bbb}, \quad t_1 + t_2 = t, \quad 0 \leq t_1 - t_2 \leq 1,
\end{align}
where $ \oplus $ is the direct sum operator. 
%In this case, the definition of the grade with respect to the Saunders subspaces is defined as in \cref{def:grade_b_CS} instead of \cref{def:grade_b}.
	%	
	%\begin{definition}
	%	\label{def:grade_b_CS}
	%	The grade of $ \bb \in \comp^{\dn} $ with respect to $ \AA \in \comp^{\dn \times \dn} $ where $\AA$ is complex-symmetric, i.e., $ \AAT = \AA $ is the positive integer $ g(\AA, \bb) $ with respect to the Saunders subspaces and \eqref{eq:saunders}, i.e., $\text{dim}(\mathcal{S}_{t}(\AA, \bb)) = \min\{t,g(\AA,\bb)\} $.
	%\end{definition}
	%

We include the complex-symmetric case as part of \cref{alg:MINRES}. We note that although the bulk of \cref{alg:MINRES} is given for the Hermitian case, the particular modifications needed for complex-symmetric settings are highlighted on the right column of \cref{alg:MINRES}. Further details regarding the construction of MINRES iterations in this case are given in  \cref{app:CSMINRES}.
%

%All of the above steps constitute the MINRES algorithm, adapted for complex-symmetric systems, which is given in \cref{alg:MINRES}. 

\begin{lemma}
	\label{lemma:residual_CS}
	In \cref{alg:MINRES} for complex-symmetric matrices, $ \rrt $ can be updated via $\rrt = s_t^2 \rrtp - \phi_{t} \bc_t \vv_{t+1}$.
\end{lemma}
\begin{proof}
	By \cite[Eqn (B.1)]{choi2013minimal} and \cref{eq:block_Q}, it holds that
	\begin{align*}
		\rrt &= \phi_{t} \VVtn \QQt^{\H} \ee_{t+1} = \phi_{t} \begin{bmatrix}
			\VVt & \vv_{t+1}
		\end{bmatrix} \begin{bmatrix}
			\QQ_{t-1}^{\H} & \\
			& 1
		\end{bmatrix} \begin{bmatrix}
			s_t \ee_t \\
			-\bc_t
		\end{bmatrix} \\
		&= \phi_{t} s_t \VVt \QQ_{t-1}^{\H} \ee_t - \phi_{t} \bc_t \vv_{t+1} = s_t^2 \phi_{t-1} \VVt \QQ_{t-1}^{\H} \ee_t - \phi_{t} \bc_t \vv_{t+1},
	\end{align*}
	as desired. 
\end{proof}

We now give the procedure to obtain the pseudo-inverse solution of \cref{eq:least_squares} from \cref{alg:MINRES} when $ \AA $ is complex-symmetric. In our analysis, we use the Tagaki singular value decomposition (SVD) of $\AA$, which is ensured to be symmetric for such matrices \cite{bunse1988singular}. Specifically, the Tagaki SVD of $ \AA $ is given by 
\begin{align*} 
	\AA = \begin{bmatrix}
		\UU & \UUp
	\end{bmatrix} \begin{bmatrix}
		\SIGMA & \zero \\
		\zero & \zero
	\end{bmatrix} \begin{bmatrix}
		\UU & \UUp
	\end{bmatrix}^{\T} = \UU \SIGMA \UU^{\T},
\end{align*}
where $ \UU \in \comp^{\dn \times r_{\AA}} $ and $ \UUp \in \comp^{\dn \times (\dn - r_{\AA})} $ are orthogonal matrices and $ \Sigma \in \real^{r_{\AA} \times r_{\AA}} $ is a diagonal matrix containing $ r_{\AA} \leq d $ positive values. Note that $ \rg{\UUp} $ is the orthogonal complement of $ \rg{\UU} $, and hence $ [\UU \hspace*{2mm} \UUp] $ is unitary. From \cref{def:pinverse}, we can also see that 
\begin{align*}
	\AAd = \begin{bmatrix}
		\bUU & \bUUp
	\end{bmatrix} \begin{bmatrix}
		\SIGMA^{-1} & \zero \\
		\zero & \zero
	\end{bmatrix} \begin{bmatrix}
		\UU & \UUp
	\end{bmatrix}^{\H}.
\end{align*}

\begin{theorem}[MN refinement for Complex-symmetric Systems]
	\label{thm:CSMINRES_dagger}
	Let $ \xxg $ be the final iterate of \cref{alg:MINRES} for solving \cref{eq:least_squares} with the complex-symmetric matrix $ \AA \in \comp^{\dn \times \dn} $. 
	\begin{enumerate}
		\item If $ \rrg = \bm{0} $, then $ \xxg = \AAd \bb $.
		\item If $ \rrg \neq \bm{0} $, let us define the MN refinement vector 
		\begin{align}
			\label{eq:CSMINRES_MN-refinement}
			\hxxd = \xxg - \frac{\dotprod{\brrg, \xxg}}{\vnorm{\brrg}^2} \brrg = \left(\eye_{d} - \frac{\brrg \rrg^{\T}}{\vnorm{\brrg}^{2}}\right) \xxg,
		\end{align}
		where $ \rrg = \bb - \AA \xxg $. Then, we have $ \hxxd = \AAd \bb $.
	\end{enumerate}
\end{theorem}
\begin{proof}
	The proof is similar to  that of \cref{thm:MINRES_dagger}, and hence is omitted.
\end{proof}
%\begin{proof}
%   When $ \bb \in \rg{\AA} $, we have $ \xxd = \hxxd = \xxg $ by \cite[Theorem 3.1]{choi2013minimal}. So we only need to consider the case $ \bb \notin \rg{\AA} $. We can write $ \xxg \in \Span(\bVVt) $ as $\xxg = b_1 \bbb + b_2 \pp_g$, where $\pp_g \in \rg{\bAA} $. So, using the matrix $ \UU $ in \cref{eq:decomp_A_Ad_CS} and noting that $ \bUU^{\H} = \UU^{\T} $, we obtain
%   \begin{align*}
	%       \xxg &= b_1 (\bUU \UUT + \bUUp \UUpT) \bbb + b_2 \pp_g = b_1 \bUU \UUT \bbb + b_2 \pp_g + b_1 \bUUp \UUpT \bbb.
	%   \end{align*}
%   Just as in the proof of \cref{thm:MINRES_dagger}, we see that $\AA \pp = \AA \xxg$ and $ \qq \in \Null(\AA) = \rg{\bUUp} $, where
%   \begin{align*}
	%       \pp &\defeq b_1 \bUU \UUT \bbb + b_2 \pp_g\\
	%       \qq &\defeq b_1 \bUUp \UUpT \bbb = b_1 \cj{\UUp \UUpH \bb} = b_{1} \brrg.
	%   \end{align*}
%   Since $ \pp $ satisfies the normal equation, i.e., $\pp$ is a solution to \cref{eq:least_squares}, \cref{lemma:MINRES_solve_exactly} implies $\pp = \AAd\bb$. Furthermore, since $ \pp \perp \qq $,  we can compute 
%   \begin{align*}
	%       b_1  = \frac{\dotprod{\brrg, \xxg}}{\vnorm{\brrg}^2},
	%   \end{align*}
%   which gives the desired result.
%\end{proof}

Similar to \cref{prop:MN-refinement_t}, the MN refinement in \cref{eq:CSMINRES_MN-refinement} can be regarded as the  orthogonal projection of the final iterate $ \xxg $ onto $ \bAA \mathcal{S}_{g}(\AA, \bb) $. To establish this, we first show that any iteration of MINRES in the complex-symmetric setting can be formulated as a special Petrov-Galerkin orthogonality condition with respect to the underlying Saunders subspace.
\begin{lemma}
	\label{lemma:CS_property}
	In \cref{alg:MINRES} for complex-symmetric matrices, we have
	\begin{align*}
		\AA \brrt &= \phi_t (\gamma_{t+1} \vvtn + \delta_{t+2} \vv_{t+2}), \\%\label{eq:Abr_CS} \\
		\dotprod{\bxx_i, \AA \brrt} &= 0, \quad 0 \leq i \leq t, \\%\label{eq:Abr_perp_CS}  \\
		\rrt &\perp \AA \mathcal{S}_{t}(\bAA, \bbb), %\label{eq:Saunders_cond} 
	\end{align*}
	where $ 0 \leq t \leq g $ and $ \vv_{g+1} = \bm{0} $.
\end{lemma}
\begin{proof}
	From \cite[Eqn (B.1)]{choi2013minimal}, we have
	\begin{align*}
		\AA \brrt = \phi_t \AA \bVV_{t+1} \QQtT \ee_{t+1} = \phi_t \begin{bmatrix}
			\VVt & \vvtn
		\end{bmatrix} \hTTt \QQtT \ee_{t+1} = \phi_t (\gamma_{t+1} \vvtn + \delta_{t+2} \vv_{t+2})
	\end{align*}
	where the last equality is obtained using a similar reasoning as in \cite[Lemma 3.3]{choi2011minres}. We then obtain $ \AA \brrt \perp \Span \{\vv_1, \vv_2, \ldots, \vv_{t} \} $, i.e., $ \dotprod{\AA \brrt, \zz} = 0 $ for any $ \zz \in \mathcal{S}_{t}(\AA, \bb) $ by the orthogonality of columns of $ \VV_{t+1} $. This implies $ \dotprod{\rrt, \AA \bzz} = 0 $, from which the desired result follows.
\end{proof}

Using \cref{lemma:CS_property}, we get the equivalent of \cref{prop:MN-refinement_t} in the complex-symmetric setting.
\begin{proposition}
	\label{prop:MN-refinement_t_CS}
	$ \hxxtN $ is the orthogonal projection of $ \xxt $ onto $ \bAA \mathcal{S}_{t}(\AA, \bb) $ in MINRES, where
	\begin{align*}
		\hxxtN \triangleq \xxt - \frac{\dotprod{\brrt, \xxt}}{\vnorm{\brrt}^2} \brrt = \left(\eye_{d} - \frac{\brrt \rrt^{\T}}{\vnorm{\brrt}^{2}}\right) \xxt.
	\end{align*}
\end{proposition}

\section{Pseudo-inverse solutions in preconditioned MINRES}
\label{sec:pMINRES}

To solve systems involving ill-conditioned matrices and, in turn, to speed up iterative procedures, preconditioning is a very effective strategy, indispensable for large ill-conditioned problems \cite{saad2011numerical,saad2003iterative,bjorck1996numerical,bjorck2015numerical}.
The primary focus of research efforts on preconditioning has been on solving consistent systems of linear equations \cite{saad2003iterative,saad2011numerical,benzi2005numerical,gould2017state,morikuni2013inner,morikuni2015convergence,rozloznik2002krylov,bjorck1996numerical,bjorck2015numerical}. 
For instance, when solving the linear system $ \AA \xx = \bb $, where $ \AA \succ \zero $, one can consider a positive definite matrix $ \MM \succ \zero $ and solve the transformed problem 
	\begin{align}
		\label{eq:pCG}
		(\MMs \AA \MMs) \; \txx = \MMs \bb,
	\end{align}
	followed by $\xx = \MMsin \txx$.
Formulation \cref{eq:pCG} is often referred to as split preconditioning. Alternative formulations also exist, e.g., the celebrated preconditioned CG solves the left-preconditioned system $\MMs \AA \xx = \MMs \bb$ \footnote{Recall that preconditioned CG solves the left-preconditioned system by operating under the inner product $\dotprod{.,.}_{\MMsin}$, whereas CG employs the standard Euclidean inner product $\dotprod{.,.}$. In this context, even if $\MMs \AA$ is not Hermitian, it remains self-adjoint with respect to this new inner product.}. 
If $\MM$ is chosen appropriately, the transformed problem can exhibit significantly improved conditioning compared to the original formulation. 
For such problems, the preconditioning matrix $ \MM $ is naturally always taken to be non-singular, see, e.g., \cite{gould2017state,morikuni2013inner,morikuni2015convergence,saad2011numerical,saad2003iterative,bjorck1996numerical,bjorck2015numerical,benzi2005numerical,choi2013minimal,choi2014algorithm}. For many iterative procedures, $\MM$ is \emph{required} to be positive definite, as in, e.g., preconditioned CG and MINRES. 

In stark contrast, 
%positive definite preconditioners for linear least-squares problems, i.e., inconsistent systems where $\bb \notin \rg{\AA} $, may 
%has received less attention \cite{bjorck1996numerical,bjorck1999preconditioners}. 
%
the challenge of obtaining a pseudo-inverse solution to inconsistent systems \cref{eq:least_squares} with $\bb \notin \rg{\AA} $ is greatly exacerbated when the underlying matrix is preconditioned. This is in part due to the fact that, unlike the case of consistent linear systems, the particular choice of the preconditioner can have a drastic effect on the type of obtained solutions in inconsistent settings.
%This is in part due to the fact that, in addition to potential speed-ups in computations, the particular choice of preconditioner can have a drastic effect on the type of obtained solutions in such inconsistent systems. 
We illustrate such effect in the following simple example.

\begin{example}
	For any $ a > 0 $, consider 
	\begin{align*}
		\AA = \begin{bmatrix}
			a & 0 \\
			0 & 0
		\end{bmatrix}, \quad \bb = \begin{bmatrix}
			1  \\
			1 
		\end{bmatrix} \notin \rg{\AA} \quad \text{and} \quad \MMs = \begin{bmatrix}
			b & 1 \\
			1 & 1
		\end{bmatrix}
	\end{align*}
	Let $ \yy $ to be the pseudo-inverse solution of the preconditioned least-squares problem related to \cref{eq:pCG} and $ \xx = \MMs \yy $.
	Then
	\begin{align*}
		\xx - \xxd = \frac{b+1}{a (b^2+1)^2}\begin{bmatrix}
			b^2+1 \\
			b^2 + b + 2
		\end{bmatrix} \quad \text{and} \quad \rr - \rrd = -\frac{b+1}{b^2+1} \begin{bmatrix}
			1 \\
			0
		\end{bmatrix},
	\end{align*}
	where $ \rr \triangleq \bb - \AA \xx $ and $ \rrd \triangleq \bb - \AA \xxd $. We see that, even though the preconditioner is positive definite, the pseudo-inverse solutions and the respective residuals in the original problem and in the preconditioned version do not coincide. In fact, we can only expect to obtain $ \rr - \rrd \rightarrow \zero $ and $ \xx - \xxd \rightarrow \zero $ as $ b \rightarrow \infty $. Hence, the choice of preconditioner in inconsistent least-squares plays a much more subtle role than in the usual consistent linear system settings. 
\end{example}

Though the construction of preconditioners is predominantly done on an ad hoc basis, a wide range of positive (semi)definite preconditioners can be created by introducing a matrix $ \SS \in \comp^{\dn \times m} $, where $ m \geq 1 $, and letting $ \MM = \SS \SSH $. In this context, $ \SS $ is often referred to as the sub-preconditioner. Naturally, depending on the structure of $ \SS $, the matrix $ \MM $ can be singular. 
Square, yet singular, sub-preconditioners $ \SS \succeq \zero $ have been explored in the context of GMRES \cite{elden2012solving} and MINRES \cite{hong2022preconditioned}. 
Here, we consider an arbitrary sub-preconditioner $ \SS $ and study our MN refinement in the presence of the resulting preconditioner $ \MM = \SS \SSH $.

Before delving any deeper, we state several facts that are used throughout the rest of this section. Let the economy SVD of $ \SS $ and $ \MM $ be given by\
\begin{align}
	\label{eq:def_Md_M_Sd_S}
	\SS = \PP \SIGMA \KK^{\H}, \quad \MM = \PP \SIGMA^2 \PP^{\H},
\end{align}\
where $ \PP \in \comp^{\dn \times r} $ and $ \KK \in \comp^{m \times r} $ are orthogonal matrices, $ \SIGMA \in \real^{r \times r} $ is a diagonal matrix with nonzero diagonals elements, and $ r \leq \min\{d,m\} $ is the rank of $\MM$. \cref{fact:M_S} lists several relationships among the factors in \cref{eq:def_Md_M_Sd_S}. The proofs are trivial and hence omitted.
\begin{fact}
	\label{fact:M_S}
	With \cref{eq:def_Md_M_Sd_S}, we have
	\begin{subequations}
		\label{eq:M_S_property}
		\begin{align}
			\MMd  &= \PP \SIGMA^{-2} \PP^{\H}, \quad \SSd = \KK \SIGMA^{-1}\PP^{\H}, \label{eq:Md_Sd}\\
			\MM &= \SS \SSH, \quad \MMd = \SSdH \SSd, \label{eq:Md_STS}\\
			\MM \MMd &= \MMd \MM = \SS \SSd = \SSdH \SSH = \PP \PPH, \label{eq:Md_PPH} \\
			\SSd \SS &= \SSH \SSdH = \KK \KKH.
			%           \quad \SST \SSdT = \bSSd \bSS = \bKK \KKT, 
			\label{eq:Md_KKH}
		\end{align}
	\end{subequations}
\end{fact}

\subsection{Preconditioned Hermitian systems}
\label{sec:pMINRES_H}
As one might expect, for linear least-squares problems arising from inconsistent systems, one can relax the invertibility or positive definiteness requirement for $ \MM $. Specific to our work here, when $ \AA $ is Hermitian, the preconditioner $ \MM \in \comp^{\dn \times \dn} $ will only be required to be PSD. Recently, \cite{hong2022preconditioned} has considered such PSD preconditioners for MINRES in the context of solving consistent linear systems with singular matrix $ \AA $. It was shown that, since $ \AA $ and $ \MM $ are Hermitian and $\dotprod{\AA\MM\vv, \ww}_{\MM} = \dotprod{\vv,\AA\MM\ww}_{\MM}$ for any $ \vv,\ww \in \real^{d} $, the matrix $ \AA\MM $ is self-adjoint with respect to the $\dotprod{\vv,\ww}_{\MM} = \vv^{\H} \MM \ww$ inner-product. This, in turn, gives rise to the right-preconditioning of $ \AA $ as $ \AA\MM $. In such a consistent setting, the iterates of the right-preconditioned MINRES algorithm are generated as
\begin{subequations}
	\label{eq:right_prec}
	\begin{align}
		\cxxt &= \argmin_{\cxx \in \Kt{\AA\MM,\bb}} \vnorm{\bb - \AA \MM \cxx}_{\MM}^2, \label{eq:right_prec_yy}\\
		\xxt &= \MM \cxxt, \label{eq:right_prec_xx}
	\end{align}
\end{subequations}
where  $ \xxo = \cxx_{0} = \bm{0} $ and $ \vnorm{.}^{2}_{\MM} \defeq \dotprod{.,.}_{\MM} $ defines a semi-norm. 

This motivates us to consider the following more general problem 
\begin{align}
	\label{eq:right_prec_02}
	\xxt = 
	%   \argmin_{\xx \in \MM \Kt{\AA\MM,\bb}} \vnorm{\bb - \AA \MM \MMd \xx}_{\MM}^2 = 
	\argmin_{\xx \in \MM \Kt{\AA\MM,\bb}} \vnorm{\bb - \AA \xx}_{\MM}^2.
\end{align}
It can be easily shown that the iterate $ \xxt $ from \cref{eq:right_prec} is a solution to \cref{eq:right_prec_02}. However, the converse is not necessarily true, i.e., for any $ \xxt $ from \cref{eq:right_prec_02}, there might not exist $ \cxxt \in \Kt{\AA\MM,\bb} $ satisfying \cref{eq:right_prec_xx}. 
Indeed, for a given $ \xxt $ from \cref{eq:right_prec_02}, to have \cref{eq:right_prec_xx} we must have $ \cxxt  = \MMd \xxt + \left(\eye - \MMd\MM\right) \zz $ for some $ \zz \in \real^{d}$. In other words, we must have $ \cxxt \in \MMd\MM \Kt{\AA\MM,\bb} \oplus \Null(\MM) $. Dropping the inconsequential term $ \Null(\MM) $, it follows that the solutions of \cref{eq:right_prec,eq:right_prec_02} coincide in the restricted case where $ \bb \in \rg{\AA} = \rg{\MM} $, which is precisely the setting considered in \cite{hong2022preconditioned}. 
Here, we consider the more general formulation \cref{eq:right_prec_02}. 
Since
\begin{align*}
	\MM \Kt{\AA\MM,\bb} = \MM \Span\left\{\bb,\AA \MM \bb,\ldots,(\AA\MM)^{t-1} \bb \right\} = \Kt{\MM\AA,\MM\bb},
\end{align*}
the formulation \cref{eq:right_prec_02} is also equivalent to 
\begin{align}
	\label{eq:right_prec_03}
	\xxt = \argmin_{\xx \in \Kt{\MM\AA,\MM\bb}} \vnorm{\bb - \AA \xx}_{\MM}^2.
\end{align}
Formulation \cref{eq:right_prec_03} constitutes the starting point for our theoretical analysis and algorithmic derivations.

\subsubsection*{Derivation of the preconditioned MINRES algorithm}
Consider formulation \cref{eq:right_prec_03} with $ \AA \in \herm^{d \times d} $ and suppose the preconditioner $ \MM $ is given as $ \MM  = \SS \SSH $, where $ \SS \in \comp^{\dn \times m} $ for some $ m \geq 1 $ is the sub-preconditioner matrix. We have  
\begin{align*}
	\xxt = \argmin_{\xx \in \Kt{\MM\AA,\MM\bb}} \vnorm{\bb - \AA \xx}_{\MM}^2 = \argmin_{\xx \in \SS \Kt{\SSH\AA\SS,\SSH\bb}} \vnorm{\SSH\left(\bb - \AA \xx\right)}^{2}.
\end{align*}
This allows us to consider the equivalent problem
\begin{subequations}
	\label{eq:preq_MINRES}
	\begin{align}
		\txxt &= \argmin_{\txx \in \Kt{\tAA, \tbb}} \vnorm{\tbb - \tAA \txx}^2, \label{eq:pMINRES}\\
		\xxt &= \SS \txxt, \label{eq:def_x_r}
	\end{align}
\end{subequations}
where we have defined 
\begin{align}
	\label{eq:def_tA_tb_tr_hr}
	\tAA \triangleq \SSH \AA \SS \in \herm^{m \times m}, \quad \text{and} \quad \tbb \triangleq \SSH \bb \in \comp^{m}.
\end{align}
The residual of the system in \cref{eq:pMINRES} is given by $\trrt = \tbb - \tAA \txxt$, which implies
\begin{align}
	\label{eq:trrt_rrt}
	\trrt = \SSH \rrt, \quad \text{where} \quad \rrt = \bb - \AA \xxt.
\end{align}

Clearly, the matrix in the least-squares problem \cref{eq:pMINRES} is itself Hermitian. As a result, the Lanczos process \cite{paige1975solution,liu2022minres} underlying the MINRES algorithm applied to \cref{eq:pMINRES} amounts to
\begin{align}
	\label{eq:lanczos}
	\betatn \tvv_{t+1} = \tAA \tvvt - \alphat \tvvt - \betat \tvv_{t-1}, \quad 1 \leq t \leq \tg.
\end{align}
Denoting
\begin{align}
	\label{eq:def_z_w_d}
	\betat \tvvt = \SSH \zzt, \quad \wwt = \MM \zzt = \betat \SS \tvvt,
\end{align}
and using \cref{eq:def_z_w_d,eq:Md_STS}, we obtain
\begin{align*}   
	\SSH \zz_{t+1} = \frac{1}{\betat} \SSH \AA \MM \zzt - \frac{\alphat}{\betat} \SSH \zzt - \frac{\betat}{\beta_{t-1}} \SSH \zz_{t-1}.
\end{align*}
This relation allows us to define the following three-term recurrence relation for $ \zzt $: 
\begin{align}
	\label{eq:updates_z_w}
	\zztn = \frac{1}{\betat} \AA \MM \zzt - \frac{\alphat}{\betat} \zzt - \frac{\betat}{\beta_{t-1}} \zz_{t-1} = \frac{1}{\betat} \AA \wwt - \frac{\alphat}{\betat} \zzt - \frac{\betat}{\beta_{t-1}} \zz_{t-1},
\end{align}
with $\zz_{0} = \zero$, $ \zz_1 = \bb $, and $ \beta_0 = \beta_1 = \| \tbb \| $, where $ 1 \leq t \leq \tg $. Thus, the updates \cref{eq:updates_z_w} yield the Lanczos process \cref{eq:lanczos}.

To present the main result of this section, we first need to establish a few technical lemmas. \cref{lemma:Kry_tA_tb} expresses the corresponding Krylov subspace $ \Kt{\tAA, \tbb} $ in terms of the underlying components $ \AA$, $\bb$, $\SS$, and $\MM$. 
\begin{lemma}
	\label{lemma:Kry_tA_tb}
	For any $1 \leq t \leq \tg$, we have $\Kt{\tAA, \tbb} = \SSd \Kt{\MM \AA, \MM \bb}$.
\end{lemma}
\begin{proof}
	By \cref{eq:def_tA_tb_tr_hr,eq:Md_STS,eq:Md_KKH} using the identity $\SSH = \SSd \SS \SSH = \SSd\MM $, we have
	\begin{align*}
		\Kt{\tAA, \tbb} = \Span\{\SSH \bb, \SSH \AA \MM \bb, \ldots, \SSH \AA \left[\MM \AA\right]^{t-2} \MM \bb\} = \SSd \Kt{\MM \AA, \MM \bb}.
	\end{align*}
\end{proof}

Next, we give some properties of the vectors $ \zz_{i} $ and $ \ww_{i} $. 
\begin{lemma}
	\label{lemma:Kry_spans_z_w}
	For any $1\leq t \leq \tg$, we have
	\begin{align*}
		\Span \{\ww_{1}, \ww_{2}, \ldots, \ww_{t}\} &= \Kt{\MM \AA, \MM \bb},\\
		\PP \PPH \Span\{\zz_{1}, \zz_{2}, \ldots, \zz_{t}\} &= \PP \PPH \Kt{\AA \MM, \bb}, 
	\end{align*}
	where $ \PP $ is as in \cref{eq:def_Md_M_Sd_S}.
	Furthermore, $ \ww_{\tg+1} = \zero $ and $ \zz_{\tg+1} \in \Null(\MM) $. 
\end{lemma}
\begin{proof}
	By \cref{eq:def_z_w_d,lemma:Kry_tA_tb,eq:Md_PPH,eq:Md_STS}, we have
	\begin{align*}
		\Span\left\{\ww_1, \ww_2, \ldots, \ww_t\right\} &= \SS \times \Span\left\{\tvv_1, \tvv_2, \ldots, \tvv_t\right\} = \SS \Kt{\tAA, \tbb} \\
		&= \SS \SSd \Kt{\MM \AA, \MM \bb} = \Kt{\MM \AA, \MM \bb},
	\end{align*}
	and 
	\begin{align*}
		\PP \PPH \Span\left\{\zz_1, \zz_2, \ldots, \zz_t\right\} &= \SSdH \times \Span\left\{\tvv_1, \tvv_2, \ldots, \tvv_t\right\} = \SSdH \SSd \Kt{\MM \AA, \MM \bb} \\
		&= \MMd \Kt{\MM \AA, \MM \bb} = \PP \PPH \Kt{\AA \MM, \bb},
	\end{align*}
	where the vectors $\tvv_i$ are those appearing in the Lanczos process \cref{eq:lanczos}.
	
	At $ t = \tg + 1 $, from \cref{eq:def_z_w_d} and the fact that $ \tvv_{\tg+1} = \zero $, it follows that $ \ww_{\tg+1} = \zero $ and $ \SSH \zz_{\tg+1} = \zero $, i.e., $ \zz_{\tg+1} \in \Null(\MM) $.
\end{proof}

We now show how the coefficients $ \alpha_{t} $ and $ \beta_{t} $ in \cref{eq:updates_z_w} can be computed, i.e., how to construct $ \hTTt $ in \cref{eq:tridiagonal_T}. Note that, by construction, $ \beta_{1} = \| \tbb \| = \sqrt{\dotprod{\zz_{1}, \ww_{1}}} $.
\begin{lemma}
	\label{lemma:alpha_beta}
	For $ 1 \leq t \leq \tg $, we can compute $ \alpha_{t} $ and $ \beta_{t} $ in \cref{eq:updates_z_w} as
	\begin{align*}
		\alphat = \frac{1}{\betat^2} \dotprod{\wwt, \AA \wwt}, \quad \betatn = \sqrt{\dotprod{\zztn, \wwtn}}.
	\end{align*}
\end{lemma}
\begin{proof}
	By \cref{eq:lanczos,eq:def_z_w_d,eq:Md_STS} and the orthonormality of the Lanczos vectors $ \tvv_{i} $, we obtain
	\begin{align*}
		\alphat = \dotprod{\tvvt, \tAA \tvvt} = \dotprod{\frac{1}{\betat} \SSH \zzt, \left(\SSH \AA \SS\right) \frac{1}{\betat} \SSH \zzt} = \frac{1}{\betat^2} \dotprod{\wwt, \AA \wwt}.
	\end{align*}
	By \cref{eq:def_z_w_d,eq:Md_STS} and the facts that $ \beta_{t+1} > 0 $ and $ \vnorm{\tvvtn} = 1 $ for any $ 1 \leq t \leq \tg-1 $, we get
	\begin{align*}
		\betatn = \sqrt{\vnorm{\betatn \tvvtn}^2} = \sqrt{\dotprod{\SSH \zztn, \SSH \zztn}} = \sqrt{\dotprod{\zztn, \wwtn}}.
	\end{align*}
	For $ t = \tg $, this relation continues to hold because $ \beta_{\tg+1} = 0 $ and by \cref{lemma:Kry_spans_z_w}, $ \ww_{\tg+1} = \zero $.
\end{proof}

Now, recall the update direction within MINRES is given by the three-term recurrence relation \cite{liu2022minres,paige1975solution}
\begin{align}
	\label{eq:updates_v_d}
	\tdd_{t} = \frac{1}{\gamma^{[2]}_{t} } \left(\tvv_{t} - \epsilon_{t} \tdd_{t-2} - \delta^{[2]}_{t} \tdd_{t-1}\right), \quad 1 \leq t \leq \tg,
\end{align}
where $ \tdd_{0} = \tdd_{-1} = \zero $. It is easy to see that $ \tddt \in \Span\{\tvv_1, \tvv_2, \ldots, \tvv_{t} \} $. Let $ \ddt $ be a vector such that $ \ddt = \SS \tddt $, and define $\dd_{0} = \dd_{-1} = \bm{0} $. Multiplying both sides of \cref{eq:updates_v_d} by $ \SS $ gives
\begin{align*}
	\dd_{t} = \SS \tddt = \frac{1}{\gamma^{[2]}_{t} } \left(\SS \tvvt- \epsilon_{t} \SS \tdd_{t-2} - \delta^{[2]}_{t} \SS \tdd_{t-1}\right) = \frac{1}{\gamma^{[2]}_{t} } \left(\SS \tvvt- \epsilon_{t} \dd_{t-2} - \delta^{[2]}_{t} \dd_{t-1}\right),
\end{align*}
which, using \cref{eq:def_z_w_d}, yields the following three-term recurrence relation:
\begin{align}
	\label{eq:update_x_d}
	\dd_{t} = \frac{1}{\gamma^{[2]}_{t} } \left(\frac{1}{\betat} \wwt - \epsilon_{t} \dd_{t-2} - \delta^{[2]}_{t} \dd_{t-1}\right), \quad 1 \leq t \leq \tg.
\end{align}
Clearly, we have ``\cref{eq:update_x_d} $\implies$ \cref{eq:updates_v_d}''. To see this, note that \cref{lemma:Kry_tA_tb} and the identity $ \SSd \SS \SSd = \SSd $ imply that $\SSd \ddt = \SSd \SS \tddt = \tddt$ and $ \tvvt = \SSd \SS \tvvt = \SSd \wwt / \betat $. Hence, multiplying both sides of \cref{eq:update_x_d} by $ \SSd $, we obtain \cref{eq:updates_v_d}. Also, by \cref{lemma:Kry_spans_z_w}, this construction implies that $ \ddt \in \Span\{\ww_1, \ww_2, \ldots, \ww_{t} \} = \Kt{\MM \AA, \MM \bb} $ for $1 \leq t \leq \tg$. 
%
%   This can be easily seen by multiplying both sides of \cref{eq:update_x_d} by $ \SSd $ and noting that $\SSd \wwt/\betat = \SSd \SS \tvvt = \tvvt$ by \cref{lemma:Kry_tA_tb}.
%
Finally, from \cref{eq:def_x_r,eq:update_x_d}, it follows that the update is given by
\begin{align*}
	\xxtn = \SS \txxtn = \SS (\txxt + \taut \tddt) = \xxt + \taut \ddt.
\end{align*}
Initializing with $ \xx_0 = \bm{0} $ gives $ \xxt \in \Kt{\MM \AA, \MM \bb}$. 

Furthermore, let us define $ \tddt = \SSH \cddt $, $ \txxt = \SSH \cxxt $ and construct updates
\begin{align}
	\label{eq:update_cxx_d}
	\cdd_{t} = \frac{1}{\gamma^{[2]}_{t} } \left(\frac{1}{\betat} \zzt - \epsilon_{t} \cdd_{t-2} - \delta^{[2]}_{t} \cdd_{t-1}\right) \quad \text{and} \quad \cxxt = \cxx_{t-1} + \taut \cddt,
\end{align}
where $ \cxx_0 = \cdd_0 = \cdd_{-1} = \bm{0} $. By \cref{lemma:Kry_spans_z_w}, we obtain $ \cddt \in \Span\{\zz_1, \zz_2, \ldots, \zz_{t} \} = \Kt{\AA \MM, \bb} $ and $ \cxxt \in \Kt{\AA \MM, \bb} $. Multiplying both sides of \cref{eq:update_cxx_d} by $ \SSH $ recovers the three-term recurrence relation \cref{eq:updates_v_d}. Clearly, by \cref{eq:def_x_r}, $ \xxt = \SS \txxt = \SS \SSH \cxxt = \MM \cxxt $ is exactly the solution in \cref{eq:right_prec_yy}.

We also have the following recurrence relation on quantities that are connected to the residual.

\begin{lemma}
	\label{lemma:residuals}
	For any $1\leq t \leq \tg$, define $ \hrrt $ as 
	\begin{align}
		\label{eq:def_tA_tb_tr_hr_hrrt}
		\hrrt = \SS \trrt \in \comp^{\dn},
	\end{align}
	where $\trrt = \tbb - \tAA \txxt$. We have
	\begin{subequations}
		\label{eq:residuals}
		\begin{align}
			\hrrt &= s_{t}^2 \hrrtp - \frac{\phi_{t} c_{t}}{\betatn} \wwtn, \label{eq:hr}  \\
			\PP \PPH \rrt &= \PP \PPH \left(s_t^2 \rr_{t-1} - \frac{\phi_{t} c_t}{\beta_{t+1}} \zz_{t+1}\right), \label{eq:PPHr}
		\end{align}
	\end{subequations}
	where $ \PP $ and $ \ww_{t} $ are respectively defined in \cref{eq:def_Md_M_Sd_S,eq:def_z_w_d}.
\end{lemma}
\begin{proof}
	By \cref{eq:def_tA_tb_tr_hr,eq:def_z_w_d} and multiplying both sides of the residual update in \cref{alg:MINRES} by $ \SS $, we get
	\begin{align*}
		\hrrt = \SS \trr_{t} = s_{t}^2 \SS \trr_{t-1} - \phi_{t} c_{t} \SS \tvv_{t+1} = s_{t}^2 \hrrtp - \frac{\phi_{t} c_{t}}{\betatn} \wwtn.
	\end{align*}
	Now, by \cref{eq:def_tA_tb_tr_hr,eq:Md_PPH,eq:Md_STS,eq:def_z_w_d,eq:trrt_rrt} and noting that $ \SSdH = \SSdH \SSd \SS = \MMd \SS $, we have
	\begin{align*}
		\PP \PPH \rrt = \SSdH \trrt = \MMd \SS (s_t^2  \trr_{t-1} - \phi_{t} c_{t} \tvv_{t+1}) 
		= \PP \PPH(s_t^2  \rrtp - \frac{\phi_{t} c_t}{\beta_{t+1}} \zz_{t+1}).
	\end{align*}
\end{proof}

The following result shows that $ \xxt $ and $ \hrrt $ belong to a certain Krylov subspace.
\begin{lemma}
	\label{lemma:Kry_x_hr}
	With \cref{eq:def_tA_tb_tr_hr,eq:def_x_r} and $ \xx_{0} = \zero $, we have for any $ 1 \leq t \leq \tg $,
	\begin{align*}
		\xxt \in \Kt{\MM \AA, \MM \bb},
	\end{align*} 
	and $\Span \{\hrr_{0}, \hrr_{1}, \ldots, \hrr_{t-1}\} = \Kt{\MM \AA, \MM \bb}$.
\end{lemma}
\begin{proof}
	The first result follows from \cref{lemma:Kry_spans_z_w} and our construction of the update direction, as discussed above. The second result is obtained using \cref{eq:hr,lemma:Kry_spans_z_w} and $ \hrr_{0} = \ww_{1} $.
\end{proof}

\begin{remark}
	\label{rm:all_vectors_subspaces}
	Inspecting \cref{eq:PPHr} reveals that a new vector $ \crrt $ can be defined as 
	\begin{align}
		\label{eq:cr}
		\crrt = s_t^2 \crrtp - \frac{\phi_{t} c_t}{\beta_{t+1}} \zz_{t+1}, \quad 1 \leq t \leq \tg,
	\end{align}
	where $ \crr_{-1} \defeq \zero $. This way, we always have $ \PP \PPH \crrt = \PP \PPH \rrt $. Hence, we can compute $ \PP \PPH \rrt $ without performing an extra matrix-vector product $\AA \xxt$. From \cref{eq:updates_z_w,eq:update_cxx_d} as well as $\zz_{0} = \zero$ and $ \zz_1 = \bb $, we have $\zzt \in \Kt{\AA \MM, \bb}$, which in turn implies $ \cddt, \cxxt, \crrtp \in \Kt{\AA \MM, \bb} $. In addition, from \cref{lemma:Kry_x_hr,lemma:Kry_spans_z_w,eq:update_x_d} it follows that $ \wwt, \ddt, \xxt, \hrrtp \in \Kt{\MM \AA, \MM \bb}$. We conclude that all vectors $ \wwt, \ddt, \xxt, \hrrtp $ and $ \zzt, \cddt, \cxxt, \crrtp $ are solely dependent on $ \MM $, and hence are invariant to the particular choices of $ \SS $ that satisfy $\MM = \SS\SSH$.   
\end{remark}    

All these iterative relations give us the preconditioned MINRES algorithm, constructed from \cref{eq:right_prec_03} and  depicted in \cref{alg:pMINRES}.

\begin{algorithm}[t!]
	\caption{The left column presents the preconditioned MINRES algorithm for the Hermitian case, while the right column outlines the particular modifications needed for complex-symmetric settings. \\[0.5ex] \textbf{Hermitian} \hfill \textbf{\cpp{Complex-symmetric}}}
	\label{alg:pMINRES}
	\begin{tikzpicture}[scale=1]
		\node at (0,0) {\begin{minipage}{.98\linewidth}
				\begin{algorithmic}[1]
					\STATE \textbf{Initialization:} $ \AA, \MM \in \herm^{n \times n} $, $ \bb \in \comp^{n} $, $ \MM \succeq \zero $. \hfill\COMMENT{$ \cpp{\AAT = \AA} \in \comp^{n \times n} $}
					\vspace{1mm}
					\STATE $ \crr_0 = \zz_1 = \bb $, $ \ww_1 = \MM \zz_1 $, $ \beta_1 = \sqrt{\dotprod{\zz_1, \ww_1}} $, \hfill\COMMENT{$ \ww_1 = \MM \cpp{\bzz_1} $}
					\vspace{1mm}
					\STATE $ \hrr_{0} = \ww_1 $, $ \phi_0 = \beta_1 $, $ c_0 = -1 $, $ s_0 = \delta_1 = 0 $, \hfill\COMMENT{$ \beta_1 = \sqrt{\dotprod{\cpp{\bzz_1}, \ww_1}} $}
					\vspace{1mm}
					\STATE $ \zz_0 = \xx_0 = \dd_0 = \dd_{-1} = \zero $, $ t = 1 $,
					\vspace{1mm}
					\WHILE{True}
					\vspace{1mm}
					\STATE $ \qqt = \AA \wwt / \beta_{t} $, $\alpha_{t} = \dotprod{\wwt / \beta_{t},\qqt} $, \hfill\COMMENT{$\alpha_{t} = \dotprod{\cpp{\bwwt} / \beta_{t},\qqt} $}
					\vspace{1mm}
					\STATE $ \zztn = \qqt - \alphat \zzt / \beta_{t} - \betat \zz_{t-1} / \beta_{t-1} $, 
					\vspace{1mm}
					\STATE $ \wwtn = \MM \zztn $, \hfill\COMMENT{$ \wwtn = \MM \cpp{\bzztn} $}
					\vspace{1mm}
					\STATE $ \beta_{t+1} = \sqrt{\dotprod{\zz_{t+1}, \ww_{t+1}}} $, \hfill\COMMENT{$ \beta_{t+1} = \sqrt{\dotprod{\cpp{\bzz_{t+1}}, \ww_{t+1}}} $}
					\vspace{2mm}
					\STATE  $\displaystyle 
					\begin{bmatrix}
						\delta^{[2]}_{t} & \epsilon_{t+1} \\
						\gamma_{t} & \delta_{t+1}
					\end{bmatrix} = \begin{bmatrix}
						c_{t-1} & s_{t-1} \\
						s_{t-1} & -c_{t-1}
					\end{bmatrix} \begin{bmatrix}
						\delta_{t} & 0 \\
						\alpha_{t} & \beta_{t+1}
					\end{bmatrix}$, \hfill\COMMENT{$\begin{bmatrix}
							\cpp{\bc_{t-1}} & s_{t-1} \\
							s_{t-1} & -c_{t-1}
						\end{bmatrix}$}
					\vspace{1mm}
					\STATE $ \gamma_{t}^{[2]} = \sqrt{\abs{\gamma_{t}}^2 + \beta_{t+1}^2} $,
					\vspace{1mm}
					\IF{ $ \gamma_{t}^{[2]} \neq 0 $ }
					\vspace{1mm}
					\STATE $ c_{t} = \gamma_{t}/\gamma_{t}^{[2]} $, $ s_{t} = \beta_{t+1} / \gamma_{t}^{[2]} $, 
					\vspace{1mm}
					\STATE $ \tau_{t} = c_{t} \phi_{t-1} $, $ \phi_{t} = s_{t} \phi_{t-1} $, \hfill\COMMENT{$ \tau_{t} = \cpp{\bc_{t}} \phi_{t-1} $}
					\vspace{1mm}
					\STATE $ \dd_{t} = \left(\wwt/\betat - \delta^{[2]}_{t} \dd_{t-1} - \epsilon_{t} \dd_{t-2} \right) / \gamma^{[2]}_{t} $, 
					%       $ \hdd_{t} = \left(\zzt/\betat - \delta^{[2]}_{t} \hdd_{t-1} - \epsilon_{t} \hdd_{t-2} \right) / \gamma^{[2]}_{t} $
					\vspace{1mm}
					\STATE $ \xx_{t} = \xx_{t-1} + \tau_{t} \dd_{t} $, 
					\vspace{1mm}
					\STATE $ \crrt = s_t^2 \crr_{t-1} - \phi_{t} c_t \zz_{t+1} / \beta_{t+1} $, \hfill\COMMENT{$ \crrt = s_t^2 \crr_{t-1} - \phi_{t} {\cpp{\bc_t}} \zz_{t+1} / \beta_{t+1} $}
					\vspace{1mm}
					\IF{$ \beta_{t+1} \neq 0 $}
					\vspace{1mm}
					\STATE  $ \hrr_{t} = s_{t}^2 \hrr_{t-1} - \phi_{t} c_{t} \ww_{t+1} / \beta_{t+1} $, \hfill\COMMENT{$ \hrr_{t} = s_{t}^2 \hrr_{t-1} - \phi_{t} \cpp{\bc_{t} \bww_{t+1}} / \beta_{t+1} $}
					\vspace{1mm}
					\ELSE
					\vspace{1mm}
					\STATE $ \hrr_{t} = \zero $, \textbf{return}\;\;$ \xx_{t} $.
					%\vspace{1mm}
					\ENDIF
					\ELSE
					\vspace{1mm} 
					\STATE $ c_{t} = \tau_{t} = 0 $, $ s_{t} = 1 $, $ \phi_{t} = \phi_{t-1} $, 
					\STATE $ \hrr_{t} = \hrr_{t-1} $, $ \crr_{t} = \crr_{t-1} $, $ \xx_{t} = \xx_{t-1} $, \textbf{return}\;\;$ \xx_{t} $.
					\vspace{1mm}
					\ENDIF
					\vspace{1mm}
					\STATE $ t \leftarrow t+1 $,
					\vspace{1mm}
					\ENDWHILE
				\end{algorithmic}
		\end{minipage}};
		\draw[densely dotted] (1.5,8) -- (1.5,-8);
	\end{tikzpicture}
\end{algorithm}

As a sanity check, if we set the ideal preconditioner $ \MM = \AAd $, we obtain $ \AAd \bb = \xxt \in \krylovt{\AAd \AA, \AAd \bb} $ for any $ t $, which by \cref{lemma:Kry_x_hr} implies that \cref{alg:pMINRES} terminates at the very first iteration.

We now give the MN refinement procedure for an iterate of \cref{alg:pMINRES} applied to Hermitian systems.

\begin{theorem}[MN refinement for Preconditioned Hermitian Systems]
	\label{thm:MN-refinement}
	Let $ \xxtg $ be the final iterate of \cref{alg:pMINRES} with the Hermitian matrix $ \AA \in \herm^{\dn \times \dn} $. 
	\begin{enumerate}
		\item If $ \hrrtg = \bm{0} $, then we must have $ \xxtg = \SS \txxtg $, where $ \txxtg = \tAAd \tbb $.
		\item If $ \hrrtg \neq \bm{0} $, define the MN refinement vector 
		\begin{align}
			\label{eq:P_MN-refinement}
			\hxxd = \xxtg - \frac{\dotprod{\crrtg, \xxtg}}{\dotprod{\hrrtg, \crrtg}} \hrrtg = \left(\eye_{d} - \frac{\hrrtg \crrtg^{\H}}{\hrrtg^{\H} \crrtg}\right)\xxtg,
		\end{align}
		where $\hrrtg$ and $ \crrtg $ are defined respectively in \cref{eq:hr,eq:cr}. Then, we have $ \hxxd = \SS \txxd$, where $\txxd = \tAAd \tbb $. 
	\end{enumerate}
\end{theorem}
\begin{proof}   
	Following the proof in \cref{thm:MINRES_dagger}, we obtain the desired result for $ \hrrtg = \bm{0} $ trivially. Now we consider $ \hrrtg \neq \bm{0} $. By \cref{thm:MINRES_dagger}, we have
	\begin{align*}
		\txxtg = \txxd + \frac{\dotprod{\trrtg, \txxtg}}{\vnorm{\trrtg}^2} \trrtg.
	\end{align*}
	By \cref{eq:def_x_r,eq:Md_PPH,eq:trrt_rrt,eq:def_tA_tb_tr_hr_hrrt},
	\begin{align*}
		\xxtg = \SS \txxtg = \SS \txxd + \frac{\dotprod{\SSH \rrtg, \txxtg}}{\dotprod{\trrtg, \SSH \rrtg}} \SS \trrtg = \SS \txxd + \frac{\dotprod{\crrtg, \xxtg}}{\dotprod{\hrrtg, \crrtg}} \hrrtg.
	\end{align*}
	Hence, the desired result follows because $ \xxg, \hrrg \in \rg{\SS \SSd} = \rg{\PP \PPH} $ as well as $ \PP \PPH \crrt = \PP \PPH \rrt $ from \cref{rm:all_vectors_subspaces}.
\end{proof}

Similar to \cref{prop:MN-refinement_t}, the general MN refinement at every iteration $ 1 \leq t \leq \tg $ can be formulated as 
\begin{align}
	\label{eq:P_MN-refinement_t}
	\hxx^{\N}_t = \xxt - \frac{\dotprod{\crrt, \xxt}}{\dotprod{\hrrt, \crrt}} \hrrt.
\end{align}
Of course, in general, the MN refinement solution does not coincide with the pseudo-inverse solution to the original unpreconditioned problem \cref{eq:least_squares}. However, it turns out that in the special case where $ \rg{\MM} = \rg{\AA} $,  one can indeed recover $ \AAd\bb $ using our MN refinement.
\begin{corollary}
	\label{coro:pseudo_P}
	If $ \rg{\MM} = \rg{\AA} $, then $ \hrrtg = \bm{0} $ and $ \xxtg = \AAd \bb $.
\end{corollary}
\begin{proof}
	From $ \rg{\MM} = \rg{\AA} $, it follows that $ \tbb \in \rg{\tAA} $ and $ \hrrtg = \bm{0} $. By applying \cite[Fact 6.4.10]{bernstein2009matrix} twice as well as using \cref{eq:Md_PPH,thm:MN-refinement}, we obtain
	\begin{align*}
		\xxtg = \SS \tAAd \tbb = \SS \left[\SS^\H \AA \SS\right]^{\dagger} \SS^\H \bb = \PP \PPH \AAd \PP \PPH \bb = \AAd \bb.
	\end{align*}
\end{proof}

In the more general case where $ \rg{\MM} \neq \rg{\AA} $, we might have $ \tbb \not\in \rg{\tAA} $. This significantly complicates the task of establishing conditions for the recovery of $\AAd\bb$ from the preconditioned problem. More investigations in this direction are left for future work.

By construction, \cref{alg:pMINRES} is analytically equivalent to \cref{alg:sub_pMINRES}, which involves applying \cref{alg:MINRES} for solving \cref{eq:pMINRES} to obtain iterates $ \txxt, \; 1 \leq t \leq \min \{m, \tg\} $, and then recovering $ \xxt = \SS \txxt \in \comp^{\dn} $ by \cref{eq:def_x_r}.
In this light, when $ m \ll \dn $, the sub-preconditioned algorithm \cref{alg:sub_pMINRES} can be regarded as a dimensionality-reduced version of \cref{alg:pMINRES}, i.e., we essentially first project onto the lower dimensional space $ \comp^{m} $, use \cref{alg:sub_pMINRES}, and then  project back onto the original space $ \comp^{\dn} $. This allows for significantly less storage/computations in \cref{alg:sub_pMINRES} than what is required in \cref{alg:pMINRES}; see experiments in \cref{sec:exp_image} for more discussion. 

\begin{algorithm}
	\caption{The left column presents the sub-preconditioned MINRES algorithm for the Hermitian case, while the right column outlines the particular modifications needed for complex-symmetric settings. \\[0.5ex]  \textbf{Hermitian} \hfill \textbf{\cpp{Complex-symmetric}}}
	\label{alg:sub_pMINRES}
	\begin{tikzpicture}[scale=1]
		\node at (0,0) {\begin{minipage}{.98\linewidth}
				\begin{algorithmic}[1]
					\STATE \textbf{Input:} $ \AA \in \herm^{n \times n} $, $ \bb \in \comp^{n} $, $ \SS \in \comp^{n \times m} $, $ t \leq m $ \hfill\COMMENT{$ \cpp{\AAT = \AA} \in \comp^{n \times n} $}
					\vspace{1mm}
					\STATE Construct $ \tAA = \SSH \AA \SS $, $ \tbb = \SSH \bb $, \hfill\COMMENT{$ \tAA = \cpp{\SST} \AA \SS $, $ \tbb = \cpp{\SST} \bb $}
					\vspace{1mm}
					\STATE Obtain $ \txxt $ with \cref{alg:MINRES} using inputs $ \tAA $, $ \tbb $,
					\vspace{1mm}
					\STATE $ \xxt = \SS \txxt $, $ \hrrt = \SS \trrt $,
					\vspace{1mm}
					\STATE \textbf{Output:} $ \xxt $
				\end{algorithmic}
		\end{minipage}};
		\draw[densely dotted] (2.2,1.3) -- (2.2,-1.3);
	\end{tikzpicture}
\end{algorithm}

\subsection{Preconditioned complex-symmetric systems}
\label{sec:pMINRES_CS}
Our construction of the preconditioned MINRES in the complex-symmetric setting is motivated by \cite{choi2013minimal} and generally follows the standard framework as in \cite{choi2006iterative,choi2011minres,saad2003iterative,rozloznik2002krylov}; see \cref{app:pMINRES_CS} for more details.
We now give the MN refinement procedure for iterates of \cref{alg:pMINRES} applied to complex-symmetric systems.
\begin{theorem}[MN refinement for Preconditioned Complex-symmetric Systems]
	\label{thm:MN-refinement_CS}
	Let $ \xxtg $ be the final iterate of \cref{alg:pMINRES} and suppose that $ \AA \in \symm^{\dn \times \dn} $ is complex-symmetric. 
	\begin{enumerate}
		\item If $ \hrrtg = \bm{0} $, then we must have $ \xxtg = \SS \txxtg $, where $ \txxtg = \tAAd \tbb $.
		\item If $ \hrrtg \neq \bm{0} $, define the MN refinement vector 
		\begin{align*}
			\hxxd = \xxtg - \frac{\dotprod{\cj{\crrtg}, \xxtg}}{\dotprod{\cj{\hrrtg}, \cj{\crrtg}}} \cj{\hrrtg} = \left(\eye_{d} - \cjl{\frac{\hrrtg \crrtg^{\H} }{\hrrtg^{\H} \crrtg}}\right) \xxtg.
		\end{align*}
		where $\hrrtg$ and $ \crrtg $ are defined respectively in \cref{eq:hr_CS,eq:cr_CS}. Then, it holds that $ \hxxd = \SS \txxd$ where $\txxd = \tAAd \tbb $.      
	\end{enumerate}
\end{theorem}
\begin{proof}   
	The proof is similar to that of \cref{thm:MN-refinement}, and hence is omitted.
\end{proof}
%\begin{proof}   
%   The proof is similar to that of \cref{thm:MN-refinement}. From \cref{eq:def_x_r,lemma:Kry_x_hr_CS,eq:Md_PPH,eq:CSMINRES_MN-refinement,eq:trrt_rrt_CS,eq:def_tA_tb_tr_hr_hrrt_CS}, we get
%   \begin{align*}
	%       \xxtg = \SS \SSd \xxtg = \SS \txxtg &= \SS \tAAd \tbb + \frac{\dotprod{\cj{\trrtg}, \txxtg}}{\dotprod{\cj{\trrtg}, \cj{\trrtg}}} \SS \cj{\trrtg} \\
	%       &= \SS \tAAd \tbb + \frac{\dotprod{\SSH \brrtg, \txxtg}}{\dotprod{\cj{\trrtg}, \SSH \brrtg}} \SS \cj{\trrtg} \\
	%       &= \SS \tAAd \tbb + \frac{\dotprod{\cj{\crrtg}, \xxtg}}{\dotprod{\cj{\hrrtg}, \cj{\crrtg}}} \cj{\hrrtg}.
	%   \end{align*}
%   By further noting $ \xxtg, \cj{\hrrtg} \in \rg{\MM} = \rg{\PP \PPH} $ and $ \bPP \PPT \crrt = \bPP \PPT \rrt $ from \cref{rm:all_vectors_subspaces_CS}, we obtain the desired result.
%\end{proof}

Just like \cref{eq:P_MN-refinement_t}, a general MN refinement at every iteration $ 1 \leq t \leq \tg $ can be defined as 
\begin{align*}
	\hxx^{\N}_t = \xxt - \frac{\dotprod{\cj{\crrt}, \xxt}}{\dotprod{\cj{\crrt}, \cj{\hrrt}}} \cj{\hrrt}.
\end{align*}
Similar to \cref{coro:pseudo_P}, \cref{coro:pseudo_P_CS} shows that such MN refinement gives the pseudo-inverse solution of the original unpreconditioned problem, i.e., $ \AAd\bb $.
\begin{corollary}
	\label{coro:pseudo_P_CS}
	When $ \rg{\bMM} = \rg{\AA} $, we must have $ \hrrtg = \bm{0} $ and $ \xxtg = \AAd \bb $.
\end{corollary}
\begin{proof}
	From $ \rg{\bMM} = \rg{\AA} $, we obtain $ \tbb \in \rg{\tAA} $ and $ \hrrtg = \bSS \trrtg =  \bm{0} $. By applying \cite[Fact 6.4.10]{bernstein2009matrix} twice and using \cref{eq:Md_PPH,thm:MN-refinement}, we obtain
	\begin{align*}
		\xxtg = \SS \tAAd \tbb = \SS \left[\SS^\T \AA \SS\right]^{\dagger} \SS^\T \bb = \PP \PPH \AAd \bPP \PPT \bb = \AAd \bb.
	\end{align*}
\end{proof}

%\begin{remark}
%   \label{rm:reoth_CS}
%   The reorthogonalization strategy within \cref{alg:pMINRES} for complex-symmetric systems is done similarly to that in \cref{rm:reoth}, i.e., 
%   \begin{align*}
	%       \zzt = \zzt - \YY_{t-1} \zzt, \quad
	%       \wwt = \wwt - \YY_{t-1}^{\T} \wwt,
	%   \end{align*}
%   where $ \YY_0 = \zero $ and $\YY_{t} \triangleq [\frac{\zz_1}{\beta_1}, \frac{\zz_2}{\beta_2},\dots,\frac{\zz_{t}}{\beta_{t}}] [\frac{\ww_1}{\beta_1}, \frac{\ww_2}{\beta_2},\dots,\frac{\ww_{t}}{\beta_{t}}]^{\T} = \YY_{t-1} + \frac{1}{\betat^2} \zz_{t} \ww_{t}^{\T}$.
%\end{remark}

\begin{figure}[!thbp]
	\centering
	\subfigure[Relative Error of the Plain Iterates]{
		\includegraphics[scale = 0.38]{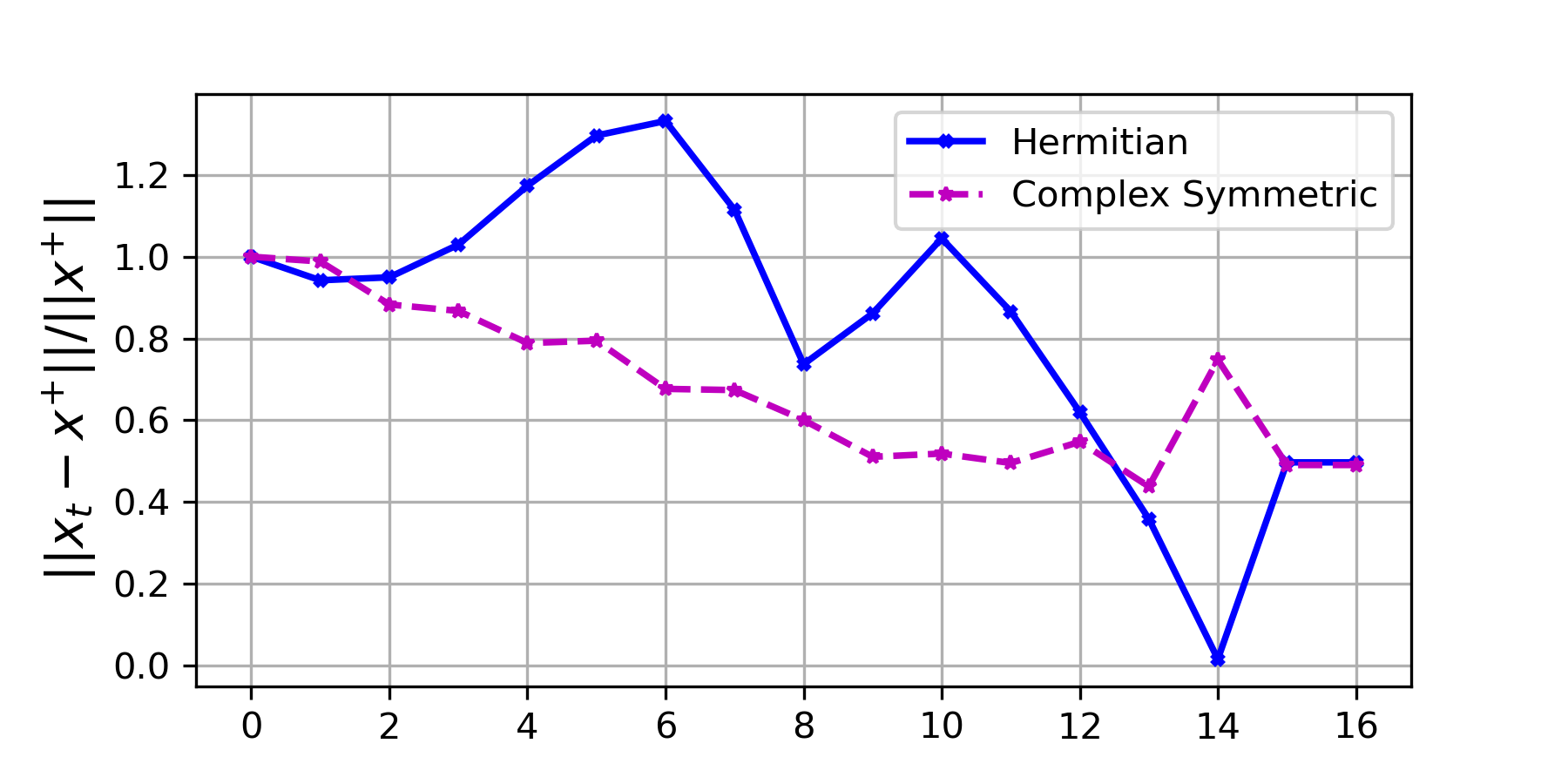}}
	\subfigure[Relative Error of the MN refinement Iterates]{
		\includegraphics[scale = 0.38]{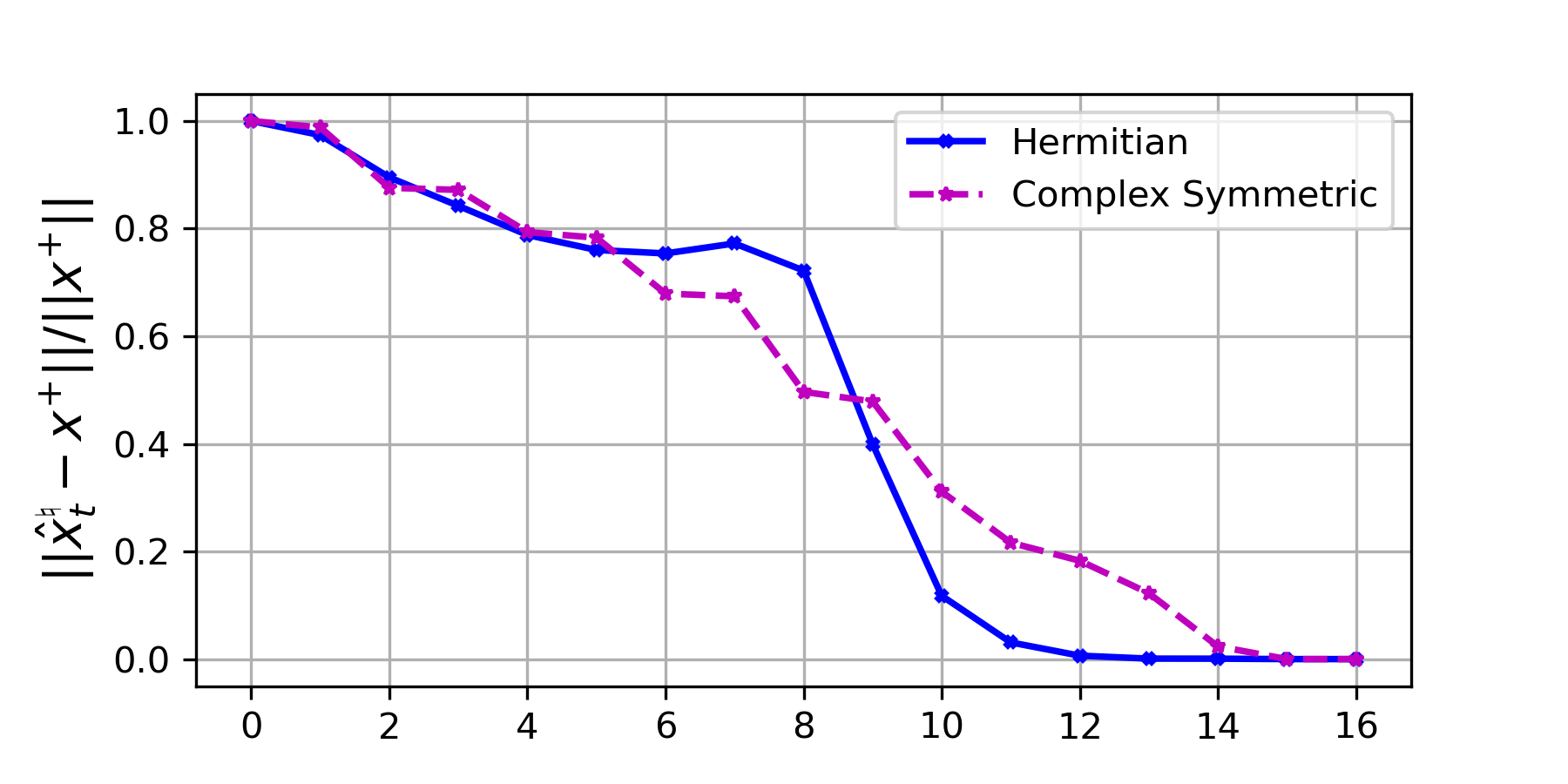}}
	\caption{Verifying \cref{thm:MINRES_dagger,thm:CSMINRES_dagger}. Relative error of the plain (a) and MN refinement (b) iterates of \cref{alg:MINRES} applied to \cref{eq:least_squares} with Hermitian (sold blue) and complex-symmetric (dashed magenta) matrices. The relative error is computed with respect to the pseudo-inverse solution $ \AAd \bb $ in each case. The $x$-axis denotes the iteration counter. \label{fig:pseudo}}
\end{figure}

We end this section by noting that, similar to Hermitian systems, in the case where $ \AA $ is complex-symmetric \cref{alg:pMINRES} is analytically equivalent to \cref{alg:sub_pMINRES}. Hence, when $ m \ll \dn $, the sub-preconditioned algorithm \cref{alg:sub_pMINRES} can be regarded as a dimensionality reduced version of \cref{alg:pMINRES} and allows for significantly less storage/computations than what is required in \cref{alg:pMINRES}. 

\section{Numerical experiments}
\label{sec:exp}

In this section, we provide several numerical examples illustrating our theoretical findings. We first verify our theoretical results in \cref{sec:exp:MN-refinement} using synthetic data. We subsequently aim to explore related applications in image deblurring in  \cref{sec:exp_image}. 

\subsection{Pseudo-inverse solutions via MN refinement: Theorems \ref{thm:MINRES_dagger} and \ref{thm:CSMINRES_dagger}}
\label{sec:exp:MN-refinement}
We first demonstrate \cref{thm:MINRES_dagger,thm:CSMINRES_dagger} on some simple and small synthetic problems. We consider \cref{eq:least_squares} with $d = 20$ and for Hermitian as well as complex-symmetric synthetic matrices. For both matrices, the rank is set to $ r = 15 $. We also let $ \bb = \mathbf{1} $, i.e., the vector of all ones, which ensures that the right-hand side vector $ \bb $ does not lie entirely in the range of these matrices.  In \cref{fig:pseudo}, we plot the relative error of the iterates of \cref{alg:MINRES}, $ \xxt $, as well as the MN refinement iterates $ \hat{\xx}_{t}^{\dagger} $ using \cref{eq:MINRES_MN-refinement} and \cref{eq:CSMINRES_MN-refinement} to the true pseudo-inverse solution $ \xxd = \AAd \bb $, namely $ \| \xxt - \xxd \| / \| \xxd \| $ and  $ \| \hxxtN - \xxd \| / \| \xxd \| $. As can be seen from \cref{fig:pseudo}, in sharp contrast to the plain MINRES iterates, the final MN refinement iterate in each setting coincides with the respective pseudo-inverse solution, corroborating the results of \cref{thm:MINRES_dagger,thm:CSMINRES_dagger}.

\begin{figure}[!thbp]
	\centering
	% \vspace*{-3mm}
	%   \hspace*{-1cm}
	\includegraphics[trim={7.8cm 2.5cm 7.8cm 2.5cm},clip,width=.24\textwidth]{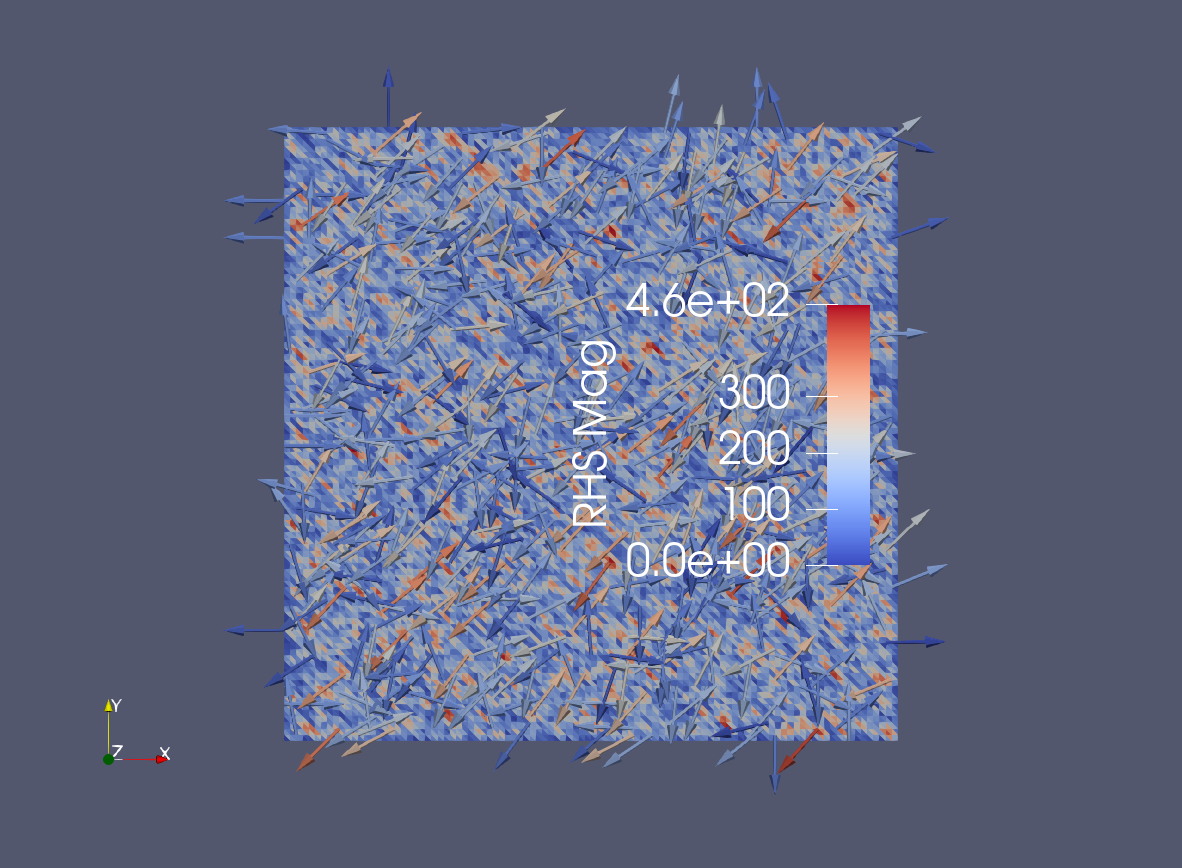}
	\includegraphics[trim={7.8cm 2.5cm 7.8cm 2.5cm},clip,width=.24\textwidth]{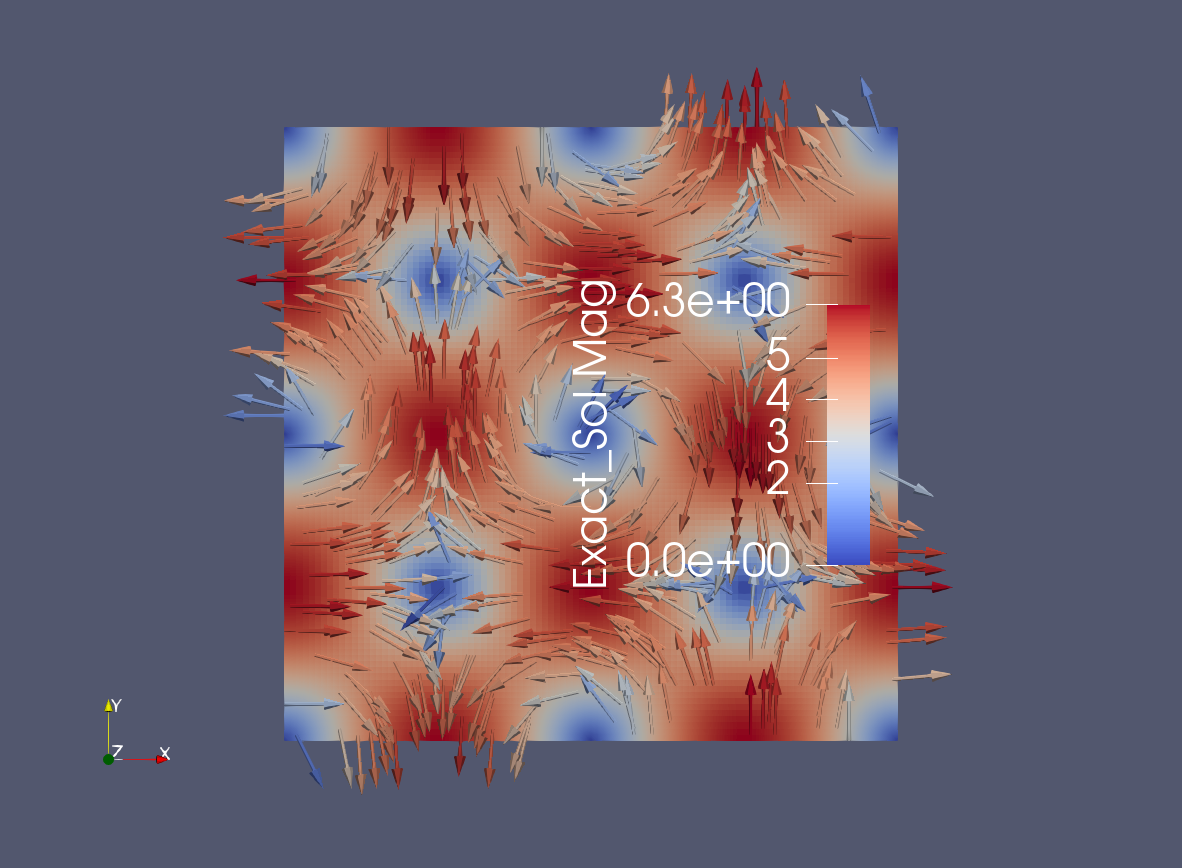}
	\includegraphics[trim={7.8cm 2.5cm 7.8cm 2.5cm},clip,width=.24\textwidth]{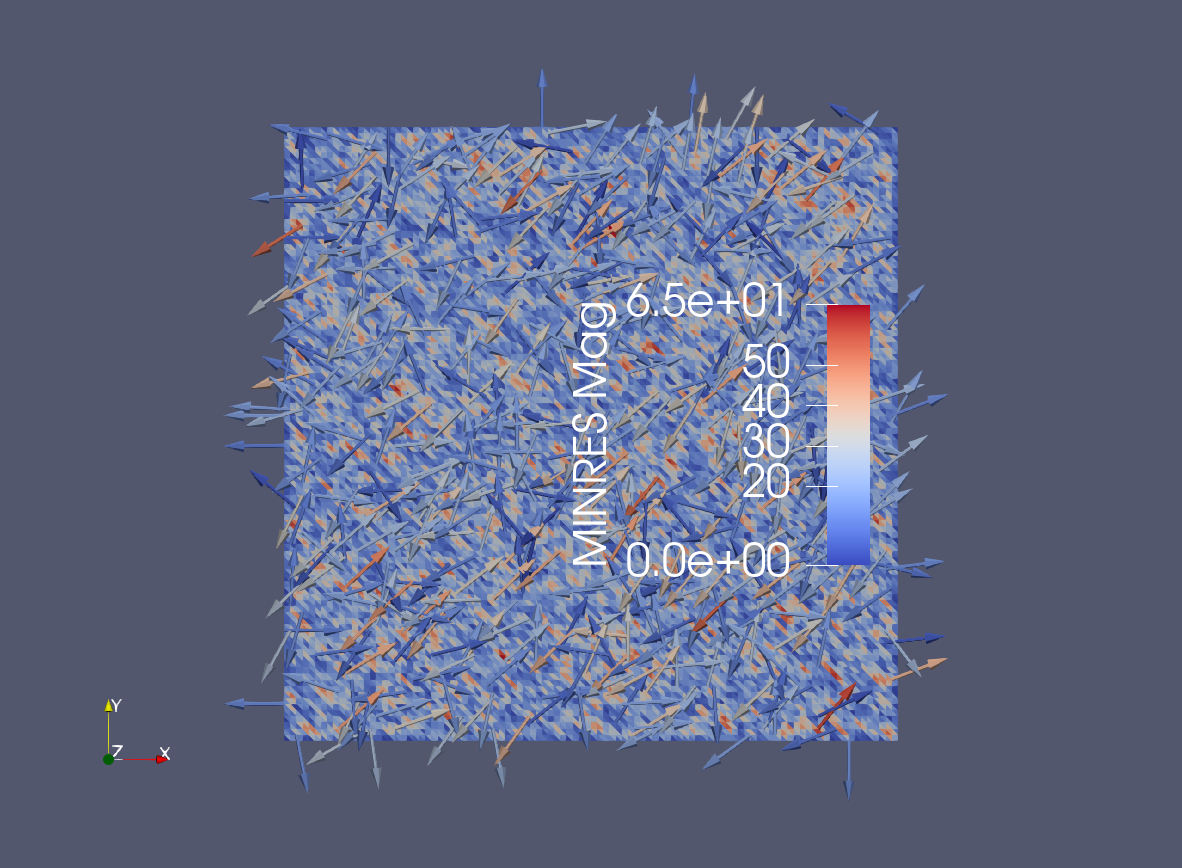}
	\includegraphics[trim={7.8cm 2.5cm 7.8cm 2.5cm},clip,width=.24\textwidth]{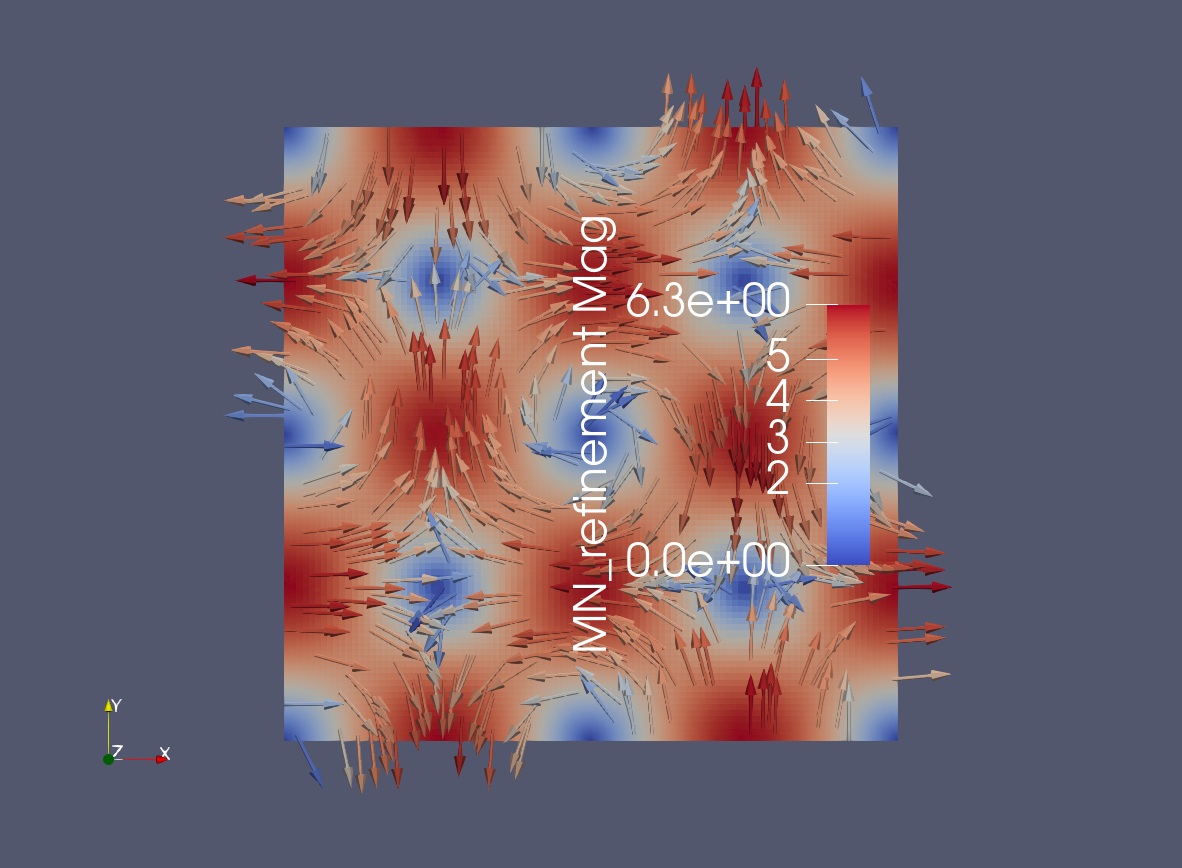}
	\hspace*{-0.1cm}    
	\includegraphics[trim={7.8cm 2.5cm 7.8cm 2.5cm},clip,width=.24\textwidth]{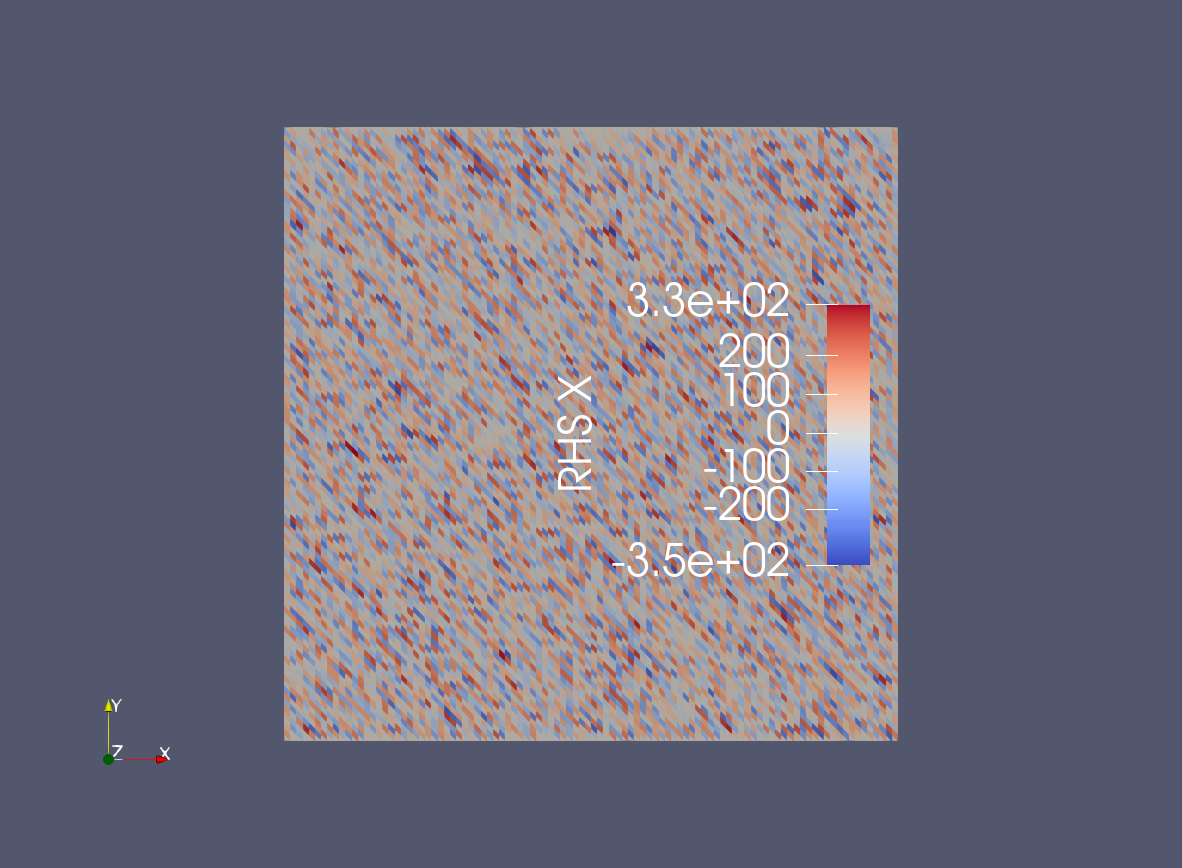}
	\includegraphics[trim={7.8cm 2.5cm 7.8cm 2.5cm},clip,width=.24\textwidth]{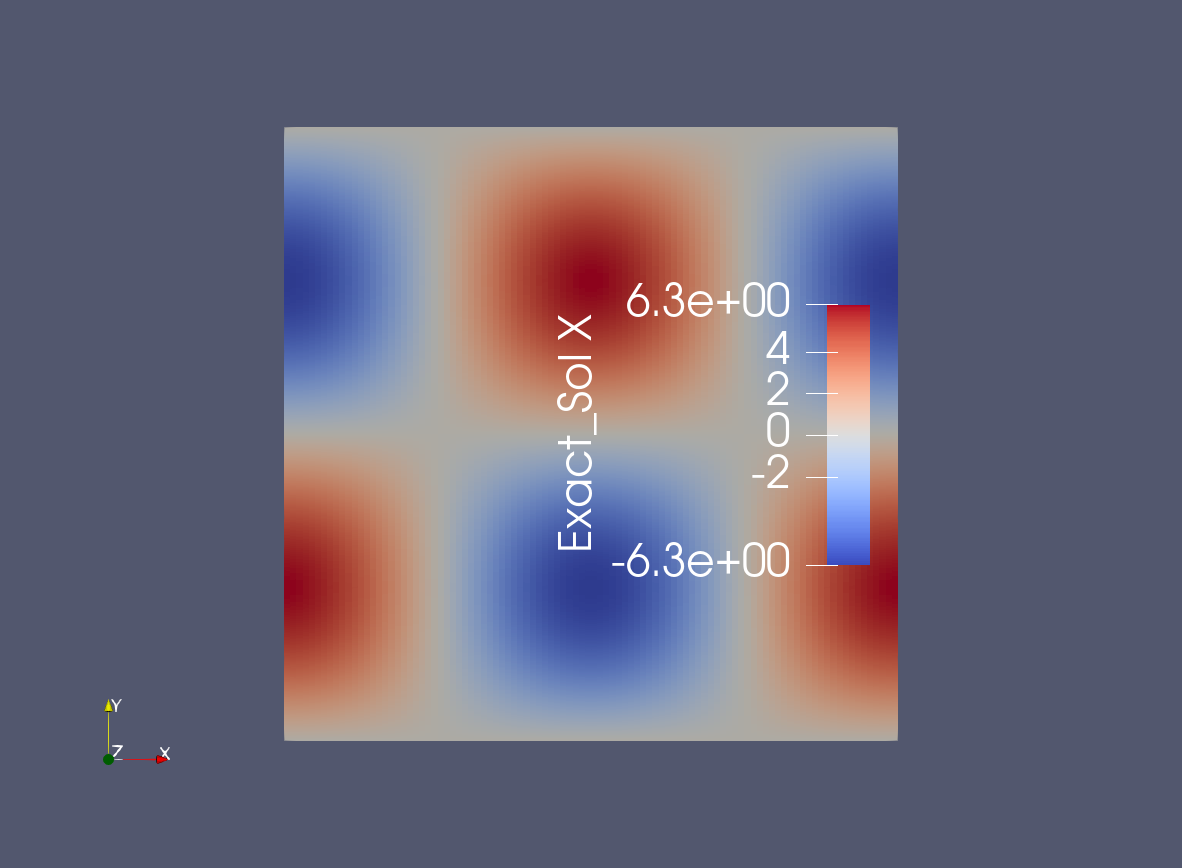}
	\includegraphics[trim={7.8cm 2.5cm 7.8cm 2.5cm},clip,width=.24\textwidth]{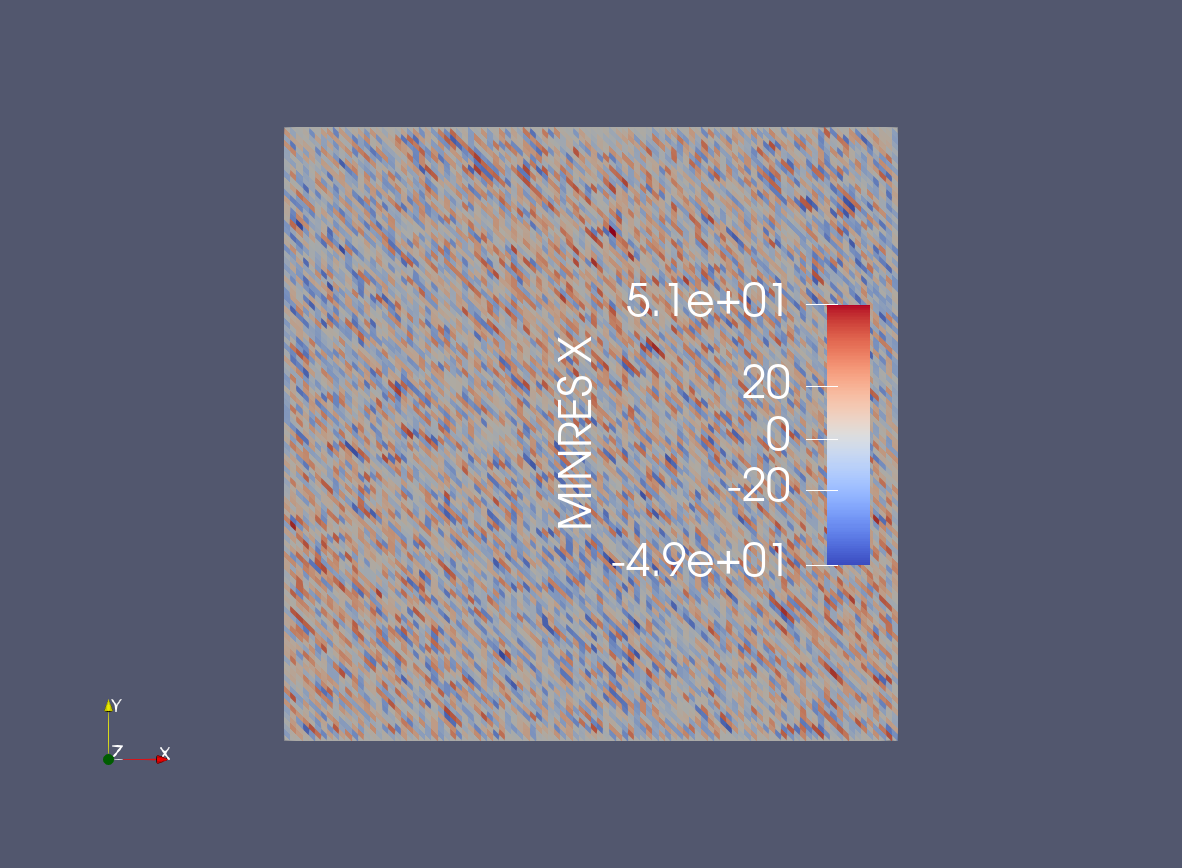}
	\includegraphics[trim={7.8cm 2.5cm 7.8cm 2.5cm},clip,width=.24\textwidth]{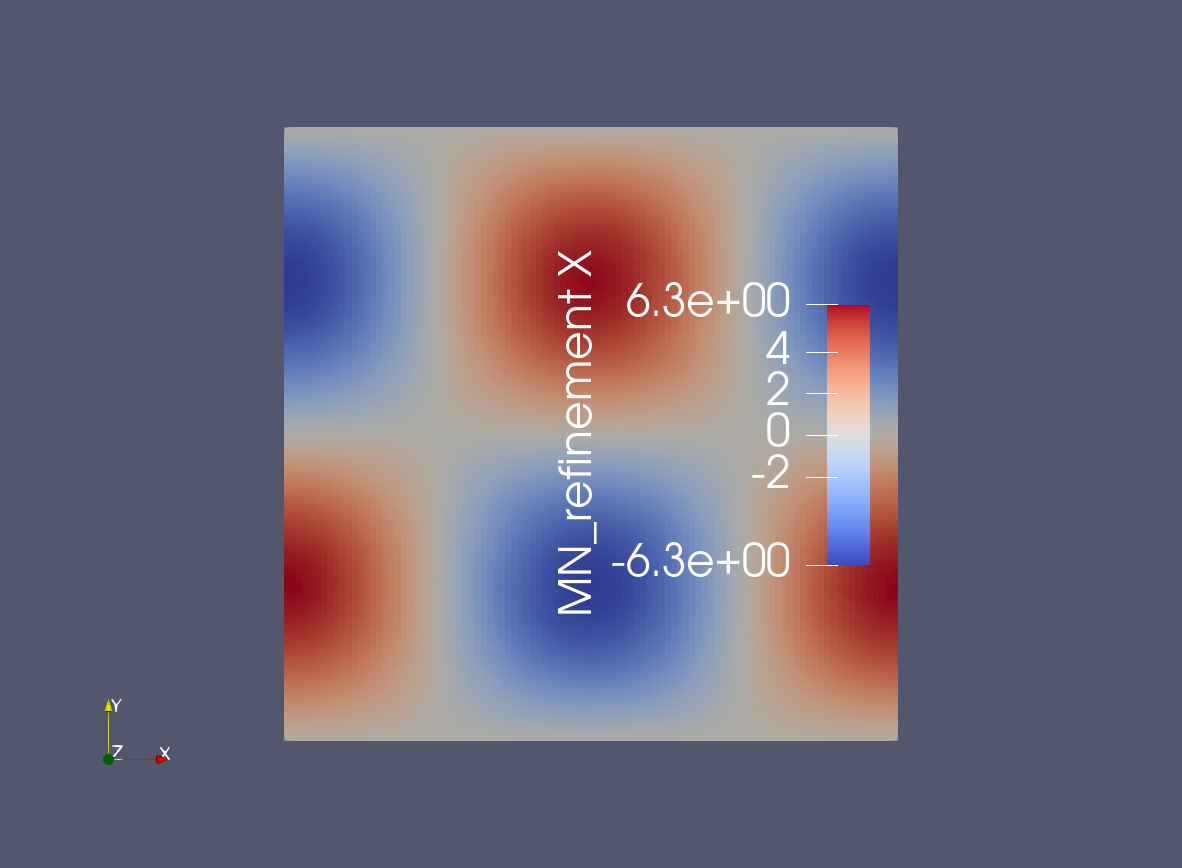}
	\hspace*{-0.02cm}    
	\includegraphics[trim={7.8cm 2.5cm 7.8cm 2.5cm},clip,width=.24\textwidth]{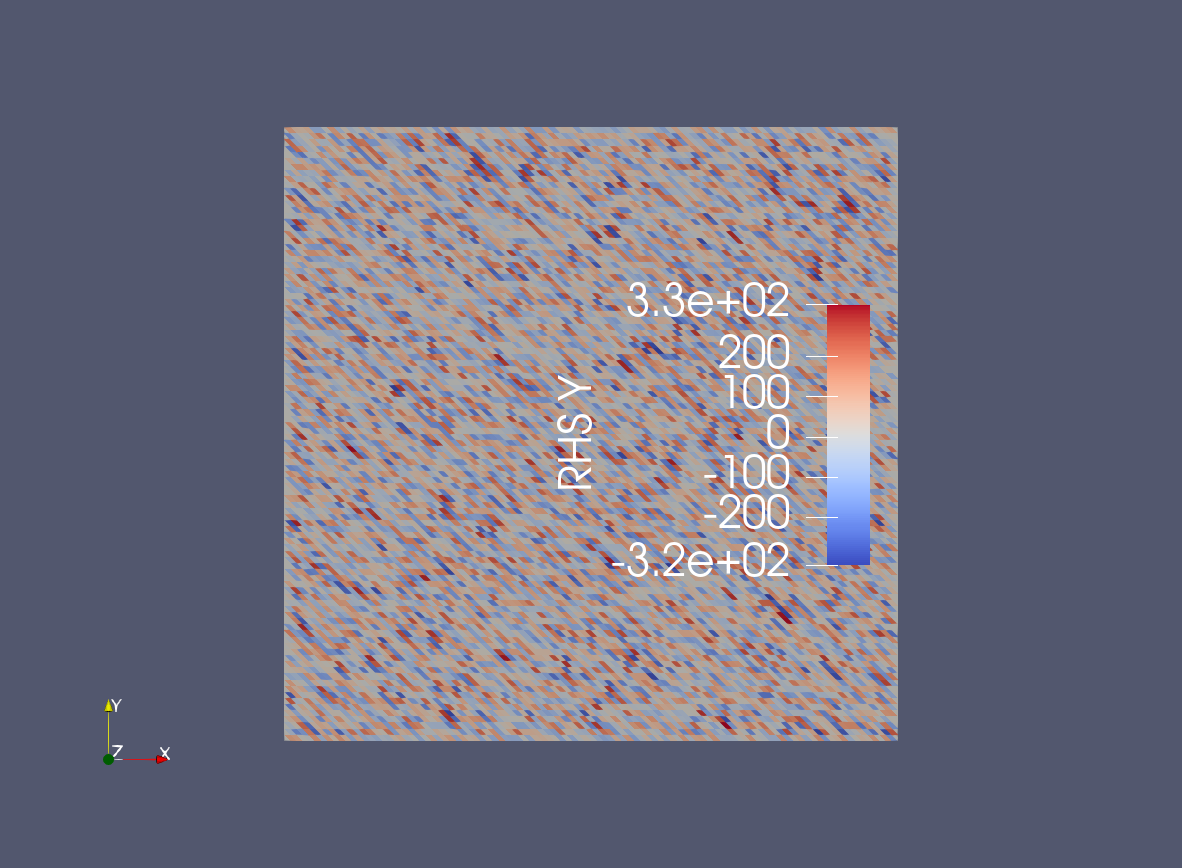}
	\includegraphics[trim={7.8cm 2.5cm 7.8cm 2.5cm},clip,width=.24\textwidth]{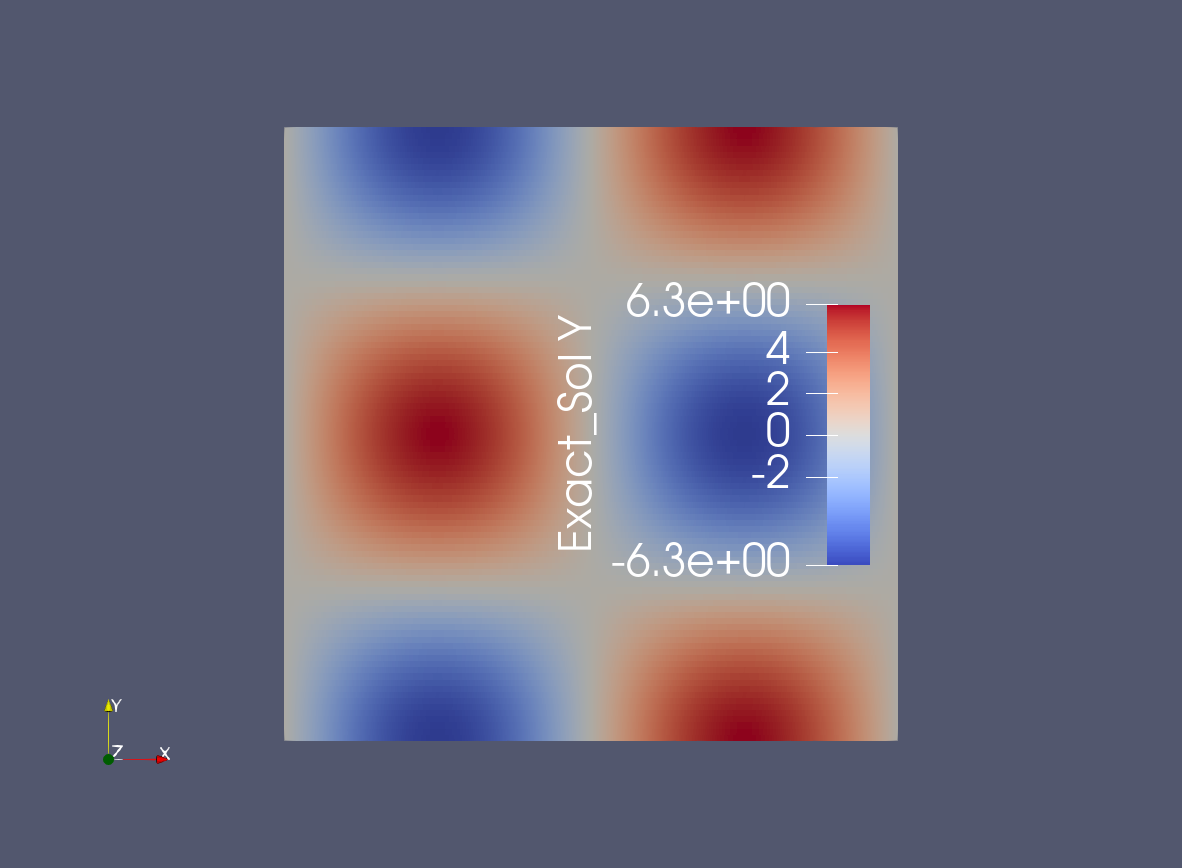}
	\includegraphics[trim={7.8cm 2.5cm 7.8cm 2.5cm},clip,width=.24\textwidth]{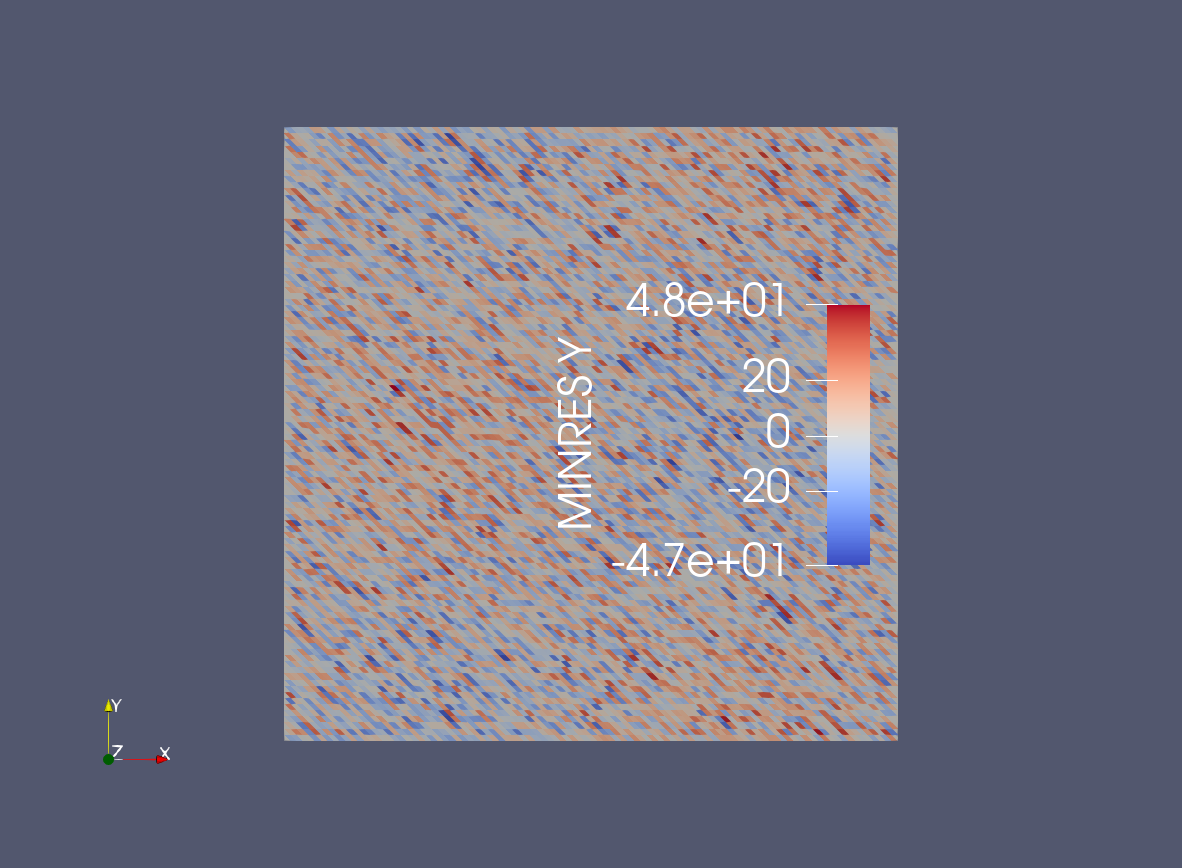}
	\includegraphics[trim={7.8cm 2.5cm 7.8cm 2.5cm},clip,width=.24\textwidth]{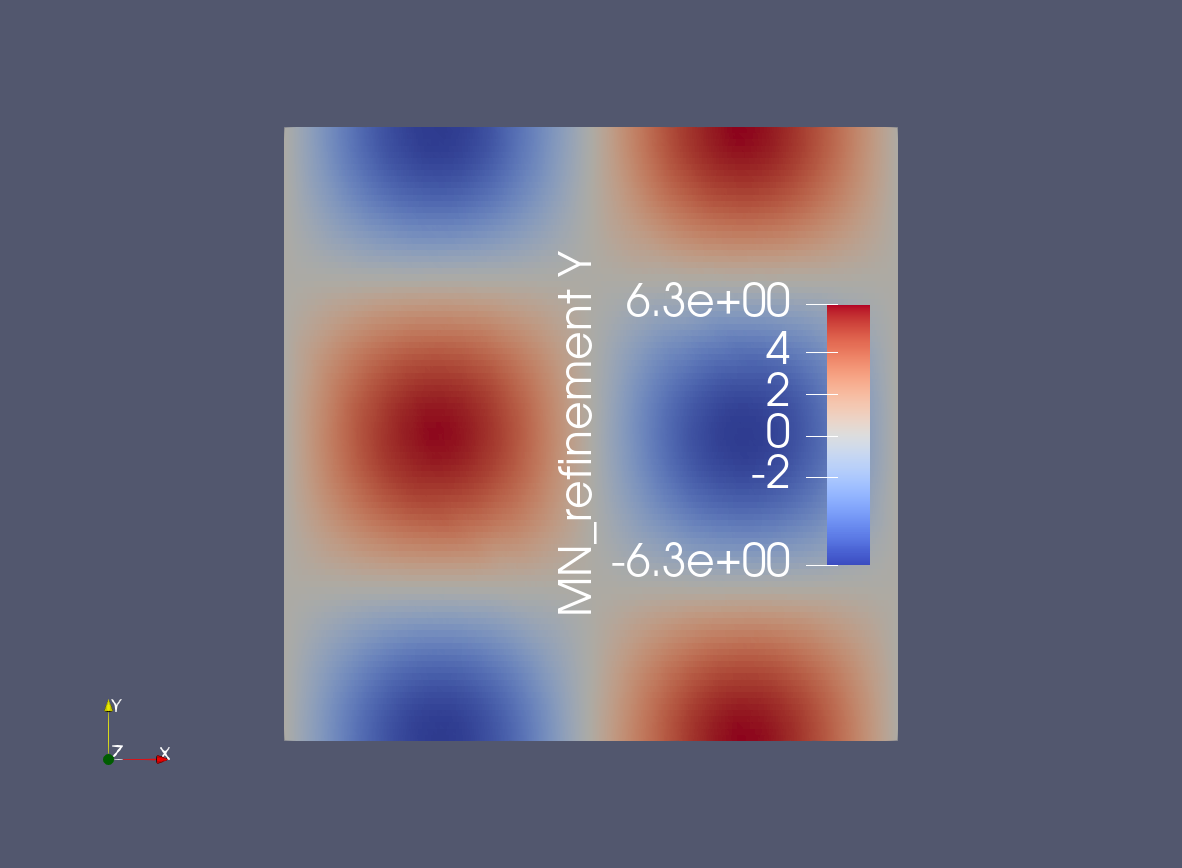}
	\caption{From top to bottom: the magnitude with orientations, the X-axis components, and the Y-axis components. From left to right: the right-hand-side $ \mathbf{f} $, the exact solution $ \mathbf{u}_{\textnormal{exact}} $, the final MINRES iterate without MN refinement, and the MN refinement result using \cref{prop:MN-refinement_t}. A recorded video demonstrating how the MN refinement results convergence to $ \mathbf{u}_{\textnormal{exact}} $ iteratively can be found at {\href{https://www.youtube.com/watch?v=ivRa-O9DCMI}{https://www.youtube.com/watch?v=ivRa-O9DCMI}}. \label{fig:curlcurl_bcs}}
\end{figure}

\subsection{Numerical solution of PDEs by MN refinement: \cref{prop:MN-refinement_t}}
\label{sec:exp:curlcurl}
We consider the following Maxwell problem \cite{mardal2011preconditioning}:
\begin{subequations}
	\label{eq:maxiwell}
	\begin{align}
		\textnormal{curl} \textnormal{ curl } \mathbf{u} &= \mathbf{f} \quad \textnormal{in } \Omega \label{eq:curlcurl} \\
		\mathbf{u} \times \mathbf{n} &= 0 \quad \textnormal{on } \partial \Omega, \label{eq:bcs}
	\end{align}
\end{subequations}
where $ \partial \Omega $ represents the boundary of $ \Omega \in \reals^2 $, and $ \mathbf{n} $ is the outwards-pointing unit normal on the boundary $ \partial \Omega $. We are interested in this problem because \cref{eq:maxiwell} contains infinitely many solutions for a general right-hand side $ \mathbf{f} $ \cite{mardal2011preconditioning}, while we aim to recover the pseudo-inverse solution among them.
We use Nédélec finite elements to discretize \cref{eq:maxiwell} with a regular triangular mesh with $ 100 $ elements along each edge, resulting in $ d = 30,200 $. The implementation\footnote{Codes are available at \href{https://github.com/yangliu-op/Curl-Curl}{https://github.com/yangliu-op/Curl-Curl}} of this experiment is using Firedrake \cite{FiredrakeUserManual} with a MINRES implementation from PETSc \cite{balay2019petsc}. 
Let $ \AA $ refer to the finite element matrix associated with the discretization of the $ \textnormal{curl} \textnormal{ curl} $ operator, that is $ \AA_{ij} = \int_{\Omega} \textnormal{ curl } \mathbf{\phi}_j \cdot \textnormal{ curl } \mathbf{\phi}_i$, where $ \{\mathbf{\phi}_j\}_{1}^{d} $, are basis functions spanning the Nédélec finite element space of the first kind \cite{nedelec1980mixed}. We construct our problem in $ \reals^2 $ with $ \mathbf{f} = \textnormal{curl} \textnormal{ curl } \mathbf{u}_{\textnormal{exact}} + \mathbf{e} $, where $ \mathbf{e} $ is the projection of uniformly distributed random noise $ \mathcal{U}(-1, 1) $ onto the null space of the operator $ \AA $. 
Let $ \rr $ denote the residual of \cref{eq:curlcurl}. We run MINRES until $ \AA \rr $ is sufficiently small. The termination condition is set to be $ \| \AA \rr \| \leq \eta \| \AA \uu \| $ with $ \eta = 0.1 $, which can be controlled via the \texttt{$ksp\_minres\_nutol$} option in PETSc. At termination, we apply \cref{prop:MN-refinement_t} to the final iterate. 
%Since we generated $ \ee $ in the null space, the MN refinement from \cref{thm:MINRES_dagger} recovers the exact solution, $ u_{\textnormal{exact}} $, when $ t \rightarrow g $. Nonetheless, we show that we can still recover a good approximation with early stopping. 
The results are gathered in \cref{fig:curlcurl_bcs}, which is generated by Paraview \cite{ahrens200536}. It is clear that the MINRES final iterate, without MN refinement, contains a great deal of noisy features. 
% We construct a positive definite preconditioner discretizing the $ H(\textnormal{curl} )$ Riesz map $ \MM = \AA + \BB $,  where $ \BB $ is the $ L^2 $ mass matrix, $ \BB_{ij} = \int_{\Omega} \phi_j \cdot \phi_i $. 
However, we can recover a good approximation of $ \mathbf{u}_{\textnormal{exact}} $ using \cref{prop:MN-refinement_t}.

\begin{figure}[!thbp]
	\centering
	\vspace*{-3mm}
	\hspace*{-7mm}
	\includegraphics[trim={0cm 0.5cm 0cm 0.5cm},clip,width=1\textwidth]{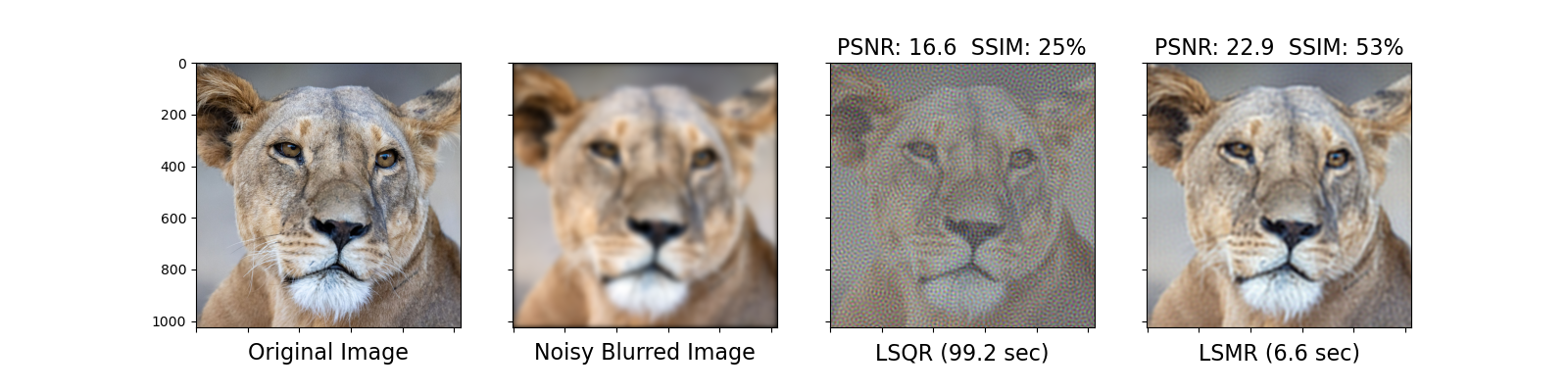}
	\hspace*{-7mm}
	\includegraphics[trim={0cm 0.5cm 0cm 0.5cm},clip,width=1\textwidth]{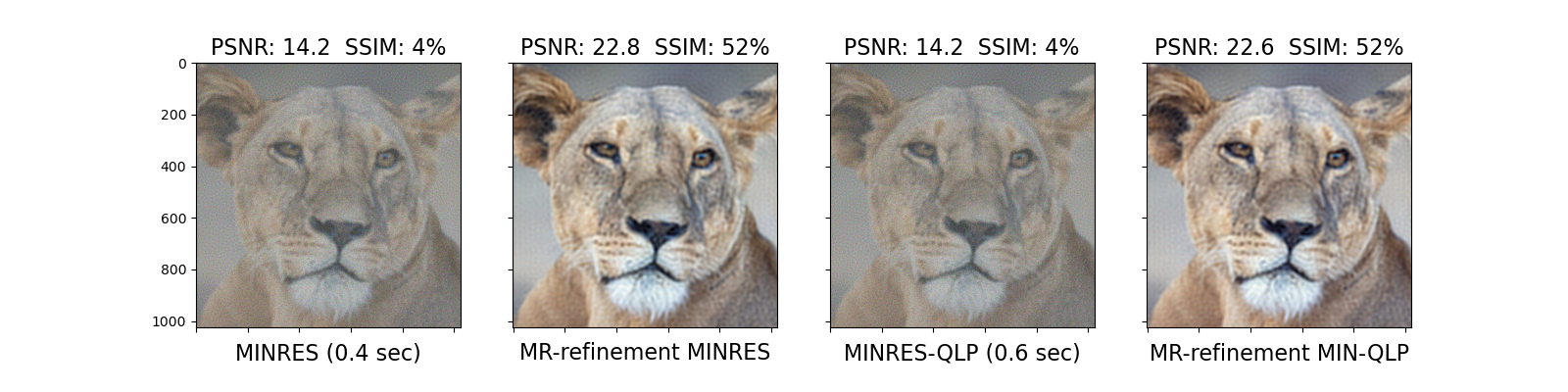}
	%   \vspace*{-3mm}
	\caption{From the top left to the bottom right: the original image, the noisy blurred image, the image deblurred with LSQR and with LSMR; the image obtained from \cref{alg:MINRES}, and the MN refinement image from \cref{prop:MN-refinement_t}, that obtained from MINRES-QLP and its corresponding MN refinement image from \cref{prop:MN-refinement_t}. The ``wall-clock'' time shows the total required seconds for each iterative algorithm to satisfy the inexactness criterion $\| \AA \rr \| \leq 10^{-5} \| \AA \bb \|$.} \label{fig:deblurring}
\end{figure}

\subsection{Image deblurring by MN refinement: \cref{thm:MINRES_dagger,thm:MN-refinement}}
\label{sec:exp_image}
We now turn our attention to a large-scale setting in the context of image deblurring.
We emphasize that our goal here is not to obtain a cutting-edge image deblurring technique that can achieve the state-of-the-art, but rather to explore the potentials of the MN refinement in a real world application.
We consider the following image blurring model:
\begin{align*}
	\BB = \sA(\XX) + \EE,
\end{align*}
where $ \BB \in \real^{n \times n} $ is a noisy and blurred version of an image $ \XX \in \real^{n \times n}$, $ \sA: \real^{n \times n} \to \real^{n \times n} $ represents the linear Gaussian smoothing/blurring operator \cite{shapiro2001computer} as a 2D convolutional operator, and $\EE$ is some noise matrix. Simple deblurring can be done by solving the least-squares problem \cref{eq:least_squares}, where $ \xx = \vect(\XX)\in \real^{n^{2}}  $, $ \bb = \vect(\BB)\in \real^{n^{2}} $, and $ \AA \in \real^{n^{2} \times n^{2}} $ is the real symmetric matrix representation of the Gaussian blur operator $\sA $. Though $ \AA $ has typically full rank, it often contains several small singular values, which make it almost numerically singular. In lieu of employing any particular explicit regularization, we terminate the iterations early before fitting to the ``noisy subspaces'' corresponding to the extremely small singular values. Our aim is to investigate the proposed MN refinements, with or without preconditioning, for deblurring images. Specifically, we investigate the application of \cref{eq:MN-refinement_t,eq:P_MN-refinement_t}. 

When $ n $ is large, it is not practical to store $ \AA $ explicitly. To remedy this, we generate a Toeplitz matrix $ \ZZ \in \real^{n \times n} $, from which we can implicitly define the blurring matrix as $ \AA = \ZZ \otimes \ZZ $, where $ \otimes $ denotes the Kronecker product. 
For our experiments\footnote{Codes are available at \href{https://github.com/yangliu-op/Image-Debluring}{https://github.com/yangliu-op/Image-Debluring}}, we consider a colored image with $ n = 1,024 $, which leads to a large-scaled deblurring problem with $ d = n^2 $ over a million. We solve for each channel separately and normalize our final result to regenerate the color image. The Gaussian blurring matrix $ \AA $ is generated with bandwidth $ 100 $ and standard deviation $ 10 $. For each color channel, the elements of the corresponding noise matrix are generated i.i.d$.$ from the standard normal distribution. Finally, in our experiments, since $ n\gg1 $, we apply \cref{eq:MN-refinement_t} with \cref{alg:sub_pMINRES} as an equivalent alternative to applying \cref{eq:P_MN-refinement_t} with \cref{alg:pMINRES}.

\begin{figure}[!thbp]
	\centering
	\vspace*{-3mm}
	\hspace*{-7mm}
	\includegraphics[width=1\textwidth]{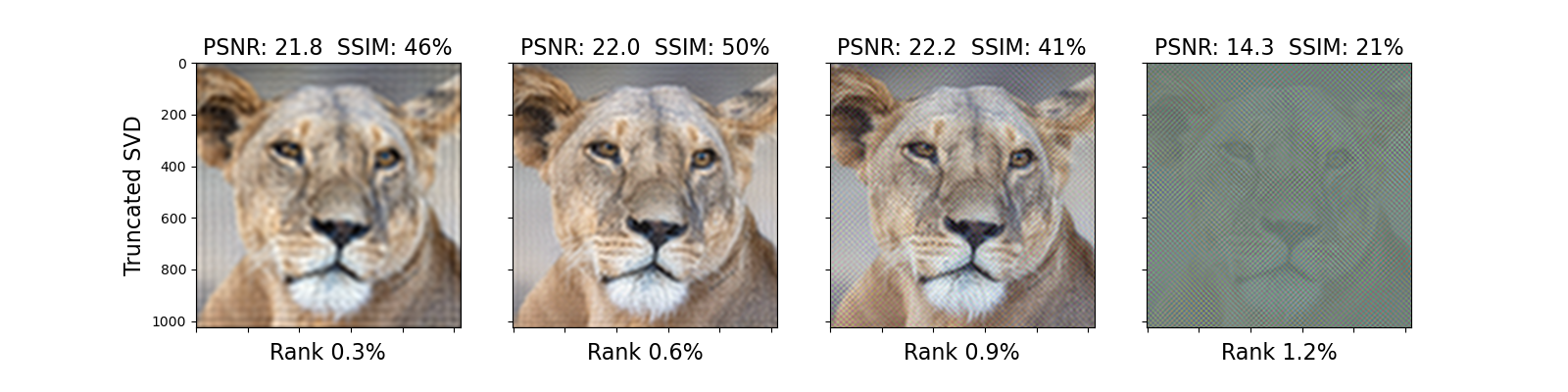}
	\vspace*{-3mm}
	\caption{Image deblurring with Truncated SVD. ``Rank $a$\%'', ``$a$\%'' represents the rank-ratio $ r^2/n^2 $, where $ r $ represents the rank being used in the Truncated SVD method. \label{fig:Truncated_deblurring}}
\end{figure}

\subsubsection*{Deblurring with \cref{alg:MINRES} and \cref{prop:MN-refinement_t}}
In our experiments, we first apply \cref{alg:MINRES}, and evaluate the quality of the deblurred image from the MINRES solution as well as the MN refinement one obtained from \cref{prop:MN-refinement_t}.
%
%According to \cref{rm:PMR}, this is equivalent to solve the preconditioned least-squared system \cref{eq:pMINRES} with MINRES solver in $ \real^{m^2 \times m^2} $ if $ \SS \in \real^{n^2 \times m^2} $ and then recover to $ \real^{n^2} $, which indeed can be solved quite efficiently with $ m \lll n $. Further more, if the residual is large, we can also apply the MN-refinement technique as in \cref{thm:MN-refinement}. 
As a benchmark, we consider comparing the solutions obtained from LSQR \cite{paige1982lsqr}, which has proven to enjoy favorable stability properties for solving large-scale least-squares problems \cite{wathen2022some}, as well as the standard truncated SVD. Both these methods, though not the state-of-the-art, have long been used for image deblurring problems \cite{hansen1998,hansen2006deblurring}. We also consider LSMR \cite{fong2011lsmr} as an alternative general-purpose solver, as well as MINRES-QLP, which is specifically designed for solving (almost) singular systems \cite{choi2011minres}. Note that every iteration of LSQR and LSMR requires two matrix-vector multiplications, whereas MINRES and MINRES-QLP involve only one such operation per iteration. 

This experiment was conducted using an Nvidia Geforce RTX 4090. In our implementations, we set the maximum iterations for all solvers to $ 20,000 $. Since the problem is highly ill-conditioned, instead of tracking the residual norm, we monitor the relative residual of the normal equation as $\| \AA \rr \| \leq 10^{-5} \| \AA \bb \|$. At termination of MINRES and MINRES-QLP, we apply the MN refinement procedure from \cref{prop:MN-refinement_t} to eliminate the contributions from the noisy subspaces corresponding to extremely small singular values of $ \AA $. 
To evaluate the quality of the reconstructed images, we use two common metrics, namely the peak-signal-to-noise ratio (PSNR), and the structural similarity index measure (SSIM) \cite{hore2010image,wang2004image}. 
%Both of the values can be found an the top of the deblurred images in \cref{fig:deblurring,fig:PMR_range_deblurring,fig:PMR_deblurring,fig:Truncated_deblurring}. 
Deblurred images with higher value in either metric are deemed to be of relatively better quality.

The results are depicted in \cref{fig:deblurring,fig:Truncated_deblurring}. From \cref{fig:deblurring}, it is clear that, LSQR not only returns a very poor solution, but it also takes the longest to terminate. Without MN refinement, the reconstructed image using LSMR is of higher quality compared with the ones directly obtained from \cref{alg:MINRES} and MINRES-QLP, albeit with a significantly more computation time. However, when the MN refinement is applied as outlined in \cref{prop:MN-refinement_t}, the image quality from \cref{alg:MINRES} and MINRES-QLP drastically improves, becoming highly competitive with that of LSMR, while requiring a fraction of the runtime. Although MINRES-QLP is guaranteed to recover the pseudo-inverse solution at $ t=g $, it is clear that in the case of early termination and without MN refinement, it fails to recover a good quality solution just like MINRES; see also \cref{rem:minres-minresqlp}. %Indeed, LSMR is equivalent to MINRES applying to the normal equation, which implies not only the two matrix-vector multiplication, but the convergence rate will be dependent on the squared condition number.
We also apply truncated SVD \cite{hansen2006deblurring} to give another benchmark. We can see from \cref{fig:Truncated_deblurring} that the performance of the truncated SVD is reasonable when the singular value threshold increases. However, it is also clear that the method is highly sensitive to the choice of this threshold.

\begin{figure}[!t]
	\centering
	\vspace*{-3mm}
	\hspace*{-7mm}
	\includegraphics[width=1\textwidth]{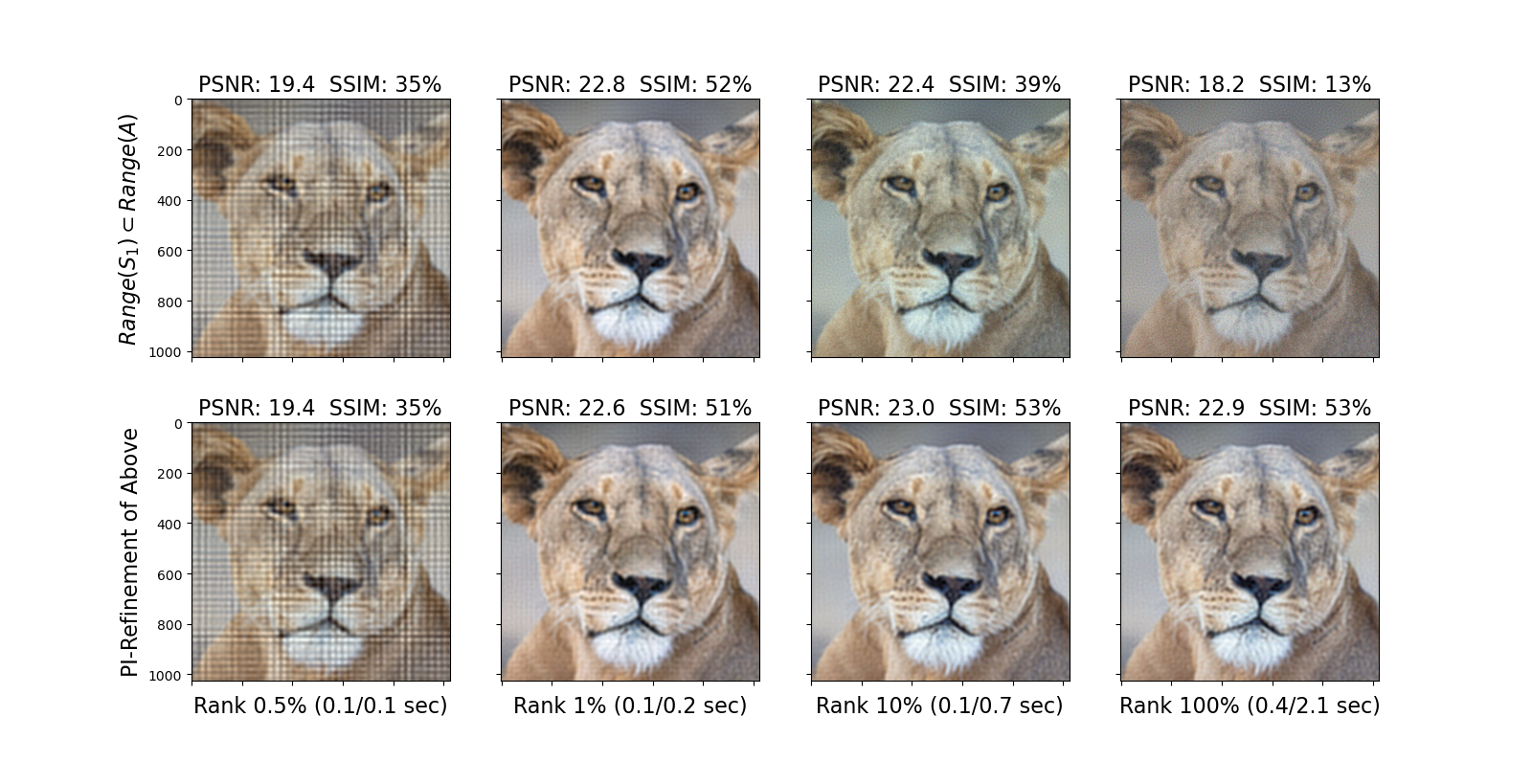}
	\vspace*{-3mm}
	\caption{Image deblurring using \cref{alg:sub_pMINRES} with $ \SS_1 \in \real^{n^2 \times r_1^2} $. In our naming convention, ``Rank $a$\% ($b$/$c$ sec)'', ``$a$\%'' represents the rank-ratio $ r_1^2/n^2 $, ``$c$'' is total wall-clock time of the full deblurring process (including the incomplete QR decomposition to compute $ \SS_1 $), and ``$b$'' shows the time taken by \cref{alg:sub_pMINRES} only. The figures in the bottom row are the MN refinement versions of ones in the top row. \label{fig:PMR_range_deblurring}}
\end{figure}

\begin{figure}[!t]
	\centering
	\vspace*{-3mm}
	\hspace*{-7mm}
	\includegraphics[width=1\textwidth]{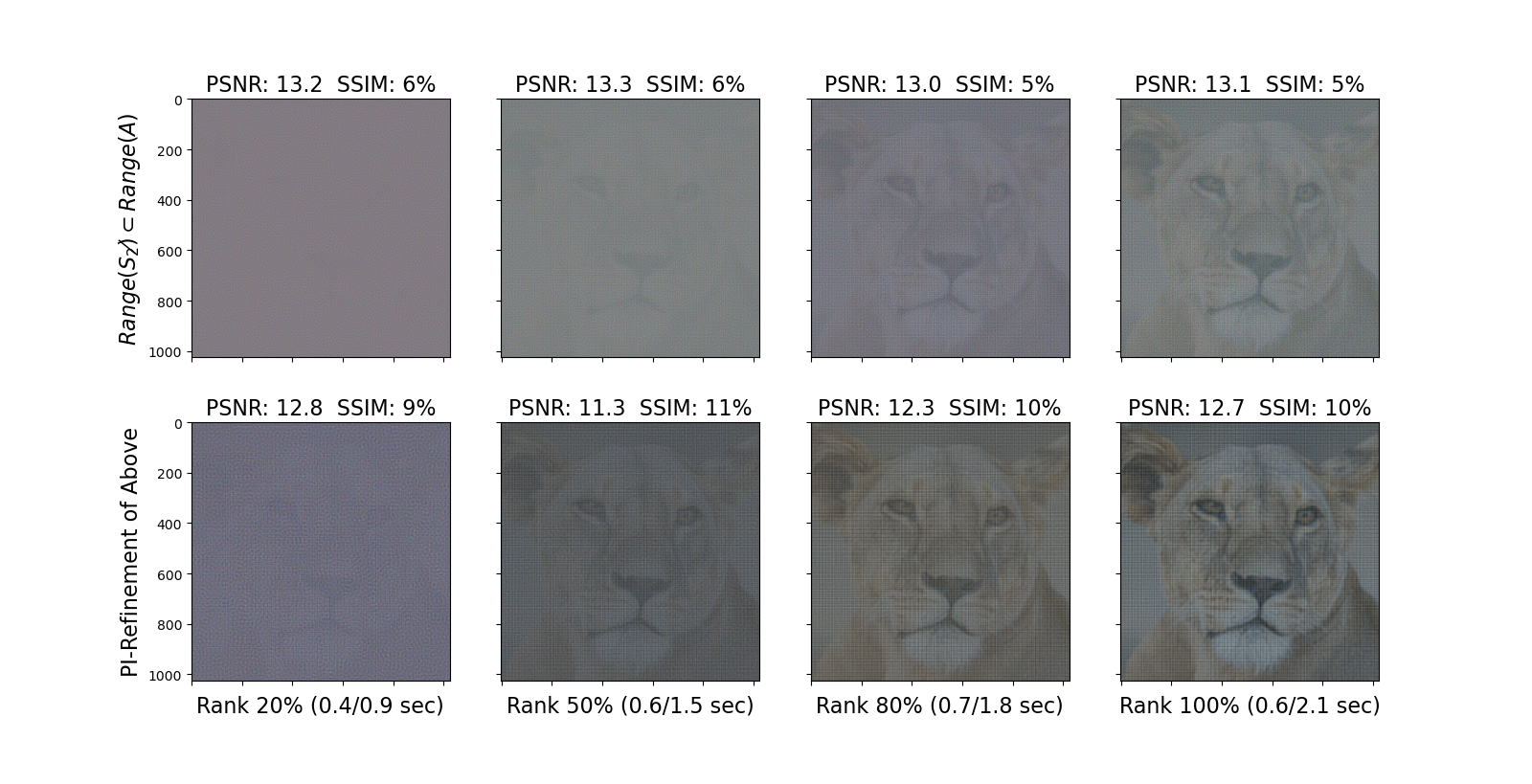}
	\vspace*{-3mm}
	\caption{Image deblurring using \cref{alg:sub_pMINRES} with $ \SS_2 \in \real^{n^2 \times r_2^2}  $. In our naming convention, ``Rank $a$\% ($b$/$c$ sec)'', ``$a$\%'' represents the rank-ratio $ r_2^2/n^2 $, ``$c$'' is total wall-clock time of the full deblurring process (including the incomplete QR decomposition to compute $ \SS_2 $), and ``$b$'' shows the time taken by \cref{alg:sub_pMINRES} only. The figures in the bottom row are the MN refinement versions of ones in the top row. \label{fig:PMR_deblurring}}
\end{figure}

\subsubsection*{Deblurring with \cref{alg:sub_pMINRES} and \cref{eq:P_MN-refinement_t}}
We now consider incorporating preconditioning withing our image deblurring problem. Recall that preconditioning has long been used in this context \cite{dell2017structure,chen2016preconditioning,hansen2006deblurring,bianchi2019structure}.  
To fit our settings here, we consider appropriate singular preconditioners $ \MM \succeq \zero $,  and explore the deblurring quality of the MINRES solution of \cref{alg:sub_pMINRES}, as well as the MN refinement solution by applying \cref{prop:MN-refinement_t} to the preconditioned system \cref{eq:preq_MINRES}.

We construct two different types of sub-preconditioners. For $i =1,2 $, we let $ \SS_i = \CC_i \otimes \CC_i \in \real^{n^2 \times r_i^2}, \; i=1,2 $, where $ \CC_i = \QQ_i \SIGMA_i \in \real^{n \times r_i}, \; i =1 ,2 $ and the factors $ \SIGMA_i $, $ \QQ_i $ and the corresponding ranks $ r_i $ are generated as follows:
\begin{enumerate}
	\item To construct $ \QQ_{i} $, we first generate a $ n \times n $ matrix $ \hCC $ whose elements are drawn independently from the standard normal distribution. We then choose $ \QQ_1 \in \real^{n \times r_1} $ and $ \QQ_2 \in \real^{n \times r_2} $ via the incomplete QR decomposition of $ \ZZ \hCC = \QQ_1 \RR_1 $ and $ \hCC = \QQ_2 \RR_2 $ with the corresponding ranks $ r_1 $ and $ r_2 $.
	\item We construct the diagonal matrices $ \SIGMA_i \in \real^{r_i \times r_i}, \; i =1,2 $ with positive diagonal entries such that the smallest and the largest diagonal elements are, respectively $ 1 $ and $ 2 $, and other elements are linearly spaced values in-between.
\end{enumerate}

By construction, $ \rg{\SS_1} \subseteq \rg{\AA} $, while this is not true for $ \SS_2 $. 

Considering $ \SS_2 $ as an example, we use the following Kronecker product properties to avoid explicit storage of the matrix $ \AA $: 
\begin{align*}
	\tAA \vect(\tXX) &= \SST_2 \AA \SS_2 \vect(\tXX) = \left(\CC_2^{\T} \otimes \CC_2^{\T}\right) \left(\ZZ \otimes \ZZ\right) \left(\CC_2 \otimes \CC_2\right) \vect(\tXX) \\
	&= \left(\left(\CC_2^{\T} \ZZ \CC_2\right) \otimes \left(\CC_2^{\T} \ZZ \CC_2 \right)\right) \vect(\tXX) = \vect(\CC_2^{\T} \ZZ \CC_2 \tXX \CC_2^{\T} \ZZ \CC_2), \\
	\tbb &= \SST_2 \vect(\BB) = \left(\CC_2^{\T} \otimes \CC_2^{\T}\right) \vect(\BB) = \vect(\CC_2^{\T} \BB \CC_2) ,
\end{align*}
and $\xx = \vect(\XX) = \SS_2 \vect(\tXX) = \left(\CC_2 \otimes \CC_2\right) \vect(\tXX) = \vect(\CC_2 \tXX \CC_2^{\T})$.

\cref{fig:PMR_range_deblurring} depicts the reconstruction quality using the sub-preconditioner $ \SS_1 $, where for rank-ratios larger than $ 1\% $, we obtain reasonable quality results. The effect of the MN refinement from \cref{thm:MN-refinement} is also clearly visible, in particular for higher values of the rank-ratio. Notably, the running time of \cref{alg:sub_pMINRES} with $ \SS_1 $ is a small portion of the total deblurring time, including the QR decomposition to construct $ \SS_1 $. This is a direct consequence of the dimensionality reduction nature of \cref{alg:sub_pMINRES}, which employs the sub-preconditioner $ \SS_1 $, as opposed to the full preconditioner $ \MM $. Also, from \cref{fig:PMR_range_deblurring} it is clear that \cref{alg:sub_pMINRES} exhibits a great degree of stability with respect to the different choices of the rank for the sub-preconditioner $ \SS_1 $, as compared with the truncated SVD in \cref{fig:Truncated_deblurring}.
In \cref{fig:PMR_deblurring}, however, we see poor quality reconstructions. This is entirely due to the fact that, by construction, $ \rg{\SS_2} $ does not align well with the range space of $ \AA $, so that $ \SS_2 $ constitutes a poor sub-preconditioner. Hence, when the underlying matrix is (nearly) singular and $ \bb $ substantially aligns with the eigen-space corresponding to small eigenvalues of $\AA$, one should apply preconditioners that can counteract the negative effects of such alignments. 
This observation underlines the fact that, unlike the typical settings where the preconditioning matrix is assumed to have full-rank, the situation can be drastically different with singular preconditioners and one can face a multitude of challenges. In this respect, although singular preconditioners can be regarded as a way to reduce the dimensionality of the problem, more studies need to be done to explore various effects and properties of such preconditioners in various contexts.

\begin{remark}
	\label{rem:minres-minresqlp}
	From \cref{fig:deblurring} and \cref{table:average_iterations}, we see remarkable similarities between the final results of MINRES-QLP and MINRES. This warrants a further look into the dynamic of MINRES-QLP across all its iterations, as compared with MINRES. For this, we consider using both MINRES and MINRES-QLP for solving the same synthetic problem from \cref{sec:exp:MN-refinement} but with a symmetric matrix $ \AA $. From \cref{fig:pseudo_QLP}, we see that the iterates from both these methods are almost identical throughout the life of the algorithms, except at the very last iteration, when $ t = g = 16 $. This indicates that all the  extra computations within MINRES-QLP at every iteration might not in fact contribute to an improved solution and it is primarily the last iteration that brings about the intended effect of obtaining the pseudo-inverse solution.
\end{remark}

\begin{table}[!htbp]
	\caption{Average number of iterations before termination for three channels using the listed solvers for the experiments of \cref{fig:deblurring,fig:PMR_deblurring,fig:PMR_range_deblurring}. ``MIN-QLP'' refers to MINRES-QLP and the ``percentage values'' refer to the corresponding rank-ratio defined in \cref{fig:PMR_deblurring,fig:PMR_range_deblurring}.}
	\centering
	\scalebox{.8}{
		\begin{tabular}{ |c|c|c|c|c|c|c|c|c|c|c|c| }
			%               \hline \multicolumn{3}{|c|}{Summary of Results} \\
			%               \hline \multicolumn{3}{|c|}{} \\[-2.5ex]
			\hline
			\multicolumn{4}{|c|}{\centering \cref{fig:deblurring}}
			& \multicolumn{4}{c|}{\centering \cref{fig:PMR_range_deblurring}} & \multicolumn{4}{c|}{\centering \cref{fig:PMR_deblurring}} \\ 
			\hline 
			LSQR & LSMR & MINRES & MIN-QLP & 0.5\% & 1\% & 10\% & 100\% & 20\% & 50\% & 80\% & 100\% \\ [0.5ex]
			\hline
			3931.7 & 242.7 & 30 & 30 & 136.0 & 71.3 & 52.0 & 43.7 & 40.7 & 38.7 & 30.7 & 30 \\ [0.5ex]
			\hline
	\end{tabular}}
	\label{table:average_iterations}
\end{table}

\begin{figure}[!thbp]
	\centering
	\subfigure[Plain Iterates]{
		\includegraphics[scale = 0.38]{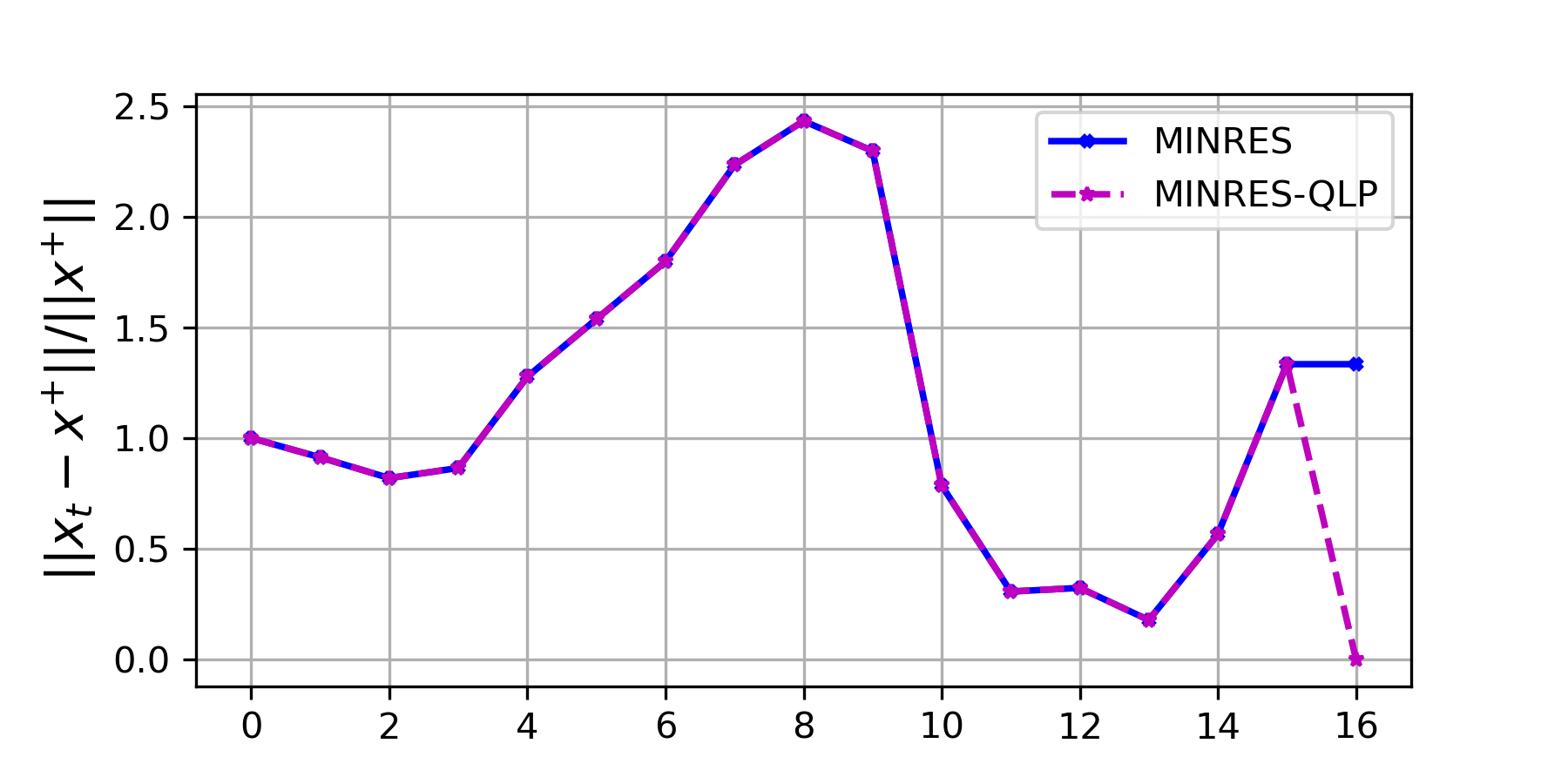}}
	\subfigure[MN refinement Iterates]{
		\includegraphics[scale = 0.38]{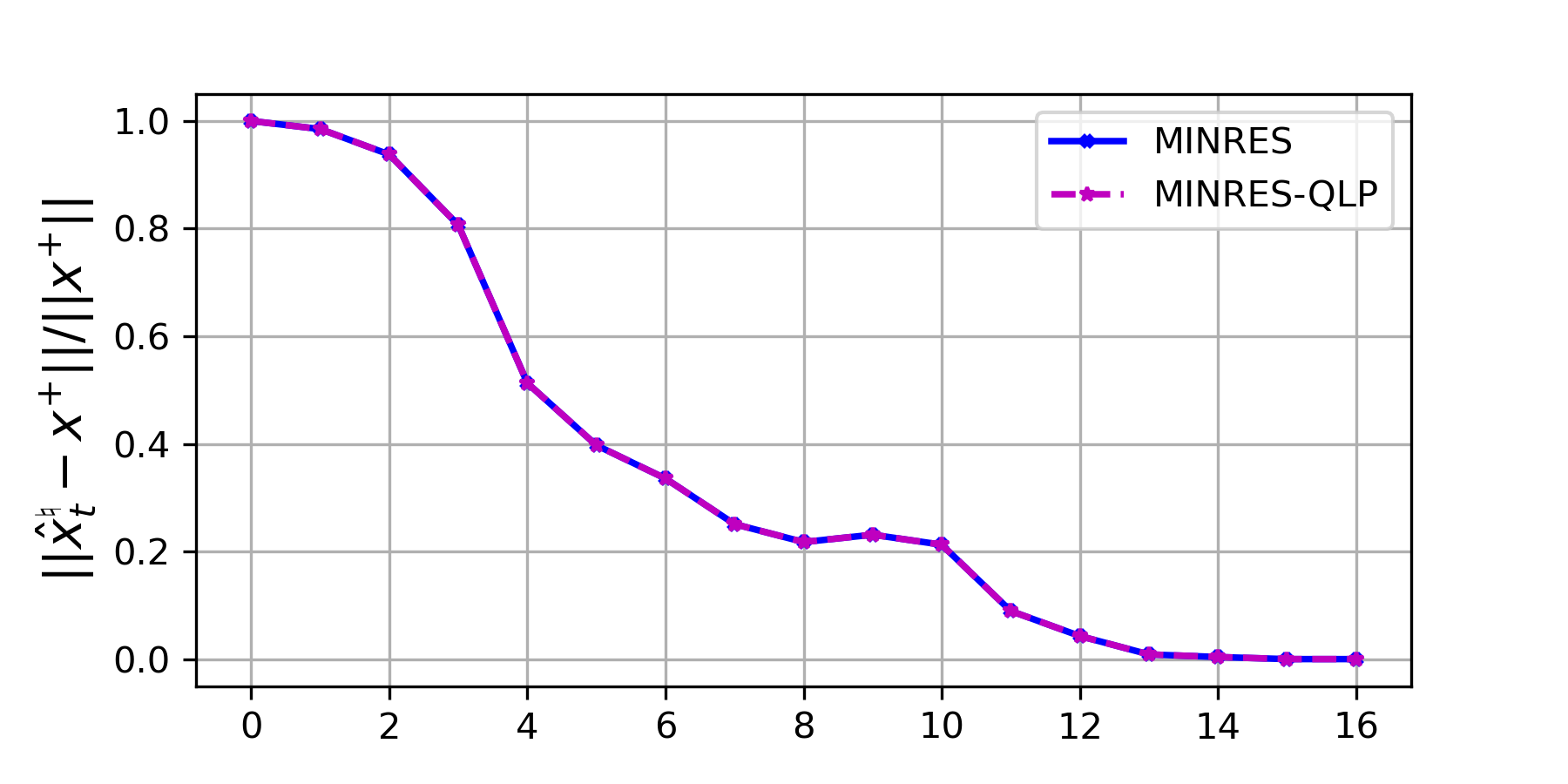}}
	\caption{Comparison between MINRES and MINRES-QLP, with and without MN refinement \cref{prop:MN-refinement_t}. Relative error of the plain (a) and the MN refinement (b) iterates of \cref{alg:MINRES} (sold blue) and MINRES-QLP (dashed magenta) applied to \cref{eq:least_squares} with Symmetric matrix. The relative error is computed with respect to the pseudo-inverse solution $ \xxd = \AAd \bb $. The $x$-axis denotes the iteration counter. It is clear that dynamics of MINRES and MINRES-QLP are almost identical except the very last iteration. However, the MN refinement iterates using \cref{prop:MN-refinement_t} look almost indistinguishable. \label{fig:pseudo_QLP}}
\end{figure}

\section{Conclusions}
\label{sec:conclusion}
We considered the minimum residual method (MINRES) for solving Hermitian linear least-squares problems involving singular matrices. It is well-known that, unless the right hand-side vector entirely lies in the range-space of the matrix, MINRES does not necessarily converge to the pseudo-inverse solution of the problem. We propose a novel MN refinement, which can readily recover such a pseudo-inverse solution from the final iterate of MINRES. We provide similar results in the context of preconditioned MINRES using positive semi-definite and singular preconditioners. We also showed that, when such singular preconditioners are formed from sub-preconditioners, we can equivalently reformulate the original preconditioned MINRES algorithm in lower dimensions, which can be solved more efficiently. We provide several numerical examples to further shed light on our theoretical results and to explore their use-cases in some real world application. 

\section*{Acknowledgments}
\noindent We express our sincere gratitude to Pablo Brubeck Martinez for his insightful and valuable discussions on the application of MINRES in solving Maxwell problems. We would like to thank the Associate Editor and the two anonymous reviewers for their detailed and constructive comments, which have helped greatly to improve the quality and presentation of the manuscript. Finally, we extend our deep gratitude to Michael Saunders, whose invaluable suggestions were instrumental in further refinements of this paper. 

%\section*{Disclosure statement}
%\noindent No potential conflict of interest was reported by the authors.

        \appendix
	\section{MINRES for complex-symmetric systems}
\label{app:CSMINRES}
\subsubsection*{Saunders process}
Recall the Saunders process \cite{choi2013minimal,saunders1988two} and the Saunders subspace \cref{eq:saunders}. 
For complex-symmetric systems, after $ t $ iterations of the Saunders process, and in the absence of round-off errors, the Saunders vectors form an orthogonal matrix $ \VV_{t+1} = [ \vv_1, \vv_2, \dots, \vv_{t+1} ] \in \comp^{\dn \times (t+1)} $ whose columns span $ \mathcal{S}_{t+1}(\AA, \bb) $ and satisfy the relation $ \AA \bVV_{t} = \VV_{t+1} \hTT_{t} $. Here, $ \hTT_{t} $ is a tridiagonal upper-Hessenberg matrix of the form
\begin{align}
	\label{eq:tridiagonal_T}
	\hTT_{t} &= 
	\begin{bmatrix}
		\alpha_1 & \beta_2 & & & \\
		\beta_2 & \alpha_2 & \beta_3 & & \\
		& \beta_3 & \alpha_3 & \ddots & \\
		& & \ddots & \ddots & \beta_{t} \\
		& & & \beta_{t} & \alpha_{t} \\
		\hdashline
		& & & & \beta_{t+1} \\
	\end{bmatrix}
	\triangleq
	\begin{bmatrix}
		\TT_{t} \\
		\beta_{t+1}\ee^\T_{t} 
	\end{bmatrix} \in \comp^{(t+1) \times t},
\end{align}
where $ \TTt^{\T} = \TTt = \VVtH \AA \bVVt = \VVtH \UU \SIGMA \UUT \bVVt \in \comp^{t \times t} $ is complex-symmetric. 
Subsequently, we get the three-term recursion
\begin{align}
	\label{eq:lanczos_CS}
	\AA \bvv_{t} = \beta_{t} \vv_{t-1} + \alpha_{t} \vv_{t} + \beta_{t+1} \vv_{t+1}, \quad t \geq 1,
\end{align}
where $ \alpha_{t} $ can be computed as $ \alpha_{t} = \dotprod{\vv_{t}, \AA \bvv_{t}} $, and $ \beta_{t+1} > 0 $ is chosen by enforcing that $ \vv_{t+1} $ is a unit vector for all $ 1 \leq t \leq g - 1 $, i.e., $ \beta_{t+1} = \vnorm{\beta_{t+1} \vv_{t+1}} $. Letting $ \xx_{t} = \bVV_{t} \yy_{t} $ for some $ \yy_{t} \in \comp^{t} $, it follows that the residual $ \rr_{t} $ can be written as 
\begin{align*}
	\rr_{t} = \bb - \AA \xx_{t} = \bb - \AA \bVV_{t} \yy_{t} = \bb - \VV_{t+1} \hTT_{t} \yy_{t} = \VV_{t+1}(\beta_1 \ee_1 - \hTT_{t} \yy_{t}),
\end{align*}
where $ \beta_{1} = \| \bb \| $, which yields the sub-problems of complex-symmetric MINRES,
\begin{align}
	\label{eq:MINRES_sub}
	\min_{\yy_{t} \in \comp^{t}} \vnorm{\beta_1 \ee_1 - \hTT_{t} \yy_{t}}^2.
\end{align}

\begin{remark}
    Let us note that though we consider the Saunders process only in the context of complex-symmetric systems, it is more broadly adaptable to general square matrices $ \AA \in \comp^{\dn \times \dn} $; see \cite{saunders1988two}.
\end{remark}

\subsubsection*{QR decomposition}
In contrast to \cite{choi2013minimal}, to perform the QR decomposition of $\hTT_{t}$, we employ the complex-symmetric Householder reflector \cite[Exercise 10.4]{trefethen1997numerical}. This is a judicious choice to maintain some of the classical properties of MINRES in the Hermitian case such as $ \| \rrt \| = \phi_{t} > 0 $, as well as the singular value approximations using $ \gamma_{t}^{[2]} > 0 $. Doing this, we obtain $ \QQt \in \comp^{(t+1) \times (t+1)} $ and $ \hRRt \in \comp^{(t+1) \times t} $ via
\begin{subequations}
	\label{eq:T_QR_decomp}
	\begin{align}
		\QQ_{t} \hTT_{t} = \hRR_{t} \triangleq 
		\begin{bmatrix}
			\RR_{t}\\
			\zero^{\T}
		\end{bmatrix}, \quad 
		\RR_{t} =
		\begin{bmatrix}
			\gamma_1^{[2]} & \delta_2^{[2]} & \epsilon_3 & \\
			&\gamma_2^{[2]} & \delta_3^{[2]} & \ddots \\
			& & \ddots & \ddots & \epsilon_{t} \\
			& & & \gamma_{t-1}^{[2]}  & \delta_{t}^{[2]} \\
			& & & & \gamma_{t}^{[2]} \\
		\end{bmatrix}&. \label{eq:block_R} \\
		\QQ_{t} \triangleq \prod_{i=1}^{t}\QQ_{i,i+1}, \quad \QQ_{i,i+1} \triangleq 
		\begin{bmatrix}
			\eye_{i-1} & & & \\
			& \bc_{i} & s_{i} &\\
			& s_{i} & -c_{i} & \\
			& & & \eye_{t-i}
		\end{bmatrix}&, \label{eq:block_Q}
	\end{align}
\end{subequations}
where for $ 1 \leq i \leq t $, $ \gamma_{i}^{[2]} \geq 0$, $ \delta_{i}^{[2]} \in \comp $, $ \epsilon_{i} \geq 0$, $ c_{i} \in \comp $, $ 0 \leq s_{i} \leq 1 $, $ \abs{c_{i}}^2 + s_{i}^2 = 1 $, 
\begin{align*}
	%	\label{eq:c_s_r2}
	c_{i} = \frac{\gamma_{i}}{\gamma_{i}^{[2]}}, \quad s_{i} = \frac{\beta_{i+1}}{\gamma_{i}^{[2]}}, \quad \gamma_{i}^{[2]} = \bc_{i} \gamma_{i} + s_{i} \beta_{i+1} = \sqrt{\abs{\gamma_{i}}^2 + \beta_{i+1}^2},
\end{align*}
and $ \gamma_{i} $ is defined in \cref{alg:MINRES}. In fact, the same series of transformations are also simultaneously applied to $\beta_{1} \ee_1$, i.e., 
\begin{align*}
	\QQ_{t}\beta_1\ee_1 &= \beta_1 \QQ_{k,k+1} \dots \QQ_{2,3}
	\begin{bmatrix}
		\bc_1\\
		s_1\\
		\bm{0}_{k-1}
	\end{bmatrix} \\
	&= \beta_1 \QQ_{k,k+1} \dots \QQ_{3,4}
	\begin{bmatrix}
		\bc_1\\
		s_1 \bc_2\\
		s_1 s_2\\
		\bm{0}_{k-2}
	\end{bmatrix} = \beta_1
	\begin{bmatrix}
		\bc_1\\
		s_1 \bc_2\\
		\vdots\\
		s_1 s_2\dots s_{t-1} \bc_{t}\\
		s_1 s_2\dots s_{t-1} s_{t}
	\end{bmatrix} \triangleq
	\begin{bmatrix}
		\tau_1\\
		\tau_2\\
		\vdots\\
		\tau_{t}\\
		\phi_{t}
	\end{bmatrix} \triangleq \begin{bmatrix}
		\ttt_{t} \\
		\phi_{t}
	\end{bmatrix}.
\end{align*}

\begin{remark}
	\label{rm:real}
	By \cite[Theorem 6.2]{saad2011numerical}, the Lanczos process applied to a Hermitian matrix gives rise to a real symmetric tridiagonal matrix $ \TTt = \VVtH \AA \VVt \in \symm^{t \times t} $. Therefore, the QR decomposition for the preconditioned MINRES with such Hermitian matrices involve real components, e.g., $ c_t = \bc_{t} \in \reals $. This is in contrast to the complex-symmetric case.
\end{remark}

Having introduced these different terms, we can solve \cref{eq:MINRES_sub} by noting that
\begin{align*}
	\min_{\yy_{t} \in \comp^{t}}  \vnorm{\rr_{t}} = \min_{\yy_{t} \in \comp^{t}} \vnorm{\beta_1\ee_1 - \hTT_{t}\yy_{t}} &= \min_{\yy_{t} \in \comp^{t}} \vnorm{\QQ_{t}\beta_1\ee_1 - \QQ_{t} \hTT_{t} \yy_{t}} \\
	& = \min_{\yy_{t} \in \comp^{t}} \left\|{\begin{bmatrix}
			\ttt_{t}\\
			\phi_{t}
		\end{bmatrix} -
		\begin{bmatrix}
			\RR_{t} \\
			\zero^{\T}
		\end{bmatrix}\yy_{t}}\right\|.
\end{align*}
Note that this implies that $ \| \rr_{t} \| = \phi_{t} $. We also trivially have $ \beta_{1} = \phi_{0} = \vnorm{\bb} $.

\subsubsection*{Updates}
Let $ t < g $ and define $ \DD_{t} \in \comp^{m \times t} $ from solving the lower triangular system $ \RR_{t}^{\T} \DD_{t}^{\T} = \VV_{t}^{\H} $, where $ \RR_{t} $ is as in \cref{eq:T_QR_decomp}. Now, letting $ \VV_{t} = [ \VV_{t-1}, \vv_{t}] $, and using the fact that $ \RRt $ is upper-triangular, we get the recursion $ \DD_{t} = [ \DD_{t-1}, \dd_{t}] $ for some vector $ \dd_{t} $. As a result, using $ \RR_{t} \yy_{t} = \ttt_{t} $, one can update the iterate as 
\begin{align*}
	%	\label{eq:updates_x_d}
	\xx_{t} = \bVV_{t} \yy_{t} = \DD_{t} \RR_{t} \yy_{t} = \DD_{t} \ttt_{t} = \begin{bmatrix}
		\DD_{t-1} & \dd_{t}
	\end{bmatrix} \begin{bmatrix}
		\ttt_{t-1} \\
		\tau_{t}
	\end{bmatrix} = \xx_{t-1} + \tau_{t} \dd_{t},
\end{align*}
where we let $ \xx_{0} = \zero $. Furthermore, applying $\bVV_{t} = \DD_{t} \RR_{t}$ and the upper-triangular form of $\RR_{t}$ in \eqref{eq:T_QR_decomp}, we obtain the following relationship for $ \vv_{t} $:
\begin{align*}
	%	\label{eq:updates_v_d_CS}
	\bvv_{t} = \epsilon_{t} \dd_{t-2} + \delta^{[2]}_{t} \dd_{t-1} + \gamma^{[2]}_{t} \dd_{t}.
\end{align*}

\section{Preconditioned MINRES for complex-symmetric syst- ems} \label{app:pMINRES_CS} Motivated by \cref{sec:pMINRES_H}, with complex-symmetric $ \AA \in \symm^{m \times m} $, and $ \MM $ given as $ \MM  = \SS \SSH $ for some sub-preconditioner matrix $ \SS \in \comp^{\dn \times m} $, we again see that 
\begin{align*}
	\xxt = \argmin_{\xx \in \St{\bMM\AA,\bMM\bb}} \vnorm{\bb - \AA \xx}_{\bMM}^2 = \argmin_{\xx \in \bSS \St{\SST\AA\SS,\SST\bb}} \vnorm{\SST\left(\bb - \AA \xx\right)}^{2}.
\end{align*}
This allows us to consider \cref{eq:preq_MINRES} with
\begin{align}
	\label{eq:def_tA_tb_tr_hr_CS}
	\tAA \triangleq \SST \AA \SS \in \symm^{m \times m}, \quad \text{and} \quad \tbb \triangleq \SST \bb \in \comp^{m}.
\end{align}
Using \cref{eq:def_x_r}, for the residual of the system in \cref{eq:pMINRES}, $\trrt = \tbb - \tAA \txxt$, we have 
\begin{align}
	\label{eq:trrt_rrt_CS}
	\trrt = \SST \rrt \in \comp^{m}, \quad \text{where} \quad  \rrt = \bb - \AA \xx. 
\end{align}

Similar to \cref{sec:pMINRES_H}, we now construct the preconditioned MINRES algorithm for complex-symmetric systems. For this, note that $\tAA$, defined in \cref{eq:def_tA_tb_tr_hr_CS}, is also itself complex-symmetric. Hence, following \cref{sec:CSMINRES}, the Saunders process \cref{eq:lanczos_CS} yields
\begin{align}
	\label{eq:saunders_process}
	\betatn \tvv_{t+1} = \tAA \cj{\tvvt} - \alphat \tvvt - \betat \tvv_{t-1}.
\end{align}
Denoting 
\begin{align}
	\label{eq:def_z_w_d_CS}
	\betat \tvvt = \SST \zzt, \quad \wwt = \MM \bzzt = \betat \SS \cj{\tvvt},
\end{align}
%By \cref{lemma:Kry_spans_z_w}, we are able to construct $ \zzt \in \comp^{\dn} $ and $ \wwt \in \comp^{\dn} $. 
and using \cref{eq:def_z_w_d_CS,eq:Md_STS}, we obtain
\begin{align*}   
	\SST \zz_{t+1} = \frac{1}{\betat} \SST \AA \MM \bzzt - \frac{\alphat}{\betat} \SST \zzt - \frac{\betat}{\beta_{t-1}} \SST \zz_{t-1}.
\end{align*}
This relation suggests that to guarantee \cref{eq:saunders_process}, we can define the following three-term recurrence on $ \zzt $ 
\begin{align}
	\label{eq:updates_z_w_CS}
	\zztn = \frac{1}{\betat} \AA \MM \bzzt - \frac{\alphat}{\betat} \zzt - \frac{\betat}{\beta_{t-1}} \zz_{t-1} = \frac{1}{\betat} \AA \wwt - \frac{\alphat}{\betat} \zzt - \frac{\betat}{\beta_{t-1}} \zz_{t-1}.
\end{align}
where $ \beta_0 = \beta_1 = \| \tbb \| $.

The main result of this section relies on a few technical lemmas, which we now present. \cref{lemma:Kry_tA_tb_CS}, similar to \cref{lemma:Kry_tA_tb}, gives an alternative characterization of the Krylov sub-space in \cref{eq:pMINRES}.
\begin{lemma}
	\label{lemma:Kry_tA_tb_CS}
	For any $1 \leq t \leq \tg$, we have $\sds{t}{\tAA, \tbb} = \SST \sds{t}{\AA \MM, \bb}$.
\end{lemma}
\begin{proof}
	By \cref{eq:def_tA_tb_tr_hr_CS,eq:Md_STS,eq:Md_KKH} and the definition of the Saunders subspace in \cref{eq:saunders}, we have
	\begin{align*}
		\sds{t}{\tAA, \tbb} &= \kryl{t_1}{\tAA \cj{\tAA}, \tbb} \oplus \kryl{t_2}{\tAA \cj{\tAA}, \tAA \cj{\tbb}},
	\end{align*}
	where $ t_1 + t_2 = t $, $ 0 \leq t_1 - t_2 \leq 1 $ and
	\begin{align*}
		\kryl{t_1}{\tAA \cj{\tAA}, \tbb} &= \Span\left\{\SST \bb, \SST \AA \MM \bAA \bMM \bb, \ldots, \SST \left[\AA \MM \bAA \bMM\right]^{t_1-1} \bb\right\} \\
		&= \SST \kryl{t_1}{\AA \MM \bAA \bMM, \tbb}.
	\end{align*}
	and
	\begin{align*}
		\kryl{t_2}{\tAA \cj{\tAA}, \tAA \cj{\tbb}} &= \Span\left\{\SST \AA \MM \bbb, \ldots, \SST \left[\AA \MM \bAA \bMM\right]^{t_2-1} \AA \MM \bbb\right\} \\
		&= \SST \kryl{t_2}{\AA \MM \bAA \bMM, \AA \MM \bbb}.
	\end{align*}
	The result follows by combining the latter expressions.
\end{proof}

Using \cref{lemma:Kry_tA_tb_CS}, we can now gain insight into the properties of the vectors $ \zz_{i} $ and $ \ww_{i} $. The proofs of \cref{lemma:Kry_spans_z_w_CS,lemma:alpha_beta_CS,lemma:residuals_CS,lemma:Kry_x_hr_CS} are similar to that of \cref{lemma:Kry_spans_z_w,lemma:alpha_beta,lemma:residuals,lemma:Kry_x_hr}, respectively, and hence is omitted.
\begin{lemma}
	\label{lemma:Kry_spans_z_w_CS}
	For any $1\leq t \leq \tg$, we have
	\begin{align*}
		\Span \{\ww_{1}, \ww_{2}, \ldots, \ww_{t}\} &= \MM \sds{t}{\bAA \bMM, \bbb},\\
		\bPP \PPT \Span \{\zz_{1}, \zz_{2}, \ldots, \zz_{t}\} &= \bPP \PPT \sds{t}{\AA \MM, \bb}, 
	\end{align*}
	where $ \PP $ is as in \cref{eq:def_Md_M_Sd_S}. Furthermore, $ \ww_{\tg+1} = \zero $ and $ \zz_{\tg+1} \in \Null(\MM) $.
\end{lemma}
% \begin{proof}
%     Again, recall that for all $1 \leq t \leq \tg $, we have $ \beta_{t} > 0 $. 
% 	Now, similar to the proof of \cref{lemma:Kry_spans_z_w}, by \cref{eq:def_z_w_d_CS,lemma:Kry_tA_tb_CS,eq:Md_STS,eq:saunders}, it holds that
% 	\begin{align*}
% 		\Span \{\ww_{1}, \ww_{2}, \ldots, \ww_{t}\} &= \SS \times \Span \{\cj{\tvv_{1}}, \cj{\tvv_{2}}, \ldots, \cj{\tvv_{t}}\} \\
% 		&= \SS \times \cj{\Span \{\tvv_{1}, \tvv_{2}, \ldots, \tvv_{t}\}} \\
% 		&= \SS \cj{\sds{t}{\tAA, \tbb}} = \MM \sds{t}{\bAA \bMM, \bbb},
% 	\end{align*}
% 	and
% 	\begin{align*}
% 		\bPP \PPT \Span \{\zz_{1}, \zz_{2}, \ldots, \zz_{t}\} &= \SSdT \SST \Span \{\zz_{1}, \zz_{2}, \ldots, \zz_{t}\} \\
% 		&= \SSdT \times \Span \{\tvv_{1}, \tvv_{2}, \ldots, \tvv_{t}\} \\
% 		&= \SSdT \sds{t}{\tAA, \tbb} = \bPP \PPT \sds{t}{\AA \MM, \bb},
% 	\end{align*}
% 	where $ \SSdT \defeq (\SSd)^{\T} $. Also, at $ t = \tg+1 $, by \cref{eq:def_z_w_d_CS,eq:def_Md_M_Sd_S} and the fact $ \tvv_{\tg+1} = \zero $, we obtain $ \ww_{\tg+1} = \zero $ and $ \SST \zz_{\tg+1} = \zero $, i.e., $ \zz_{\tg+1} \in \Null(\bMM) $.
% \end{proof}

\cref{lemma:alpha_beta_CS} shows how we can compute $ \alphat $ and $ \betat $ in \cref{eq:updates_z_w_CS}, i.e., how to construct $ \tTTt $ in \cref{eq:tridiagonal_T}. 

\begin{lemma}
	\label{lemma:alpha_beta_CS}
	For $ 1 \leq t \leq \tg $, we can compute $ \alpha_{t} $ and $ \beta_{t} $ in \cref{eq:updates_z_w_CS} as
	\begin{align*}
		\alphat = \frac{1}{\betat^2} \dotprod{\bwwt, \AA \wwt}, \quad \betatn = \sqrt{\dotprod{\bzztn, \wwtn}}.
	\end{align*}
\end{lemma}
% \begin{proof}
% 	The proof is similar to  that of \cref{lemma:alpha_beta}, and hence is omitted.
%  By \cref{eq:saunders_process}, \cref{eq:def_z_w_d_CS}, \cref{lemma:Kry_spans_z_w_CS}, \cref{eq:Md_STS}, and the orthonormality of the Saunders vectors $ \tvv_{i} $, we obtain
% 	\begin{align*}
% 		\alphat = \dotprod{\tvvt, \tAA \cj{\tvvt} } = \dotprod{\frac{1}{\betat} \SST \zzt, \left(\SST \AA \SS\right) \frac{1}{\betat} \SSH \bzzt} = \frac{1}{\betat^2} \dotprod{\bwwt, \AA \wwt}.
% 	\end{align*}
% 	By \cref{eq:def_z_w_d_CS,eq:Md_STS}, and the facts that $ \beta_{t+1} > 0 $ and $ \vnorm{\tvvtn} = 1 $ for any $ 1 \leq t \leq \tg-1 $, we get
% 	\begin{align*}
% 		\betatn = \sqrt{\vnorm{\betatn \tvvtn}^2} = \sqrt{\vnorm{\betatn \bvvtn}^2} = \sqrt{\dotprod{\bzztn, \MM \bzztn}} = \sqrt{\dotprod{\bzztn, \wwtn}}.
% 	\end{align*}
% 	For $ t = \tg $, this relation continues to hold since $ \beta_{\tg+1} = 0 $ and by \cref{lemma:Kry_spans_z_w_CS}, $ \ww_{\tg+1} = \zero $.
% \end{proof}
%
Since $ \MM \succeq \zero $ and $\dotprod{\bzzt, \wwt} = \dotprod{\bzzt, \MM \bzzt}$, $ \betat $ is well-defined for all $ 1 \leq t \leq \tg $, where $ \beta_{1} = \| \tbb \| = \sqrt{\dotprod{\bzz_{1}, \ww_{1}}} $. 

Recall the update direction in MINRES applied to \cref{eq:pMINRES} with the complex-symmetric matrix is given by 
\begin{align}
	\label{eq:updates_v_d_CS}
	\tdd_{t} = \frac{1}{\gamma^{[2]}_{t}}\left( \cj{\tvvt} - \epsilon_{t} \tdd_{t-2} - \delta^{[2]}_{t} \tdd_{t-1} \right), \quad 1 \leq t \leq \tg,
\end{align}
where $ \tdd_{0} = \tdd_{-1} = \zero $. 
Let $ \ddt $ be a vector such that $ \ddt = \SS \tddt $ and set $ \dd_{0} = \dd_{-1} = \bm{0} $. It follows that 
\begin{align*}
	\ddt = \SS \tddt &= \frac{1}{\gamma^{[2]}_{t}}\left( \SS \cj{\tvvt} - \epsilon_{t} \SS \tdd_{t-2} - \delta^{[2]}_{t} \SS \tdd_{t-1} \right) \\
	&= \frac{1}{\gamma^{[2]}_{t}}\left( \SS \cj{\tvvt} - \epsilon_{t} \dd_{t-2} - \delta^{[2]}_{t} \dd_{t-1} \right).
\end{align*}
From \cref{eq:def_z_w_d_CS,eq:Md_STS}, we obtain the following three-term recurrence relation 
\begin{align}
	\label{eq:update_x_d_CS}
	\dd_{t} = \frac{1}{\gamma^{[2]}_{t} } \left(\frac{1}{\betat} \wwt - \epsilon_{t} \dd_{t-2} - \delta^{[2]}_{t} \dd_{t-1}\right), \quad 1 \leq t \leq \tg.
\end{align}
Clearly, \cref{eq:update_x_d_CS} implies \cref{eq:updates_v_d_CS}. To see this, note that \cref{lemma:Kry_tA_tb_CS} implies $ \cj{\tvvt} \in \rg{\SSH}$, and as a result $ \tddt \in \rg{\SSH}$ by \cref{eq:updates_v_d_CS}. Hence, the identity $ \SSd \SS \SSH = \SSH $ implies that $\SSd \ddt = \SSd \SS \tddt = \tddt $ and $\cj{\tvvt} = \SSd \wwt / \betat $ by \cref{eq:def_z_w_d_CS}. Now, multiplying both sides of \cref{eq:update_x_d_CS} by $ \SSd $ gives \cref{eq:updates_v_d_CS}. 
Also, by \cref{lemma:Kry_spans_z_w_CS}, this construction implies that $ \ddt \in \Span\{\ww_1, \ww_2, \ldots, \ww_{t} \} = \MM \sds{t}{\bAA \bMM, \bbb} $ for $ 1 \leq t \leq \tg $. 
%
%	This can be easily seen by multiplying both sides of \cref{eq:update_x_d} by $ \SSd $ and noting that $\SSd \wwt/\betat = \SSd \SS \tvvt = \tvvt$ by \cref{lemma:Kry_tA_tb}.
%
Finally, from \cref{eq:def_x_r,eq:update_x_d_CS}, it follows that the update is given by
\begin{align*}
	\xxtn = \SS \txxtn = \SS \txxt + \tau \SS \tddt = \xxt + \tau \ddt.
\end{align*}
Initializing with letting $ \xx_0 = \bm{0} $, gives $ \xxt \in \MM \sds{t}{\bAA \bMM, \bbb} $. 

We also have the following recurrence relation on quantities connected to the residual in this complex-symmetric setting, similar to \cref{lemma:residuals} for the case of Hermitian systems.
\begin{lemma}
	\label{lemma:residuals_CS}
	For any $1\leq t \leq \tg$, define $ \hrrt = \bSS \trrt \in \comp^{\dn} $, where $\trrt = \tbb - \tAA \txxt$, then
	\begin{subequations}		
		\begin{align}
			\hrrt &= s_{t}^2 \hrrtp - \frac{\phi_{t} \bc_{t}}{\betatn} \bwwtn  \label{eq:hr_CS}, \\
			\bPP \PPT \rrt &= \bPP \PPT \left(s_t^2 \rr_{t-1} - \frac{\phi_{t} \bc_t}{\beta_{t+1}} \zz_{t+1}\right), \label{eq:PPTr_CS}
		\end{align}
	\end{subequations}
	where $ \PP $ and $ \ww_{t} $ are, respectively, defined in \cref{eq:def_Md_M_Sd_S,eq:def_z_w_d_CS}.
\end{lemma}
% \begin{proof}
% 	By \cref{eq:def_tA_tb_tr_hr_CS,eq:def_z_w_d_CS}, and multiplying both sides of the residual update in \cref{alg:pMINRES} (cf.\ \cref{lemma:residual_CS}) by $ \SS $, we get 
% 	\begin{align*}
% 		\hrrt = \bSS \trr_{t} = s_{t}^2 \bSS \trr_{t-1} - \phi_{t} \bc_{t} \bSS \tvv_{t+1} = s_{t}^2 \hrrtp - \frac{\phi_{t} \bc_{t}}{\betatn} \bwwtn.
% 	\end{align*}
% 	Similarly, by \cref{eq:def_tA_tb_tr_hr_CS,eq:Md_PPH,eq:Md_STS,eq:def_z_w_d_CS,eq:trrt_rrt_CS}, and noting that $ \SSdT = \SSdT \bSSd \bSS = \bMMd \bSS $, we have
% 	\begin{align*}
% 		\bPP \PPT \rrt &= \SSdT \SST \rrt = \SSdT \trrt =  \bMMd \bSS \trrt = s_t^2 \bMMd \bSS \trr_{t-1} - \phi_{t} \bc_{t} \bMMd \bSS \tvv_{t+1} \\
% 		&=  s_t^2 \bMMd \bMM \rr_{t-1} - \frac{\phi_{t} \bc_t}{\beta_{t+1}} \bMMd \bMM \zz_{t+1} = s_t^2 \bPP \PPT \rrtp - \frac{\phi_{t} \bc_t}{\beta_{t+1}} \bPP \PPT \zz_{t+1}.
% 	\end{align*}
% \end{proof}

Next, we show that $ \xxt $ and the residuals $ \hrrt $ from \cref{alg:pMINRES} for complex-symmetric systems  belong to a certain Krylov subspace; see \cref{lemma:Kry_x_hr} for the corresponding result for the Hermitian case.
\begin{lemma}
	\label{lemma:Kry_x_hr_CS}
	With \cref{eq:def_tA_tb_tr_hr_CS,eq:def_x_r} and $ \xx_{0} = \zero $, we have for any $ 1 \leq t \leq \tg $
	\begin{align*}
		\xxt \in \MM \sds{t}{\bAA \bMM, \bbb},
	\end{align*}
	and $\Span \{\hrr_{0}, \hrr_{1}, \ldots, \hrr_{t-1}\} = \bMM \sds{t}{\AA \MM, \bb}$. 
\end{lemma}
% \begin{proof}
% 	The first result is obtained by \cref{lemma:Kry_spans_z_w_CS} and our construction of the update direction, as discussed above. The second result is also obtained by \cref{eq:hr_CS,lemma:Kry_spans_z_w_CS} and noting that $ \hrr_{0} = \ww_{1} $.
% \end{proof}

\begin{remark}
	\label{rm:all_vectors_subspaces_CS}
	Similar to \cref{rm:all_vectors_subspaces} for the Hermitian systems, \cref{eq:PPTr_CS} implies that in the complex-symmetric settings, we can define a new vector $ \crrt $ by
	\begin{align}
		\label{eq:cr_CS}
		\crrt = s_t^2 \crrtp - \frac{\phi_{t} \bc_t}{\beta_{t+1}} \zz_{t+1},
	\end{align}
	where $ \crr_{-1} = \zero $, which then implies $ \bPP \PPT \crrt = \bPP \PPT \rrt $, i.e., computing $ \bPP \PPT \rrt $ can be done without performing an extra matrix-vector product $\AA \xxt$. Similar to \cref{rm:all_vectors_subspaces}, this fact as well as  \cref{eq:updates_z_w_CS,lemma:Kry_spans_z_w_CS}, along with $\zz_{0} = \zero$ and $ \zz_1 = \bb $ give $\zzt \in \sds{t}{\AA \MM, \bb}$, which in turn implies $ \crrtp \in \sds{t}{\AA \MM, \bb} $. 
	Also, from \cref{lemma:Kry_x_hr_CS,lemma:Kry_spans_z_w_CS,eq:update_x_d_CS}, it follows that $ \wwt, \ddt, \xxt \in \MM \sds{t}{\bAA \bMM, \bbb} $ and $ \hrrtp \in \bMM \sds{t}{\AA \MM, \bb}$.
	Hence, all the vectors $ \wwt, \ddt, \xxt, \hrrtp, \zzt, \crrtp $ are invariant for a given $ \MM $ to the particular choice of $ \SS $ that gives $ \MM = \SS\SSH $. 
\end{remark}

	\bibliographystyle{plain}
	\bibliography{biblio}

\end{document}